\documentclass[leqno]{amsart}
\numberwithin{equation}{section}
\usepackage{bbm}
\usepackage{mathtools}
\usepackage{enumerate}
\usepackage{amsmath}
\usepackage{amssymb}
\usepackage{amsthm}
\usepackage{color}
\usepackage{subfigure}
\usepackage{graphicx}
\usepackage{psfrag}
\usepackage[all]{xy}
\newtheorem{theorem}{Theorem}[section]
\newtheorem{corollary}[theorem]{Corollary}
\newtheorem{lemma}[theorem]{Lemma}
\newtheorem{prop}[theorem]{Proposition}
\theoremstyle{definition}
\newtheorem{defn}[theorem]{Definition}
\newtheorem{remark}[theorem]{Remark}  

\newcommand{\Le}{\rotatebox[origin=c]{180}{$\Gamma$}}
\newcommand{\Gr}{\mathrm{Gr}}
\newcommand{\mat}{p}
\newcommand{\postdiag}{P}
\newcommand{\dimerpart}{\text{\rm \DJ}}
\newcommand{\uu}{\widetilde{\dimerpart}}
\newcommand{\dimerconfig}{\delta}
\newcommand{\weight}{\mathrm{w}}
\newcommand{\vectwo}{v'}

\newcommand{\twist}[1]{\overleftarrow{#1}}
\newcommand{\dtwist}[1]{\overleftarrow{\overleftarrow{#1}}}

\newcommand{\tmat}{\twist{\mat}}
\newcommand{\tmats}[1]{\left( \vphantom{A^b} \raisebox{-1pt}{$\tmat$} \right)_{#1}}
\newcommand{\twistbracket}[1]{\left( \vphantom{A^b} \raisebox{-1pt}{$\twist{#1}$} \right)}
\newcommand{\dtwistbracket}[1]{\left( \vphantom{A^b} \raisebox{-1pt}{$\dtwist{#1}$} \right)}
\newcommand{\dtmat}{\dtwist{\mat}}
\newcommand{\dtmats}[1]{\left( \vphantom{A^{b^A}} \raisebox{-1.5pt}{$\dtmat$} \right)_{#1}}
\newcommand{\fn}{f}
\newcommand{\bilinear}[2]{\left\langle #1 \,\middle\|\, #2 \right\rangle}

\newcommand{\bilinearb}[2]{{\Big \langle} #1 \, {\Big \|}\, #2 {\Big \rangle}}

\newcommand{\coeff}[1]{\mathbf{#1}}

\newcommand{\mcoeff}[1]{[\mathbf{#1}]}
\newcommand{\mscoeff}[2]{[\sigma^{#1}(\mathbf{#2})]}
\newcommand{\minor}[1]{[#1]}
\newcommand{\mint}[2]{\Delta_{#1}^{#2}(p)}

\newcommand{\BFZcheck}[1]{\widehat{#1}}

\newcommand{\sign}{\mathrm{sign}}
\newcommand{\xx}{\widetilde{\mathbf{x}}}
\newcommand{\Q}{\widetilde{Q}}

\newcommand{\HH}{\mathcal{H}}

\newcommand{\R}{\mathcal{R}}
\newcommand{\G}{\Gamma}
\newcommand{\rr}{r}

\newcommand{\Mut}{\mathrm{Mut}}

\begin{document}

\author{R. J. Marsh}
\address{School of Mathematics, University of Leeds, Leeds, LS2 9JT, UK}
\email{marsh@maths.leeds.ac.uk}

\author{J. S. Scott}
\address{Universidad de los Andes \\
Departamento de Matem\'{a}ticas \\
Carrera 1 No. 18a 10 \\
Edificio H \\
Primer \indent Piso \\
111711 Bogot\'{a} \\
Colombia}
\email{j.scott@uniandes.edu.co}

\begin{abstract}
The homogeneous coordinate ring of the Grassmannian $\Gr_{k,n}$ has
a cluster structure defined in terms of planar diagrams known as Postnikov diagrams.
The cluster corresponding to such a diagram consists entirely of Pl\"{ucker} coordinates. We introduce a twist map on $\Gr_{k,n}$, related to the Berenstein-Fomin-Zelevinsky-twist,
and give an explicit Laurent expansion for the twist of an arbitrary Pl\"{u}cker coordinate in terms of the cluster variables associated with a fixed Postnikov diagram. The expansion arises as a (scaled) dimer partition function of a weighted version of the bipartite graph dual to the Postnikov diagram, modified by a boundary condition determined by the Pl\"{u}cker coordinate. We also relate the twist map to a maximal green sequence.
\end{abstract}



\thanks{This work was supported by the Indian Department of Atomic Energy, the Institute of Mathematical Sciences, Chennai, India, the Engineering and Physical Sciences Research Council [grant numbers EP/C01040X/2 and EP/G007497/1], the Mittag-Leffler Institute and the Sonderforschungsbereich (Collaborative Research Centre) 701 at the University of Bielefeld, Germany.}

\title{Twists of Pl\"{u}cker coordinates as dimer partition functions}
\date{5 September 2015}
\maketitle

\section{Introduction}
For positive integers $k \leq n$ let $\Gr_{k,n}$ denote the Grassmannian of all $k$-dimensional vector subspaces of $\mathbb{C}^n$. The results of~\cite{scott06} (see also~\cite{GSV03,GSV10}) prove that its homogeneous coordinate ring
$\mathbb{C}[ \Gr_{k,n} ]$ has the structure of a cluster algebra which possesses a distinguished finite family of seeds $(\xx_P, \widetilde{Q}_P)$ constructed from certain planar diagrams $P$, known as \emph{alternating strand diagrams} or \emph{Postnikov diagrams}.

The extended cluster $\xx_P$ of each seed of this kind consists entirely of Pl\"{u}cker coordinates which, in addition to the associated quiver $\widetilde{Q}_P$, can be read off directly from the Postnikov diagram. Moreover, every Pl\"ucker coordinate occurs as
an element of $\xx_P$ for some
Postnikov diagram $P$ and thus every Pl\"{u}cker coordinate is either a cluster variable or a coefficient. When $k=2$ every seed is of this form and consequently every cluster variable is a Pl\"{u}cker coordinate. In general the homogeneous coordinate ring is of wild type --- possessing infinitely many seeds and infinitely many cluster variables, which in general will not be Pl\"{u}cker coordinates and which are, at present, unclassified.

In this paper we consider a certain rational map from the Grassmannian
to itself which we call the \emph{twist}. This map may be pre-composed
with any regular function $f$ in $\mathbb{C}[\Gr_{k,n}]$ to form a twisted version $\twist{f}$; here, we consider \emph{twisted} Pl\"{u}cker coordinates.
Up to coefficients\footnote{Equal, upon specializing all coefficients
to $1$; see Definition~\ref{d:uptocoefficients}.},
the twist of any cluster variable is a cluster variable (see
Proposition~\ref{p:twistofclustervariable}).\footnote{We thank David Speyer for pointing this out.}

By the Laurent Phenomenon~\cite{fominzelevinsky02}, each cluster variable
in $\mathbb{C}[\Gr_{k,n}]$ can be expressed uniquely as a Laurent polynomial in the extended cluster of any seed. In view of this, we compute Laurent expansions for
twisted Pl\"{u}cker coordinates in the extended cluster $\xx_P$ of a seed attached to any fixed Postnikov diagram $P$. We express these Laurent expansions in terms of \emph{dimer configurations} (also known as perfect matchings) for a weighted bipartite graph $G$ subject to boundary conditions determined by the Pl\"{u}cker coordinate.
The graph $G_P$ is dual (in an appropriate sense) to the Postnikov diagram, and
its edges are weighted by monomials taken from the extended cluster $\xx_P$; see
Definition~\ref{d:weights} in Section~\ref{s:clusterstructure}.

Recall that a bipartite graph is a graph whose vertices are partitioned into two
types, or colours (black and white), where edges join vertices of a different colour.
Such a graph is said to be \emph{balanced} if there is an equal number of black and
white vertices. A \emph{dimer configuration} $\dimerconfig$ of a balanced bipartite
graph, $G$, is a collection of edges of $G$ such that each vertex of $G$ is incident
with precisely one edge in the collection. If each edge $e$ in $G$ is assigned a
weight $\weight_e$, then we define the weight $\weight_{\dimerconfig}$ of the dimer configuration
$\dimerconfig$ to be the product
$$\weight_{\dimerconfig}=\prod_{e\in \dimerconfig} \weight_e.$$
The \emph{dimer partition function} (also known as the matching polynomial) of
$G$ is given by:
$$\dimerpart_G=\sum_{\dimerconfig} \weight_{\dimerconfig},$$
where the sum is over all dimer configurations $\delta$ of $G$.

If $P$ is a Postnikov diagram, the dual bipartite graph $G_P$ is
naturally embedded in a disk, with boundary vertices labelled $1,\ldots ,n$.
If $I$ is a $k$-subset of $\{1,\ldots ,n\}$, then the induced bipartite
subgraph $G_P(I)$ of $G$ obtained by removing the boundary vertices labelled by
the elements of $I$ is balanced.

Our main result is:

\begin{theorem} \label{t:mainresult}
Let $P$ be a Postnikov diagram, with corresponding seed $(\xx_P,\widetilde{Q}_P)$ and let $G_P$ be its weighted bipartite dual graph. For any $k$-subset $I$ of $\{1,\ldots ,n\}$, we have:
$$\frac{\dimerpart_{G_P(I)}}{\prod_{x\in \mathbf{x}_P} x}=\twist{\minor{I}},$$
where $[I]$ denotes the Pl\"{u}cker coordinate of $\Gr_{k,n}$ associated to $I$, $\twist{\minor{I}}$ denotes its twist, and $\mathbf{x}_P\subseteq \xx_P$ is the (non-extended) cluster corresponding to $P$.
\end{theorem}

In particular, as an element of the rational function field of the Grassmannian, the ratio on the left hand side of this formula does not depend on $P$; this is a key step
in the proof.
This rational expression is in fact the Laurent expansion of the twisted Pl\"{u}cker coordinate $\twist{\minor{I}}$ for the seed $(\xx_P, \widetilde{Q}_P)$, since the edge-weights which contribute to the dimer partition function in the numerator are monomials in the extended cluster $\xx_P$.

We remark that it follows from the formula in Theorem~\ref{t:mainresult} that the twist map preserves the positive Grassmannian (see Corollary~\ref{c:preservespositive}).
We note that the cluster algebra structure on $\mathbb{C}[\Gr_{k,n}]$ and the positive Grassmannian play an important r\^{o}le in the study of scattering amplitudes via on-shell diagrams in~\cite{ABCGPT}.

Dimer configurations have also been used as a method for computing Laurent expansions for cluster variables for cluster algebras of finite classical type~\cite{carollprice03,musiker11}, and for cluster algebras associated to triangulations of surfaces~\cite{canakci,MS10,MSW11,MSW13}. Both cases involve cluster
algebras of finite mutation type, the homogeneous coordinate ring of the Grassmannian $\Gr_{2,n}$ (up to coefficients) being common to both cases. Note that the case considered here, i.e.\ $\mathbb{C}[\Gr_{k,n}]$, is not of finite mutation type in general.

Our interest in the twist map stems partly from its close relationship with the
BFZ-twist automorphism~\cite{BFZ96, BZ97}
defined on a unipotent cell $N^w$ in the group $N$ of all complex $n \times n$ unipotent matrices, where $w$ is the Grassmann permutation.
Note that the BFZ-twist
has a representation-theoretic formulation~\cite[Theorem 6]{GLS12}
in terms of the Auslander-Reiten translate on a categorification of the cluster algebra structure on $N^w$ using the type $A_{n-1}$ preprojective algebra.
The cell $N^w$
is birationally equivalent to $\Gr_{k,n}$, so the BFZ-twist can be transported to $\Gr_{k,n}$. We show in Section~\ref{s:BFZtwist} that the (transported) BFZ-twist and the twist defined in Section~\ref{s:twist} coincide, up to coefficients.

The results of~\cite{CAIII}, together with the BFZ-ansatz~\cite{BFZ96, BZ97},
can be used to compute Laurent expansions of BFZ-twisted minors on $N^w$ in terms of any seed associated to a reduced expression for $w$. The cluster algebra structures on $N^w$ and $\Gr_{k,n}$ are identified, up to coefficients, through the birational equivalence. In particular, every seed associated to a reduced expression
for $w$ corresponds to a seed attached to a Postnikov diagram. However, not
every Postnikov diagram is of this form, and, in view of this, our results can be seen as a proper extension of the combined results of~\cite{CAIII, BZ97}. On a combinatorial level, the approach taken by~\cite{CAIII} is different from the approach here: specifically, the formulas in~\cite{CAIII} involve an analysis of families of noncrossing paths in planar diagrams, while our formulas use dimer configurations.
An interesting question is how to reconcile these two approaches. It is also
interesting to compare the formulas obtained in this paper with the formulas in~\cite[\S1]{talaska11}, also given in terms of families of paths.

We also show that, up to coefficients, the twist can be implemented by a maximal green sequence (Theorem~\ref{t:greensummary}).
Such sequences occur both in
cluster theory~\cite{keller11} and in the analysis of BPS states; see~\cite{ACCERV14,CCV,CNV}. 
The maximal green sequence we give here can be regarded as a two-dimensional version of the sequence given in~\cite{BDP14} for quivers of type $A$, and should be compared with the mutation sequence $\tau_-\tau_+$
occurring in the Zamolodchikov periodicity conjecture~\cite{keller13,RVT93,zamolodchikov91} which was proved in the case $A_{k-1}\times A_{n-k-1}$
by~\cite{volkov07}.
Note that the cluster algebra of type $A_{k-1}\times A_{n-k-1}$ coincides with $\mathbb{C}[\Gr_{k,n}]_1$,
the cluster algebra obtained from $\mathbb{C}[\Gr_{k,n}]$ by specializing all of the coefficients to $1$ (see Definition~\ref{d:specializedcoefficients}).

The general case of the Zamolodchikov periodicity conjecture
was proved by B. Keller in~\cite{keller13}: see the introduction to~\cite{keller13} for many further references relating to this conjecture. We note that a modified version of the mutation sequence appearing in~\cite{keller13} in the type $A_{k-1} \times A_{n-k-1}$ case
corresponds to an operator product $\widehat{m}_{\square}$ studied in~\cite[\S8]{CNV} in the context of a $4d$ $\mathcal{N}=2$ field theory whose quantum monodromy is related to $(\widehat{m}_{\square})^k$.

We note that the quiver $\widetilde{Q}_{k,n}$ used in~\cite{scott06} in the analysis of the cluster algebra structure of the Grassmannian coincides with Keller's box quiver $Q\Box Q'$ (where $Q, Q'$ are alternating orientations of $A_{k-1}$ and $A_{n-k-1}$ respectively), together with a rule for attaching the coefficient vertices along the boundary.
In this case $\tau_-$ (respectively, $\tau_+$) is the composition of mutations at the
odd (respectively, even) internal vertices of $Q\Box Q'$. One can check directly that
the effect of $\tau_-\tau_+$ is to shift the coefficient vertices of $\widetilde{Q}_{k,n}$ counter-clockwise by one step, while leaving the internal part, $Q\Box Q'$, unchanged.
As $\widetilde{Q}_{k,n}$ is the quiver of a Postnikov diagram, one concludes that the effect of $\tau_-\tau_+$ on a Pl\"{u}cker coordinate $[I]$ in the initial seed is to send it to $[\sigma^{-1}(I)]$, where $\sigma$ is
anticlockwise rotation of its indices (see Definition~\ref{d:alpha}).
Consequently, $\tau_-\tau_+$ coincides
with the automorphism of $\mathbb{C}[Gr_{k,n}]$
sending $[I]$ to $[\sigma^{-1}(I)]$ for any
Pl\"{u}cker coordinate $[I]$.
This, together with Theorem~\ref{t:greensummary}, gives us the identity $(\tau^-\tau^+)^k={\gamma}^2$ up to coefficients (Corollary~\ref{c:greenidentity}), where
${\gamma}$ is an automorphism of $\mathbb{C}[\Gr_{k,n}]_1$ corresponding to a maximal green sequence.

In~\cite[\S13]{GLS11}, a mutation sequence
for the categorification of a unipotent cell
corresponding to any Weyl group element is
given; according to Keller~\cite[\S5]
{keller11} this is a maximal green sequence.
In the case of the Grassmann permutation
this mutation sequence should also give a 
maximal green sequence for the Grassmannian.
As well as relating our choice of maximal 
green sequence to the twist, we show that it 
passes only through seeds given by Postnikov 
diagrams.

The paper is organized as follows.
In Section~\ref{s:twist}, we introduce the twist map on the Grassmannian. This is
defined in terms of generalized cross products of the columns of the $k\times n$ matrix representing a point in the Grassmannian.

In Section~\ref{s:twistPluecker}, we show (in Proposition~\ref{p:twistcomputation})
that the twist of a Pl\"{u}cker coordinate $\minor{I}$ given by a $k$-subset $I$
which is a disjoint union of two cyclic intervals is, up to coefficients, a Pl\"{u}cker coordinate of the same kind, using Turnbull's identity (as expressed in the article~\cite{leclerc93}).
We suspect that these are the only Pl\"{u}cker coordinates with this property.
For example, it can be verified by hand, in the case of the Grassmannian $\Gr_{3,n}$,
that a twisted Pl\"{u}cker coordinate $\twist{\minor{I}}$, where $I$ is not a disjoint union of two cyclic intervals, is a cluster variable which is not a Pl\"{u}cker coordinate (even up to coefficients).

In Section~\ref{s:periodicity}, we compute the double twist of
a Pl\"{u}cker coordinate and use this to show that the twist map is periodic, up to coefficients (Proposition~\ref{p:periodicitymonomial}; see also Proposition~\ref{p:periodicitycoefficients}).
We then explain the relationship to the BFZ-twist in Proposition~\ref{p:relationship} in Section~\ref{s:BFZtwist}.

In Section~\ref{s:clusterstructure}, we recall the cluster structure of
the Grassmannian as described in~\cite{scott06} in terms of Postnikov diagrams~\cite{postnikov}. 

In Section~\ref{s:dimer}
we review the definition of the bipartite
graph dual to a Postnikov diagram and the blow-up and blow-down equivalences inherited from the oriented lens creation and annihilation moves for Postnikov
diagrams. We introduce a scheme for weighting the edges of such a bipartite graph, with the property that the induced weighting on dimer configurations is invariant
under blow-ups and blow-downs.
We then fix a Postnikov diagram, $P$, and a $k$-subset $I$.
We show, in Proposition~\ref{prop:flipmoveinvariance}, that the dimer partition function of $G_P(I)$, divided by the product of the elements in $\mathbf{x}_P$, is invariant under quadrilateral moves. Since any Postnikov diagram can be reached from any
other by a sequence of such moves, it follows that this scaled dimer partition function is independent of the choice of Postnikov diagram.

In Section~\ref{s:regular}, we show that the main result is true for the Pl\"{u}cker coordinates in $\xx_{\R_{k,n}}$ for a regular Postnikov diagram $\R_{k,n}$
whose dual bipartite graph, $G_{\R_{k,n}}$, is, up to some boundary
edges, part of a hexagonal tiling of the plane.
The diagram obtained from $\R_{n-k,n}$ by reversing its strands (and adding crossings
at the boundary) is again a Postnikov diagram, which we denote by $\R_{n-k,n}^*$.
We prove, in Proposition~\ref{p:uniquematching}, that $G_{\R^*_{n-k,n}}(I)$ has a unique dimer configuration whenever $[I]$ lies in $\xx_{\R_{k,n}}$.
This $k$-subset, $I$, is a disjoint union of two cyclic intervals in $\{1,\ldots ,n\}$, which allows us to compare the dimer partition function
$\dimerpart_{G_P(I)}$, where $P=\R_{n-k,n}^*$, with
the formula for $\twist{\minor{I}}$ given by Proposition~\ref{p:twistcomputation}
in Section~\ref{s:twist}.

The main result, Theorem~\ref{t:mainresult}, is shown in Section~\ref{s:mainresult},
using the fact that twists of Pl\"{u}cker coordinates and the scaled dimer partition functions both satisfy the short Pl\"{u}cker relations. In Section~\ref{s:example}, we give an example. In Section~\ref{s:green} we consider the
relationship between the twist and maximal green sequences
and in Section~\ref{s:surfaces} we discuss generalization to the surface case.

\section{Twist}
\label{s:twist}
For a positive integer $r$, an \emph{$r$-subset} of a set $S$ is a subset of $S$ of cardinality $r$.

Let $M_{k,n}(\mathbb{C})$ denote the set of complex $k\times n$ matrices.
Recall that an element of the Grassmannian $\Gr_{k,n}$ of $k$-dimensional
subspaces of $\mathbb{C}^n$ can be regarded as a rank $k$ matrix
$\mat=(\mat_{ij})\in M_{k,n}$ up to left multiplication by an element
of $\text{GL}_k(\mathbb{C})$. The rows correspond to a choice of basis of the
subspace of $\mathbb{C}^n$ and the action of $\text{GL}_k(\mathbb{C})$
corresponds to a change of basis.

Each $k$-subset $I=\{i_1<i_2<\ldots <i_k\}$ of $\{1,\ldots ,n\}$ defines a
minor of $\mat$ associated to the row-set $\{1,\ldots ,k\}$ and the column-set $I$
(written in increasing order). We denote this minor by $\minor{I}$.
Then the map taking $\mat\in M_{k,n}$ to the tuple consisting of all of the minors of $\mat$ of this form is a well-defined map (the Pl\"{u}cker embedding) from $\Gr_{k,n}$ to the projective space $\mathbb{P}^{\binom{n}{k}-1}$, identifying $\Gr_{k,n}$ with a projective subvariety defined by the Pl\"{u}cker relations.

Let $\sigma:\{1,\ldots ,n\}\rightarrow
\{1,\ldots ,n\}$ be the map taking
$i$ to $i-1$ reduced modulo $n$.
Then $\sigma$ induces a map on the
set of $k$-subsets of $\{1,\ldots ,n\}$ which we also denote by $\sigma$.

For $i\in \{1,\ldots ,n\}$, we write $\coeff{i}$ for the $k$-subset
$\{\sigma^{k-1}(i),\ldots ,\sigma(i),i\}$; the corresponding Pl\"{u}cker
coordinate is denoted $\mcoeff{i}$.

Given vectors $v_1,\ldots ,v_{k-1}$ in $\mathbb{C}^k$, the \emph{generalized cross-product} $v_1\times \cdots \times v_{k-1}$ is the unique vector in $\mathbb{C}^k$ satisfying the constraint:

$$\bilinearb{v_1 \times \cdots \times v_{k-1}}{v}=
\det(v_1,\ldots ,v_{k-1},v)$$
for all $v\in \mathbb{C}^k$, where $\bilinear{v}{v'}:=v^T v'$ is
the standard scalar product for $v,v'\in \mathbb{C}^k$. We interpret an empty
cross product (the case $k=1$) as $1$.
It follows from
basic multi-linear algebra that the cross product satisfies the \emph{contraction formula}:
$$\bilinearb{v_1 \times \cdots \times v_{k-1}}{\vectwo_1\times \cdots \times \vectwo_{k-1}}=
\det \begin{pmatrix}
\bilinear{v_1}{\vectwo_1} & \cdots & \bilinear{v_1}{\vectwo_{k-1}} \\
\vdots && \vdots \\
\bilinear{v_{k-1}}{\vectwo_1} & \cdots & \bilinear{v_{k-1}}{\vectwo_{k-1}}
\end{pmatrix}.$$

Let $\mat\in M_{k,n}$ be a $k\times n$ matrix with column vectors
$\mat_1,\ldots ,\mat_n\in\mathbb{C}^k$.

\begin{defn} \label{d:twist}
The (left) \emph{twist} $\tmat\in M_{k,n}$ is defined to be the $k\times n$ matrix whose $i$th column vector is:
\begin{equation*}
\tmats{i}=\varepsilon_i\cdot(\mat_{\sigma^{k-1}(i)}\times \mat_{\sigma^{k-2}(i)}\times \cdots \mat_{\sigma(i)}),
\end{equation*}
where
$$\varepsilon_i=\begin{cases}
(-1)^{i(k-i)} & \text{if }i\leq k-1; \\
1 & \text{if }i\geq k.
\end{cases}$$
\end{defn}

This may also be written:
$$\tmats{i}=\begin{cases}
(-1)^{k-i}\mat_1 \times \cdots \times \mat_{i-1} \times \mat_{i-k+1+n} \times \cdots \times
\mat_n & \text{if } i\leq k-1, \\
\mat_{i-k+1} \times \cdots \times \mat_{i-1} & \text{if } i\geq k.
\end{cases}$$

Note that $\tmats{\rr i}$ is the determinant of the submatrix of $p$ with column
set $\{\sigma^{k-1}(i),\ldots ,\sigma(i),i\}$ (appearing in numerical order),
in which the column where the column $i$ of $p$ appears is replaced with the vector $e_{\rr}$ with a $1$ in its $\rr$th position and zeros everywhere else.
For example, if $k=3$ and $n=5$, if $p=(p_{ri})_{1\leq r\leq 3,\,1\leq i\leq 5}$, then
\begin{align*}
&\tmat =
\begin{pmatrix}
\mint{23}{45} &
-\mint{23}{15} &
\mint{23}{12} &
\mint{23}{23} &
\mint{23}{34} \\
-\mint{13}{45} &
\mint{13}{15} &
-\mint{13}{12} &
-\mint{13}{23} &
-\mint{13}{34} \\
\mint{12}{45} &
-\mint{12}{15} &
\mint{12}{12} &
\mint{12}{23} &
\mint{12}{34}
\end{pmatrix}= \\
& \begin{pmatrix}
p_{24}p_{35}-p_{25}p_{34} &
-p_{21}p_{35}+p_{25}p_{31} &
p_{21}p_{32}-p_{22}p_{31} &
p_{22}p_{33}-p_{23}p_{32} &
p_{23}p_{34}-p_{24}p_{33} \\
-p_{14}p_{35}+p_{15}p_{34} &
p_{11}p_{35}-p_{15}p_{31} &
-p_{11}p_{32}+p_{12}p_{31} &
-p_{12}p_{33}+p_{13}p_{32} &
-p_{13}p_{34}+p_{14}p_{33} \\
p_{14}p_{25}-p_{15}p_{24} &
-p_{11}p_{25}+p_{15}p_{21} &
p_{11}p_{22}-p_{12}p_{21} &
p_{12}p_{23}-p_{13}p_{22} &
p_{13}p_{24}-p_{14}p_{23}
\end{pmatrix},
\end{align*}
where $\Delta_{X}^Y(p)$ stands for the minor of $p$ with rows
$X$ and columns $Y$.

\begin{remark} \label{r:twistk1}
Note that, if $k=1$, then $\tmat$ 
is always the $1\times n$ matrix whose entries are all equal to $1$.
\end{remark}

We shall use the
notation $\dtmat$ to denote the result of applying the twist twice to $\mat$.

\begin{lemma}
The map $\mat\mapsto \tmat$ from $M_{k,n}$ to itself induces a
well-defined rational map from $\Gr_{k,n}$ to itself.
\end{lemma}
\begin{proof}
We note first that for any vectors $v,v_1,\ldots ,v_{k-1}$ in $\mathbb{C}^k$ and $g\in GL_k(\mathbb{C})$,
\begin{align*}
\begin{split}
\bilinearb{(gv_1)\times \cdots \times (gv_{k-1})}{v} &=
\det \left( gv_1,\ldots ,gv_{k-1},v \right) \\
&= \det(g) \det \left( v_1,\ldots ,v_{k-1},g^{-1}v \right) \\
&= \det(g) \bilinearb{v_1\times \cdots \times v_{k-1}}{g^{-1}v} \\
&= \det(g) \bilinearb{(g^{-1})^T(v_1\times \cdots \times v_{k-1})}{v}.
\end{split}
\end{align*}
Hence, $$(gv_1)\times \cdots \times (gv_{k-1})=\det(g)(g^{-1})^T (v_1\times \cdots
\times v_{k-1}).$$
Let $\mat$ be a maximal rank $k\times n$ matrix and $g\in \text{GL}_k(\mathbb{C})$.
Then the $i$th column of $\twist{g\mat}$ is equal to

\begin{align*}
\begin{split}
\varepsilon_i \cdot (g\mat_{\sigma^{k-1}(i)})\times \cdots \times (g\mat_{\sigma(i)}) &=
\varepsilon_i \det(g)(g^{-1})^T (\mat_{\sigma_{k-1}(i)}\times \cdots \times \mat_{\sigma(i)}) \\
&=\det(g)(g^{-1})^T \tmats{i}.
\end{split}
\end{align*}
so the twist preserves maximal rank and does not depend on a choice of representative $p\in M_{k,n}$.

If $I$ is a $k$-subset of $\{1,\ldots ,n\}$, then we have, up to sign:
\begin{align*} 
\begin{split}
\minor{I}\twistbracket{\mat} &= \det ( \tmats{i_1},\ldots ,\tmats{i_k}) \\
&= \det(
\mat_{\sigma^{k-1}(i_1)}\times \cdots \times \mat_{\sigma(i_1)},
\ldots,
\mat_{\sigma^{k-1}(i_k)}\times \cdots \times \mat_{\sigma(i_k)})  \\
&=
\bilinearb{
(\mat_{\sigma^{k-1}(i_1)}\times \cdots \times \mat_{\sigma(i_1)})
\times \cdots \times
(\mat_{\sigma^{k-1}(i_{k-1})}\times \cdots \times \mat_{\sigma(i_{k-1})})}{
\mat_{\sigma^{k-1}(i_k)}\times \cdots \times \mat_{\sigma(i_k)}}.
\end{split}
\end{align*}
Applying the contraction formula, we see that this is a determinant whose
entries are Pl\"{u}cker coordinates. This is a homogeneous polynomial map of degree independent of the choice of $I$. It follows that the twist is a rational
map from $\Gr_{k,n}$ to itself.
\end{proof}

For example, if $k=3$ and $n=5$, we have:
\begin{equation} \label{e:twistexample}
\begin{split}
[124]\twistbracket{\mat} &=\bilinearb
{
(\mat_{4}\times \mat_{5})
\times
(\mat_{5}\times \mat_{1})
}
{
\mat_{2}\times \cdots \times \mat_{3}
} \\
&=
\begin{vmatrix} 
[245] & [345] \\
[125] & [135]
\end{vmatrix} \\
&=[245][135]-[125][345]=[145][235],
\end{split}
\end{equation}
using a Pl\"{u}cker relation (and the fact that
determinants change sign when columns are interchanged).

We shall see later (see Corollary~\ref{c:preservespositive})
that the twist preserves the totally positive Grassmannian and
the totally nonnegative Grassmannian.

The above shows that the rational map $\mat \mapsto \tmat$ induces a regular map from the affine cone of $\Gr_{k,n}$ to itself (given by the same polynomials). We denote
the induced homomorphism from the homogeneous coordinate ring $\mathbb{C}[\Gr_{k,n}]$ to itself by $\fn\mapsto \twist{\fn}$.

\section{Twists of Pl\"{u}cker coordinates}
\label{s:twistPluecker}
We adopt a notation similar to that of~\cite{leclerc93}.
Given vectors $v_1,\ldots ,v_k\in\mathbb{C}^k$, we write
the determinant of the matrix whose columns are
$v_1,\ldots ,v_k$ by the rectangular $1\times k$ tableau:
$$
\boxed{
\begin{matrix}
v_1 & v_2 & \cdots & v_k
\end{matrix}
}
$$
Given $p \in \text{Gr}_{k,n}$ and $i \in \{1, \dots, n \}$, let $p_i$ denote the $i$th column of $p$.
We shall sometimes just denote
this by $i$ when using the above notation.
So, for example, the Pl\"{u}cker coordinate $\minor{I}=\minor{\{i_1,\ldots i_k\}}$ is given by:
$$
\minor{I}=
\boxed{
\begin{matrix}
i_1 & i_2 & \cdots & i_k
\end{matrix}
}
$$
A tableau with several rows denotes the product of minors corresponding to the rows; thus, for example:
$$
\minor{I}\cdot \minor{J}=
\boxed{
\begin{matrix}
i_1 & i_2 & \cdots & i_k \\
j_1 & j_2 & \cdots & j_k
\end{matrix}
}
$$

As in~\cite[\S 1.1]{leclerc93}, we use the box notation to denote an
alternating sum of products of minors, i.e.\ if $\tau$ is a pair consisting of
a tableau $T$ as above together with a subset $A$ of the entries (indicated by drawing boxes around the elements of $A$), then $\tau$ represents the element:
$$\tau:=\sum_{w} \sign(w)\cdot w(T),$$
where the sum is over cosets $w$ in the symmetric group on the elements of $A$
(of degree $|A|$) of the subgroup preserving the boxed elements in each row of $T$, with each
$w(T)$ interpreted as a product of minors as above.

For example, we have:
\begin{spreadlines}{2ex}
\begin{align*}
\boxed{
\begin{matrix}
\boxed{\scriptstyle a} & b & \boxed{\scriptstyle c} \\
d & e & f
\end{matrix}
}&=
\boxed{\begin{matrix}
a & b & c \\
d & e & f
\end{matrix}}\,;
\\
\boxed{
\begin{matrix}
\boxed{\scriptstyle a} & \boxed{\scriptstyle b} & c \\
d & e & \boxed{\scriptstyle f}
\end{matrix}
}&=
\boxed{\begin{matrix}
a & b & c \\
d & e & f
\end{matrix}}
-
\boxed{
\begin{matrix}
f & b & c \\
d & e & a
\end{matrix}}
-
\boxed{\begin{matrix}
a & f & c \\
d & e & b
\end{matrix}}\,;
\\
\boxed{
\begin{matrix}
\boxed{\scriptstyle a} & \boxed{\scriptstyle b} & c \\[0.1cm]
\boxed{\scriptstyle d} & e & \boxed{\scriptstyle f}
\end{matrix}}
&=
\boxed{\begin{matrix}
a & b & c \\
d & e & f
\end{matrix}}
-
\boxed{\begin{matrix}
a & d & c \\
b & e & f
\end{matrix}}
+
\boxed{\begin{matrix}
a & f & c \\
b & e & d
\end{matrix}}
+
\boxed{\begin{matrix}
b & d & c \\
a & e & f
\end{matrix}}
-
\boxed{\begin{matrix}
b & f & c \\
a & e & d
\end{matrix}}
+
\boxed{\begin{matrix}
d & f & c \\
a & e & b
\end{matrix}}\, .
\end{align*}
\end{spreadlines}

\begin{remark} \label{r:permutation}
By~\cite[Prop. 1.2.1]{leclerc93}, a permutation of the
vectors (boxed or otherwise) lying in a single row of a tableau changes its value by the sign of the permutation. 
A permutation of the boxed vectors (possibly in several rows)
changes the value by the sign of the permutation. A
permutation of the rows of a tableau does not change its
value.
\end{remark}

We shall need Turnbull's identity, as stated
in~\cite[Prop. 1.2.2]{leclerc93}.

\begin{prop} \label{p:turnbull} \textbf{(Turnbull's Identity)}
Let $\tau$ be a rectangular tableau with $k$ columns. If the number of boxed entries in $\tau$ is greater than $k$ then
$\tau=0$. If not, fix a row $r$ of $\tau$. Let $A$ be the
set of boxed entries in row $r$ of $\tau$ and let $D$ be the set of boxed entries in the remaining rows of $\tau$.
Let $B$ be a subset of the unboxed entries in row $r$ of
$\tau$ of cardinality $|D|$. Let $C$ be the set
of entries in row $r$ not in $A$ or $B$.
Let $\nu$ the tableau obtained from $\tau$ by carrying out
the following operations:
\begin{itemize}
\item[(a)] Exchanging the entries in $B$ with the entries
in $D$, but not the boxes. The boxes originally around the entries in $D$ are not moved and box the entries of $B$ after the move.
\item[(b)] Boxing the elements of $C$;
\item[(c)] Removing the boxes from the entries in $A$.
\end{itemize}
Then $\tau=\nu$.
\end{prop}

Note that the value of the tableau thus created is independent of
the choice of exchange in part (a), since any two such exchanges
are related by a permutation $\alpha$ of $B$ followed by a permutation of $D$ which is identical to $\alpha$ under an identification of $B$ and $D$.

An example involving the first row is shown below.
The elements of $A$ are denoted $a_1,a_2$, and similarly
for $B$, $C$ and $D$. Note however that the entries in $A$
do not have to be adjacent in the chosen row (similarly for $B$ and $C$), and that the entries in $D$ can occur anywhere in the remaining rows.
$$
\renewcommand*{\arraystretch}{1.3}
\boxed{
\begin{matrix}
\boxed{\scriptstyle a_1} & \boxed{\scriptstyle a_2} & b_1 & b_2 & b_3 & b_4 & b_5 & c_1 & c_2 & c_3 \\
\boxed{\scriptstyle d_1} & \boxed{\scriptstyle d_2} &
\boxed{\scriptstyle d_3} & e & f & g & h & i & j & k \\
\boxed{\scriptstyle d_4} & \boxed{\scriptstyle d_5} &
l & m & n & p & q & r & s & t
\end{matrix}
}=
\boxed{
\begin{matrix}
a_1 & a_2 & d_1 & d_2 & d_3 & d_4 & d_5 & \boxed{\scriptstyle c_1} & \boxed{\scriptstyle c_2} & \boxed{\scriptstyle c_3} \\
\boxed{\scriptstyle b_1} & \boxed{\scriptstyle b_2} &
\boxed{\scriptstyle b_3} & e & f & g & h & i & j & k \\
\boxed{\scriptstyle b_4} & \boxed{\scriptstyle b_5} &
l & m & n & p & q & r & s & t
\end{matrix}
}
$$
\vskip 0.2cm
The twist of a $k\times n$ matrix $\mat$ in the notation
defined above is the $k\times n$ matrix with entries:
$$
\tmats{\rr i}=
\begin{dcases}
\boxed{
\begin{matrix}
1 & 2 & \cdots & i-1 & e_{\rr} & n+i-k+1 & n+i-k+2 & \cdots & n
\end{matrix}
}
&
1\leq i\leq k; \\
\boxed{
\begin{matrix}
i-k+1 & i-k+2 & \cdots & i-1 & e_{\rr}
\end{matrix}
}
&
k+1\leq i\leq n,
\end{dcases}
$$
for $1\leq \rr\leq k$ and $1\leq i\leq n$.

\begin{defn} \label{d:sigma}
The permutation $\sigma$ (see the start of Section~\ref{s:twist}) induces a
well-defined automorphism of $\Gr_{k,n}$:
$$\sigma(\mat)_{ri}=\begin{cases}
\mat_{r,\sigma(i)}, & i\not=1; \\
(-1)^{k-1}\mat_{r,\sigma(i)}, & i=1.
\end{cases}$$
This induces an automorphism of $\mathbb{C}[\Gr_{k,n}]$, also denoted
$\sigma$.
\end{defn}

\begin{lemma} \label{l:rotatesigns}
Let $I$ be a $k$-subset of $\{1,\ldots ,n\}$ and $p\in Gr_{k,n}$. Then we have:
\begin{enumerate}
\setlength{\itemsep}{4pt}
\item[(a)] $\sigma(\minor{I})=\minor{\sigma(I)}$;
\item[(b)] $\twist{\sigma(p)}=\sigma\left( \twist{p} \right)$;
\item[(c)] $\sigma \left( \twist{\minor{I}} \right)= \overleftarrow{\sigma(\minor{I})}.$
\end{enumerate}
\end{lemma}

\begin{proof}
We first prove (a). If $1\not\in I$, then
$$\sigma(\minor{I})(\mat)=\boxed{\begin{matrix} \sigma(i_1) & \sigma(i_2) & \cdots
& \sigma(i_k) \end{matrix}}=\minor{\sigma(I)}(p).$$
If $1\in I$ then $i_1=1$, so
\begin{align*}
\sigma(\minor{I})(\mat) &=(-1)^{k-1}\,\boxed{\begin{matrix} n & \sigma(i_2) & \cdots
& \sigma(i_k) \end{matrix}} \\
&=\boxed{\begin{matrix} \sigma(i_2) & \cdots
& \sigma(i_k) & n \end{matrix}} \\
&=\minor{\sigma(I)}(p),
\end{align*}
giving (a).

For (b), we have, by Defnition~\ref{d:twist}, that
\begin{align}
\label{e:firsttwist}
\begin{split}
\left( \twist{\sigma(\mat)} \right)_i &=
\varepsilon_i\cdot \sigma(\mat)_{\sigma^{k-1}(i)}\times
\sigma(\mat)_{\sigma^{k-2}(i)}\times \cdots 
\sigma(\mat)_{\sigma(i)} \\
&=
\begin{cases}
\varepsilon_i\cdot \mat_{\sigma^k(i)}\times
\mat_{\sigma^{k-1}(i)}\times \cdots \mat_{\sigma^2(i)},
& i\not\in \{2,\ldots ,k\}; \\
(-1)^{k-1}\varepsilon_i\cdot \mat_{\sigma^k(i)}\times
\mat_{\sigma^{k-1}(i)}\times \cdots \mat_{\sigma^2(i)},
& i\in \{2,\ldots ,k\}.
\end{cases} 
\end{split}
\end{align}
On the other hand, we have
\begin{align*}
\left( \sigma\left( \twist{p} \right)\right)_i &=
\begin{cases}
\left(\twist{p}\right)_{\sigma(i)}, & i\not=1; \\
(-1)^{k-1} \left(\twist{p}\right)_{\sigma(i)}, & i=1
\end{cases} \\
&=
\begin{cases}
\varepsilon_{\sigma(i)}p_{\sigma^k(i)}\times p_{\sigma^{k-1}(i)}\times \cdots p_{\sigma^2(i)}, & i\not=1; \\
(-1)^{k-1}\varepsilon_{\sigma(i)}p_{\sigma^k(i)}\times p_{\sigma^{k-1}(i)}\times \cdots p_{\sigma^2(i)}, & i=1.
\end{cases}
\end{align*}
If $i=1$, the sign appearing here is $(-1)^{k-1}\varepsilon_{\sigma(1)}=(-1)^{k-1}\varepsilon_n=(-1)^{k-1}=\varepsilon_1$.
If $i\in [2,k-1]$, the sign is $\varepsilon_{\sigma(i)}=(-1)^{(i-1)(k-i+1)}=(-1)^{i(k-i)}(-1)^{k-1}=(-1)^{k-1}\varepsilon_i$.
If $i=k$, the sign is $\varepsilon_{\sigma(k)}=(-1)^{k-1}=(-1)^{k-1}\varepsilon_k$.
If $i\not\in \{1,\ldots ,k\}$,
the sign is $\varepsilon_{\sigma(i)}=\varepsilon_{i-1}=1=\varepsilon_i$. Hence, in all cases,
we obtain the same sign as in~\eqref{e:firsttwist}.
The result follows.

For (c), we have:
$$\sigma \left( \twist{\minor{I}}\right)(p)=
\minor{I}(\twist{\sigma(p)})=
\minor{I}\left( \sigma\left( \twist{p} \right) \right)=
\sigma(\minor{I})(\twist{p})=
\twist{\sigma(\minor{I})}(p).$$
\end{proof}

We note in passing that it follows from Lemma~\ref{l:rotatesigns} that $\sigma$ preserves the totally positive part of $\Gr_{k,n}$.

\begin{prop} \label{p:twistcomputation}
Let $I$ be a $k$-subset of $\{1,\ldots ,n\}$ expressed as a disjoint union
of the form $I_1\cup I_2$, where
$$I_1=\{\sigma^p(i),\ldots ,\sigma(i),i\}$$
and
$$I_2=\{\sigma^q(j),\ldots ,\sigma(j),j\},$$
with $p\geq 0$, $q\geq 0$ and $p+q+2=k$.
Let $J$ be the $k$-subset:
$$\{\sigma^{p+q+1}(i),\ldots ,\sigma^{p+1}(i)\}\cup
\{\sigma^{p+q+1}(j),\ldots ,\sigma^{q+1}(j)\}.$$
Then
\begin{equation} \label{e:twistofminor}
\twist{\minor{I}}=\minor{J}\prod_{r=1}^p \mscoeff{r}{i} \prod_{r=1}^q \mscoeff{r}{j}.
\end{equation}
\end{prop}

For example, we have seen (equation~\eqref{e:twistexample})
that when $k=3$ and $n=5$,
$$[124]\twistbracket{\mat}=[145][235].$$
Note that formula~\eqref{e:twistofminor}
has the interesting property that it is 
multiplicity-free, i.e.\ no Pl\"{u}cker coordinate appears
more than once in the product.

In order to prove Proposition~\ref{p:twistcomputation}, we recall the following way of expressing
compound determinants:

\begin{prop} \cite[\S3]{leclerc93} \label{p:compound}
Let $\tau=(\tau_{ij})$ be an $m\times n$ rectangular tableau
with entries $\tau_{ij}$. Suppose that
$\tau_{i,r_i}$ is boxed for $1\leq i\leq m$, where
$1\leq r_i\leq n$ for all $i$.
Let $X$ be the $m\times m$ matrix with $X_{ij}$ given by
the single row tableau whose entries are the $i$th row of
$\tau$ with $\tau_{i,r_i}$ replaced with $\tau_{j,r_j}$.
Then $\tau=\det(X)$.
\end{prop}

In~\cite[\S3]{leclerc93}, the author gives the
following examples:

\begin{spreadlines}{2ex}
\begin{align*}
\boxed{
\begin{matrix}
\boxed{\scriptstyle a} & b & c & d \\
\boxed{\scriptstyle e} & f & g & h
\end{matrix}
}&=
\renewcommand*{\arraystretch}{2}
\left| \,\,
\begin{matrix}
\boxed{\renewcommand*{\arraystretch}{0.7}
\begin{matrix} \mathbf{a} & b & c & d \end{matrix}}
&
\boxed{\renewcommand*{\arraystretch}{0.7}
\begin{matrix} \mathbf{e} & b & c & d \end{matrix}}
\\
\boxed{\renewcommand*{\arraystretch}{0.7}
\begin{matrix} \mathbf{a} & f & g & h \end{matrix}}
&
\boxed{\renewcommand*{\arraystretch}{0.7}
\begin{matrix} \mathbf{e} & f & g & h \end{matrix}}
\end{matrix}\,\,\right|\,;
\\
\boxed{
\begin{matrix}
\boxed{\scriptstyle a} & b & c & d \\
e & \boxed{\scriptstyle{f}} & g & h \\
i & j & \boxed{\scriptstyle{k}} & l \\
\end{matrix}
}&=
\renewcommand*{\arraystretch}{2}
\left| \,\,
\begin{matrix}
\boxed{\renewcommand*{\arraystretch}{0.7}
\begin{matrix} \mathbf{a} & b & c & d \end{matrix}}
&
\boxed{\renewcommand*{\arraystretch}{0.7}
\begin{matrix} \mathbf{f} & b & c & d \end{matrix}}
&
\boxed{\renewcommand*{\arraystretch}{0.7}
\begin{matrix} \mathbf{k} & b & c & d \end{matrix}}
\\
\boxed{\renewcommand*{\arraystretch}{0.7}
\begin{matrix} e & \mathbf{a} & g & h \end{matrix}}
&
\boxed{\renewcommand*{\arraystretch}{0.7}
\begin{matrix} e & \mathbf{f} & g & h \end{matrix}}
&
\boxed{\renewcommand*{\arraystretch}{0.7}
\begin{matrix} e & \mathbf{k} & g & h \end{matrix}}
\\
\boxed{\renewcommand*{\arraystretch}{0.7}
\begin{matrix} i & j & \mathbf{a} & l \end{matrix}}
&
\boxed{\renewcommand*{\arraystretch}{0.7}
\begin{matrix} i & j & \mathbf{f} & l \end{matrix}}
&
\boxed{\renewcommand*{\arraystretch}{0.7}
\begin{matrix} i & j & \mathbf{k} & l \end{matrix}}
\end{matrix}\,\,\right|\,.
\end{align*}
\end{spreadlines}

\begin{proof}[Proof of Proposition~\ref{p:twistcomputation}.]
Let $I,I_1,I_2,J$ be as in the statement of Proposition~\ref{p:twistcomputation}.
By Lemma~\ref{l:rotatesigns}, if Proposition~\ref{p:twistcomputation}
holds for $I,I_1,I_2,J$, then it also holds for
$\sigma(I),\sigma(I_1),\sigma(I_2),\sigma(J)$.
Hence, we may assume that
$I_1=\{1,\ldots ,p+1\}$ and
$I_2=\{j-q,\ldots ,j\}$ where
$1\leq p+1<\sigma^q(j)=j-q\leq j\leq n$.
We consider two possible cases.

\noindent \textbf{Case I}: $j-q>k$. \\
By Proposition~\ref{p:compound}:
$$\twist{\minor{I}}=
\boxed{
\begin{smallmatrix}
\boxed{\scriptstyle e_1} & n-k+2 & n-k+3 && \cdots & n-1 & n \\
1 & \boxed{\scriptstyle e_2} & n-k+3 && \cdots & n-1 & n \\
1 & 2 & \boxed{\scriptstyle e_3} && \cdots & n-1 & n \\
\vdots &&\ddots &&& \vdots \\
1 & \cdots & p & \boxed{\scriptstyle e_{p+1}} & \cdots & n-1 & n \\
j-q-k+1 && \cdots &&& j-q-1 & \boxed{\scriptstyle e_{p+2}} \\
j-q-k+2 && \cdots &&& j-q & \boxed{\scriptstyle e_{p+3}} \\
\vdots &&&&& \vdots \\
j-k+1 && \cdots &&& j-1 & \boxed{\scriptstyle e_k}
\end{smallmatrix}
}
$$
Applying Turnbull's identity (Proposition~\ref{p:turnbull}) for the first row with $B$ given by the
set of all unboxed entries in the first row, we have:
$$\twist{\minor{I}}=
\boxed{
\begin{smallmatrix}
e_1 & e_2 & e_3 && \cdots & e_{k-1} & e_k \\
1 & \boxed{\scriptstyle n-k+2} & n-k+3 && \cdots & n-1 & n \\
1 & 2 & \boxed{\scriptstyle n-k+3} && \cdots & n-1 & n \\
\vdots &&\ddots &&& \vdots \\
1 & \cdots & p & \boxed{\scriptstyle n-k+p+1} & \cdots & n-1 & n \\
j-q-k+1 && \cdots &&& j-q-1 & \boxed{\scriptstyle n-k+p+2} \\
j-q-k+2 && \cdots &&& j-q & \boxed{\scriptstyle n-k+p+3} \\
\vdots &&&&& \vdots \\
j-k+1 && \cdots &&& j-1 & \boxed{\scriptstyle n}
\end{smallmatrix}
}
$$

Since rows $2$ to $p+1$ contain the entries $n-k+p+2,n-k+p+3,\ldots ,n$
(not in boxes), any permutation of the entries in the boxes giving rise to a
non-zero product of minors must insert $n-k+2,n-k+3,\ldots ,n-k+p+1$ into
the boxes in rows $2$ to $p+1$. Since $n-k+3,n-k+4,\ldots ,n$ are non-boxed
entries in row $2$, $n-k+2$ must go in row $2$ (to get a non-zero term).
Using similar arguments for $n-k+3,\ldots , n-k+p+1$, we see that any
permutation giving rise to a non-zero term must fix the boxed elements
in rows $2,\ldots ,p+1$, so we can remove those boxes. We can also remove
the first row, as it is equal to $1$. Hence,
$$\twist{\minor{I}}=
\boxed{
\begin{smallmatrix}
1 & n-k+2 & n-k+3 && \cdots & n-1 & n \\
1 & 2 & n-k+3 && \cdots & n-1 & n \\
\vdots &&\ddots &&& \vdots \\
1 & \cdots & p & n-k+p+1 & \cdots & n-1 & n \\
j-q-k+1 && \cdots &&& j-q-1 & \boxed{\scriptstyle n-k+p+2} \\
j-q-k+2 && \cdots &&& j-q & \boxed{\scriptstyle n-k+p+3} \\
\vdots &&&&& \vdots \\
j-k+1 && \cdots &&& j-1 & \boxed{\scriptstyle n}
\end{smallmatrix}
}
$$
Applying Turnbull's identity for row $p+1$ (the first row
containing boxes), taking $B$ to be the first $q$ entries
in row $p+1$, we obtain:
$$\twist{\minor{I}}=
\boxed{
\begin{smallmatrix}
1 & n-k+2 & n-k+3 &&& \cdots && n-1 & n \\
1 & 2 & n-k+3 &&& \cdots && n-1 & n \\
\vdots &&\ddots &&&&&& \vdots \\
1 & \cdots & p & n-k+p+1 && \cdots && n-1 & n \\
n-k+p+3 & \cdots &&& n & \boxed{\scriptstyle j-k+1} & \cdots & \boxed{\scriptstyle j-q-1} & n-k+p+2 \\
j-q-k+2 &&& \cdots &&&& j-q & \boxed{\scriptstyle j-q-k+1} \\
\vdots &&&&&&&& \vdots \\
j-k+1 &&& \cdots &&&& j-1 & \boxed{\scriptstyle j-k}
\end{smallmatrix}
}\, ,
$$
noting that the $n$ in the row with first entry $n-k+p+3$ could be to the
left of the $p$ in the row above.

Note that $j-k+1,\ldots ,j-q-1$ occur in the last $q$ rows, so if the
boxed entries in the row beginning $n-k+p+3$ are permuted into one of
these rows, we get a zero term. Hence the boxed entries in the last
column must be permuted within this column. But each such entry cannot
be permuted into an earlier row (without getting zero), so they must
be fixed, and we can remove the boxes in the last column. But then all
the remaining boxes lie in a single row, and we are left only with the
identity permutation (as we ignore permutations of boxes in a single
row), and thus we can remove all of the boxes to obtain:

$$
\twist{\minor{I}} =
\boxed{
\begin{smallmatrix}
1 & n-k+2 & n-k+3 &&& \cdots && n-1 & n \\
1 & 2 & n-k+3 &&& \cdots && n-1 & n \\
\vdots &&\ddots &&&&&& \vdots \\
1 & \cdots & p & n-k+p+1 && \cdots && n-1 & n \\
n-k+p+3 && \cdots && n & j-k+1 & \cdots & j-q-1 & n-k+p+2 \\
j-q-k+2 &&& \cdots &&&& j-q & j-q-k+1 \\
\vdots &&&&&&&& \vdots \\
j-k+1 &&& \cdots &&&& j-1 & j-k
\end{smallmatrix}
}\, .
$$

We may then permute the entries in the last rows, applying a cyclic
permutation once to each of the last $q$ rows and a cyclic permutation
$q$ times to the row immediately above them. This gives us a sign
contribution, $(-1)^{(k-1)(q+q)}=1$, and we obtain, noting that
$n-k+p+2=n-q$:

\begin{align*}
\begin{split}
\twist{\minor{I}} &=
\boxed{
\begin{smallmatrix}
1 & n-k+2 & n-k+3 &&& \cdots && n-1 & n \\
1 & 2 & n-k+3 &&& \cdots && n-1 & n \\
\vdots &&\ddots &&&&&& \vdots \\
1 & \cdots & p & n-k+p+1 && \cdots && n-1 & n \\
j-k+1 && \cdots && j-q-1 & n-q & \cdots & n-1 & n \\
j-q-k+1 &&& \cdots &&&&& j-q \\
\vdots &&&&&&&& \vdots \\
j-k &&& \cdots &&&&& j-1
\end{smallmatrix}
}
\\
&=\minor{J} \prod_{r=1}^p \mcoeff{r} \prod_{r=j-q}^{j-1} \mcoeff{r},
\end{split}
\end{align*}
as required.

\noindent \textbf{Case II}: $j-q\leq k\leq j$. \\
By Proposition~\ref{p:compound}:
$$\twist{\minor{I}}=
\boxed{
\begin{smallmatrix}
\boxed{\scriptstyle e_1} & n-k+2 & n-k+3 &&&& \cdots &&&& n-1 & n \\
1 & \boxed{\scriptstyle e_2} & n-k+3 &&&& \cdots &&&& n-1 & n \\
1 & 2 & \boxed{\scriptstyle e_3} &&&& \cdots &&&& n-1 & n \\
\vdots &&\ddots &&& \vdots \\
1 & \cdots & p & \boxed{\scriptstyle e_{p+1}} &&& \cdots &&&& n-1 & n \\
1 && \cdots && j-q-2 & j-q-1 & \boxed{\scriptstyle e_{p+2}} & j-q-k+1+n && \cdots & n-1 & n \\
1 && \cdots &&& j-q-1 & j-q & \boxed{\scriptstyle e_{p+3}} & j-q-k+2+n & \cdots & n-1 & n \\
\vdots &&&&& &&& \ddots &&& \vdots \\
1 &&&&& \cdots &&&&& k-1 & \boxed{\scriptstyle 2k-j} \\
2 &&&&& \cdots &&&&& k & \boxed{\scriptstyle 2k+1-j} \\
\vdots &&&&& &&&&&& \vdots \\
j-k+1 &&&&& \cdots &&&&& j-1 & \boxed{\scriptstyle k} \\
\end{smallmatrix}
}
$$
Applying Turnbull's identity for the first row,
taking $B$ to be the set of all unboxed entries in the first
row, we obtain that $\twist{\minor{I}}$ is given by:
$$\boxed{
\begin{smallmatrix}
e_1 & e_2 & e_3 &&&& \cdots &&&& e_{n-1} & e_n \\
1 & \boxed{\scriptstyle n-k+2} & n-k+3 &&&& \cdots &&&& n-1 & n \\
1 & 2 & \boxed{\scriptstyle n-k+3} &&&& \cdots &&&& n-1 & n \\
\vdots &&\ddots &&& \vdots \\
1 & \cdots & p & \boxed{\scriptstyle n-k+p+1} &&& \cdots &&&& n-1 & n \\
1 && \cdots && j-q-2 & j-q-1 & \boxed{\scriptstyle n-k+p+2} & j-q-k+1+n && \cdots & n-1 & n \\
1 && \cdots &&& j-q-1 & j-q & \boxed{\scriptstyle n-k+p+3} & j-q-k+2+n & \cdots & n-1 & n \\
\vdots &&&&& &&& \ddots &&& \vdots \\
1 &&&&& \cdots &&&&& k-1 & \boxed{\scriptstyle n+k-j} \\
2 &&&&& \cdots &&&&& k & \boxed{\scriptstyle n+k+1-j} \\
\vdots &&&&& &&&&&& \vdots \\
j-k+1 &&&&& \cdots &&&&& j-1 & \boxed{\scriptstyle n} \\
\end{smallmatrix}
}
$$
As in Case I, we can remove row $1$ and the boxes in rows $2$ to $p+1$
to obtain that $\twist{\minor{I}}$ is given by:
$$\boxed{
\begin{smallmatrix}
1 & n-k+2 & n-k+3 &&&& \cdots &&&& n-1 & n \\
1 & 2 & n-k+3 &&&& \cdots &&&& n-1 & n \\
\vdots &&\ddots &&& \vdots \\
1 & \cdots & p & n-k+p+1 &&& \cdots &&&& n-1 & n \\
1 && \cdots && j-q-2 & j-q-1 & \boxed{\scriptstyle n-k+p+2} & j-q-k+1+n && \cdots & n-1 & n \\
1 && \cdots &&& j-q-1 & j-q & \boxed{\scriptstyle n-k+p+3} & j-q-k+2+n & \cdots & n-1 & n \\
1 && \cdots &&&& j-q & j-q+1 & \boxed{\scriptstyle n-k+p+4} & \cdots & n-1 & n \\
\vdots &&&&& &&& \ddots &&& \vdots \\
1 &&&&& \cdots &&&&& k-1 & \boxed{\scriptstyle n+k-j} \\
2 &&&&& \cdots &&&&& k & \boxed{\scriptstyle n+k+1-j} \\
\vdots &&&&& &&&&&& \vdots \\
j-k+1 &&&&& \cdots &&&&& j-1 & \boxed{\scriptstyle n} \\
\end{smallmatrix}
}
$$
Applying an appropriate cyclic permutation to row $p+1$, we obtain that
$\twist{\minor{I}}$ is $(-1)^{(k-1)(j-q-1)}$ times the following:
$$
\boxed{
\begin{smallmatrix}
1 & n-k+2 & n-k+3 &&&& \cdots &&&& n-1 & n \\
1 & 2 & n-k+3 &&&& \cdots &&&& n-1 & n \\
\vdots &&\ddots &&& \vdots \\
1 & \cdots & p & n-k+p+1 &&& \cdots &&&& n-1 & n \\
\boxed{\scriptstyle n-k+p+2} & j-q-k+1+n & \cdots & n-1 & n & 1 && \cdots &&& j-q-2 & j-q-1 \\
1 && \cdots &&&& j-q & \boxed{\scriptstyle n-k+p+3} & j-q-k+2+n & \cdots & n-1 & n \\
1 && \cdots &&&& j-q & j-q+1 & \boxed{\scriptstyle n-k+p+4} & \cdots & n-1 & n \\
\vdots &&&&& &&& \ddots &&& \vdots \\
1 &&&&& \cdots &&&&& k-1 & \boxed{\scriptstyle n+k-j} \\
2 &&&&& \cdots &&&&& k & \boxed{\scriptstyle n+k+1-j} \\
\vdots &&&&& &&&&&& \vdots \\
j-k+1 &&&&& \cdots &&&&& j-1 & \boxed{\scriptstyle n} \\
\end{smallmatrix}
}
$$
Applying Turnbull's identity to row $p+1$ (the first
row containing boxed elements) with $B$ given by
the entries in row $p+1$ in columns $2,\ldots ,q+1$,
we obtain that $\twist{\minor{I}}$ is $(-1)^{(k-1)(j-q-1)}$ times the following:
$$
\boxed{
\begin{smallmatrix}
1 & n-k+2 & n-k+3 &&&& \cdots &&&& n-1 & n \\
1 & 2 & n-k+3 &&&& \cdots &&&& n-1 & n \\
\vdots &&\ddots &&& \vdots \\
1 & \cdots & p & n-k+p+1 &&& \cdots &&&& n-1 & n \\
n-k+p+2 & n-k+p+3 & \cdots && n & \boxed{\scriptstyle j-k+1} && \cdots &&& \boxed{\scriptstyle j-q-2} & \boxed{\scriptstyle j-q-1} \\
1 && \cdots &&&& j-q & \boxed{\scriptstyle j-q-k+1+n} & j-q-k+2+n & \cdots & n-1 & n \\
1 && \cdots &&&& j-q & j-q+1 & \boxed{\scriptstyle j-q-k+2+n} & \cdots & n-1 & n \\
\vdots &&&&& &&& \ddots &&& \vdots \\
1 &&&&& \cdots &&&&& k-1 & \boxed{\scriptstyle n} \\
2 &&&&& \cdots &&&&& k & \boxed{\scriptstyle 1} \\
\vdots &&&&& &&&&&& \vdots \\
j-k+1 &&&&& \cdots &&&&& j-1 & \boxed{\scriptstyle j-k} \\
\end{smallmatrix}
}
$$
As in Case (I), we can remove all the boxes, to obtain
that $\twist{\minor{I}}$ is $(-1)^{(k-1)(j-q-1)}$ times the following:
$$
\boxed{
\begin{smallmatrix}
1 & n-k+2 & n-k+3 &&&& \cdots &&&& n-1 & n \\
1 & 2 & n-k+3 &&&& \cdots &&&& n-1 & n \\
\vdots &&\ddots &&& \vdots \\
1 & \cdots & p & n-k+p+1 &&& \cdots &&&& n-1 & n \\
n-k+p+2 & n-k+p+3 & \cdots && n & j-k+1 && \cdots &&& j-q-2 & j-q-1 \\
1 && \cdots &&&& j-q & j-q-k+1+n & j-q-k+2+n & \cdots & n-1 & n \\
1 && \cdots &&&& j-q & j-q+1 & j-q-k+2+n & \cdots & n-1 & n \\
\vdots &&&&& &&& \ddots &&& \vdots \\
1 &&&&& \cdots &&&&& k-1 & n \\
2 &&&&& \cdots &&&&& k & 1 \\
\vdots &&&&& &&&&&& \vdots \\
j-k+1 &&&&& \cdots &&&&& j-1 & j-k \\
\end{smallmatrix}
}
$$
We may then permute the entries in the last rows, applying a cyclic
permutation once to each of the last $j-k$ rows and a cyclic permutation
$q+1$ times to the row immediately above them. This gives us a sign
contribution, $(-1)^{(k-1)(q+1+j-k)}$. Combined with the sign $(-1)^{(k-1)(j-q-1)}$
we already have, this becomes $1$. We obtain, noting that
$n-k+p+2=n-q$:
\begin{align*}
\begin{split}
\twist{\minor{I}} &=
\boxed{
\begin{smallmatrix}
1 & n-k+2 & n-k+3 &&&& \cdots &&&& n-1 & n \\
1 & 2 & n-k+3 &&&& \cdots &&&& n-1 & n \\
\vdots &&\ddots &&& \vdots \\
1 & \cdots & p & n-k+p+1 &&& \cdots &&&& n-1 & n \\
j-k+1 & \cdots & j-q-2 & j-q-1 & n-q & n-q+1 &&& \cdots && n-1 & n \\
1 && \cdots &&&& j-q & j-q-k+1+n & j-q-k+2+n & \cdots & n-1 & n \\
1 && \cdots &&&& j-q & j-q+1 & j-q-k+2+n & \cdots & n-1 & n \\
\vdots &&&&& &&& \ddots &&& \vdots \\
1 &&&&& \cdots &&&&& k-1 & n \\
1 &&&&& \cdots &&&&& k-1 & k \\
\vdots &&&&& &&&&&& \vdots \\
j-k & j-k+1 &&&& \cdots &&&&& j-2 & j-1 \\
\end{smallmatrix}
}
\\
&=\minor{J} \prod_{r=1}^p \mcoeff{r} \prod_{r=j-q}^{j-1} \mcoeff{r},
\end{split}
\end{align*}
as required.
The proposition is proved.
\end{proof}

\begin{remark} \label{r:coefficientcase}
If $k>1$, we can take $I$ and $J$
in Proposition~\ref{p:twistcomputation}
to be non-empty disjoint subsets whose
union is $\mcoeff{i}$. We obtain:
\begin{equation} \label{e:coefficienttwist}
\twist{\mcoeff{i}}=\mscoeff{k-1}{i} \cdots 
\mscoeff{}{i}.
\end{equation}
Note that it is easy to check this
directly in the case $k=1$: the left
hand side evaluates to $1$ and we view the
product on the right-hand-side as an empty
product.

By~\eqref{e:coefficienttwist}, the 
domain of the twist contains
the subvariety $\Gr_{k,n}^*$ consisting of 
matrices $p$ for which $\mcoeff{i}(p)\not=0$ 
for $1\leq i\leq n$; furthermore, this
subvariety is stable under the twist.
Note that this variety is the largest
\emph{open positroid subvariety} of $\Gr_{k,n}$
(see~\cite[\S1]{mullerspeyer}).
\end{remark}

\section{Periodicity}
\label{s:periodicity}

In this section, we use the cross product formulation to show that,
if $k>1$, then applying the twist $2n$ times to a Pl\"{u}cker coordinate gives back the same Pl\"{u}cker coordinate multiplied by a
monomial in the Pl\"{u}cker coordinates $\mcoeff{i}$, $i\in \{1,2,\ldots ,n\}$.

\begin{lemma} \label{l:doubletwistlemma}
Suppose that $k>1$, and let $\mat\in M_{k,n}$ be a $k\times n$ matrix with $k>1$. Then, for $1\leq i\leq n$, we have:
\begin{equation}
\dtmats{i}=(-1)^{k-1}\varepsilon_{\sigma(i)}\varepsilon_i
\mscoeff{2}{i}(\mat)\cdots \mscoeff{k-1}{i}(\mat) \cdot \mat_{\sigma^k(i)}.
\label{e:doubletwist}
\end{equation}
\end{lemma}
\begin{proof}
The formula~\eqref{e:doubletwist} is a polynomial identity in the matrix entries
of $\mat$ and so it is enough to verify this formula when $\mat$ varies over any fixed non-empty Zariski-open subset of the variety of all $k\times n$ matrices.
We restrict attention to $k\times n$ matrices $\mat$ satisfying the (open)
determinantal conditions $\mcoeff{i}(\mat)\not=0$ for all $i\in \{1,\ldots ,n\}$. In particular,
this implies that each column vector $\mat_i\not=0$ for $i\in \{1,\ldots ,n\}$.

We begin by computing the scalar product of the $i$th column $\dtmats{i}$ of $\dtmat$ with $\mat_{\sigma^k(i)}$. We have:

\begin{align} \label{e:scalarproduct}
\begin{split}
\bilinearb{\dtmats{i}}{\mat_{\sigma^k(i)}} &=
\varepsilon_i \bilinearb{\tmats{\sigma^{k-1}(i)} \times \cdots \times \tmats{\sigma(i)}}{\mat_{\sigma^k(i)}} \\
&= \varepsilon_i \det \left( \tmats{\sigma^{k-1}(i)},\ldots ,\tmats{\sigma(i)},\mat_{\sigma^k(i)} \right) \\
&=(-1)^{k-1} \varepsilon_i \bilinearb{\mat_{\sigma^k(i)}\times \tmats{\sigma^{k-1}(i)}\times
\cdots \times \tmats{\sigma^2(i)}}{\tmats{\sigma(i)}} \\
&= (-1)^{k-1}\varepsilon_{\sigma(i)} \varepsilon_i \bilinearb{
\mat_{\sigma^k(i)}\times \tmats{\sigma^{k-1}(i)} \times\cdots \times \tmats{\sigma^2(i)}}{\mat_{\sigma^k(i)}\times \cdots \times \mat_{\sigma^2(i)}}.
\end{split}
\end{align}
For $1\leq s,t\leq k-1$, we have:
\begin{align*}
\begin{split}
\bilinearb{\tmats{\sigma^s(i)}}{\mat_{\sigma^t(i)}}&=
\varepsilon_{\sigma^s(i)} \bilinearb{\mat_{\sigma^{s+k-1}(i)}\times \cdots \times \mat_{\sigma^{s+1}(i)}}{\mat_{\sigma^t(i)}} \\
&=\varepsilon_{\sigma^s(i)} \det \left( \mat_{\sigma^{s+k-1}(i)},\ldots ,\mat_{\sigma^{s+1}(i)},
\mat_{\sigma^t(i)} \right).
\end{split}
\end{align*}
If $2\leq s<t\leq k-1$, then this is zero since the column $\mat_{\sigma^t(i)}$
is repeated.
If $s=t$, we obtain                  
\begin{align*}
\begin{split}
\bilinearb{\tmats{\sigma^s(i)}}{\mat_{\sigma^s(i)}}&=
\varepsilon_{\sigma^s(i)} \bilinearb{\mat_{\sigma^{s+k-1}(i)}\times \cdots \times \mat_{\sigma^{s+1}(i)}}{\mat_{\sigma^s(i)}} \\
&=\mscoeff{s}{i}(\mat).
\end{split}
\end{align*}
Consequently, after applying the contraction formula, we obtain
$$\bilinearb{\dtmats{i}}{\mat_{\sigma^k(i)}} =
(-1)^{k-1} \varepsilon_{\sigma(i)}\varepsilon_i \det
\begin{pmatrix}
|\mat_{\sigma^k(i)}|^2 & \cdots && * \\
\vdots & \mscoeff{k-1}{i}(\mat) && \\
&& \ddots &  \\
0 &  &  & \mscoeff{2}{i}(\mat)
\end{pmatrix}.
$$
It follows that
$$\bilinearb{\dtmats{i}}{\mat_{\sigma^k(i)}} =
(-1)^{k-1}\varepsilon_{\sigma(i)}\varepsilon_i \cdot \mscoeff{2}{i}(\mat) \cdots \mscoeff{k-1}{i}(\mat) \cdot |\mat_{\sigma^k(i)}|^2,$$
which, in view of our assumptions, is non-zero.
From this and the second line of~\eqref{e:scalarproduct} we may conclude that
the vectors
$$\tmats{\sigma^{k-1}(i)},\ldots ,\tmats{\sigma(i)},\mat_{\sigma^k(i)}$$
form a basis for $\mathbb{C}^k$. Since
$$\dtmats{i}=\varepsilon_i \left( \tmats{\sigma^{k-1}(i)}\times \cdots \times \tmats{\sigma(i)} \right)$$
is clearly orthogonal to each of the basis vectors
$$\tmats{\sigma^{k-1}(i)},\ldots ,\tmats{\sigma(i)}$$
(as is $\mat_{\sigma^k(i)}$) it follows that
$$\dtmats{i}=(-1)^{k-1}\varepsilon_{\sigma(i)}\varepsilon_i
\mscoeff{2}{i}(\mat)\cdots \mscoeff{k-1}{i}(\mat) \cdot \mat_{\sigma^k(i)},$$
and we are done.
\end{proof}

\begin{corollary} \label{c:periodic}
Let $I$ be a $k$-subset of $\{1,\ldots ,n\}$, with $k>1$. Then we have:
$$\dtwist{\minor{I}}=\minor{\sigma^k(I)}\cdot \prod_{i\in I} \mscoeff{2}{i}
\cdots \mscoeff{k-1}{i}.$$
\end{corollary}
\begin{proof}
By Lemma~\ref{l:doubletwistlemma}, we have:
$$\dtwist{\minor{I}}=(-1)^{s(k-s)} \minor{\sigma^k(I)}\cdot \prod_{i\in I} (-1)^{k-1}\varepsilon_{\sigma(i)}\varepsilon_i \mscoeff{2}{i}\cdots \mscoeff{k-1}{i},$$
where
$$s=|\{i\in I\,:i\leq k\}|.$$
An easy computation shows that:
$$(-1)^{k-1}\varepsilon_{\sigma(i)}\varepsilon_i=\begin{cases}
1, & \text{if $i\leq k$}; \\
(-1)^{k-1}, & \text{otherwise}.
\end{cases}$$
Hence
$$(-1)^{s(k-s)}\prod_{i\in I}(-1)^{k-1}\varepsilon_{\sigma(i)}\varepsilon_i=
(-1)^{s(k-s)}(-1)^{(k-s)(k-1)}=(-1)^{k(k-1)-s(s-1)}=1,$$
and the result follows.
\end{proof}

\begin{prop} \label{p:periodicitymonomial}
Suppose that $k>1$ and $I$ is a $k$-subset of $\{1,\ldots ,n\}$.
Then, applying the twist $2n$ times to $[I]\in \mathbb{C}[\Gr_{k,n}]$
gives $[I]$ multiplied by a monomial in the Pl\"{u}cker coordinates
$\mcoeff{i}$, $i\in \{1,\ldots ,n\}$.
\end{prop}

\begin{proof}
The result follows from Corollary~\ref{c:periodic} and the fact
that the twist of a Pl\"{u}cker coordinate of the form $\mcoeff{i}$
with $1\leq i \leq n$ is a monomial in the
Pl\"{u}cker coordinates of the same form
(see Remark~\ref{r:coefficientcase}).
\end{proof}

We will give an cluster algebra-theoretic periodicity statement for the twist in Proposition~\ref{p:periodicitycoefficients}.

\section{Relationship to the BFZ-twist}
\label{s:BFZtwist}

Our aim in this section is to the explain the relationship between the twist discussed 
in Section~\ref{s:twist} and the
BFZ-twist map~\cite{BFZ96, BZ97}.
We consider the quotient map
from $SL_n(\mathbb{C})$ to the Grassmannian, restricted to
the appropriate unipotent cell. The image of this map
is the open positroid variety $\Gr_{k,n-k}^*$, and we give an
explicit inverse from $\Gr_{k,n-k}^*$ to the unipotent cell
(Theorem~\ref{t:inverse}).
This allows us to give a formula (Proposition~\ref{p:relationship}) relating the twist defined
in Section~\ref{s:twist} and the BFZ-twist.

Let $G=\text{SL}_n(\mathbb{C})$, and let $N$ denote the unipotent group of all complex $n\times n$ unipotent matrices in $G$.
The BFZ-\emph{twist} is 
a regular automorphism defined on a stratum
of $N$ known as a \emph{unipotent cell},
defined as follows.

Let $T$ denote the torus of diagonal matrices in $SL_n(\mathbb{C})$.
The Weyl group $W=N_{SL_n(\mathbb{C})}(T)/T$ of $SL_n(\mathbb{C})$ is the symmetric group $S_n$.
For each $w\in W$ we shall denote by $\dot{w}$ a lifting of $w$
to $N_{SL_n(\mathbb{C})}(T)$. We shall often write this just as
$w$ if there is no ambiguity.

Given $w\in S_n$, the unipotent cell $N^w$ 
associated to $w$ is defined as the intersection
$N \cap B_{-}wB_{-}$; here $B_{-}$ is the group of 
all invertible complex $n \times n$ lower 
triangular matrices.
Each unipotent cell is a quasi-affine complex 
algebraic variety whose dimension is $\ell(w)$.
Given an element $g \in N^w$ and its transpose
$g^T \in B_{-}$ the intersection $ N \cap B_{-}wg^T$ always consists a single element, denoted $\eta(g)$. The map $g \mapsto \eta(g)$ defines a regular automorphism of $N^w$ which is called the \emph{BFZ-twist}.

The concept was introduced as a tool for evaluating factorizations of unipotent matrices within a fixed cell $N^w$ into elementary \emph{Jacobi matrices}. More specifically, a combinatorial ansatz expresses the complex parameters associated
to a unique factorization of a (generic) element in a fixed unipotent cell as
a Laurent monomial in matrix minors of the corresponding (inverse) twisted element.

The Grassmannian $\Gr_{k,n}$ inherits a birational version of the twist
from the unipotent cell $N^w$ associated to the Grassmann permutation $w$
given by $w(i)= i + k \mod n$ for $i \in \{1 \dots n\}$
via the restriction $q$ of the natural projection from $G$ to $\Gr_{k,n}$, defined by
mapping an element $g \in \text{SL}_n\big( \mathbb{C} \big)$ to the $k \times n$ submatrix given by its first $k$ rows, as we shall now explain. Note that $q$ can be identified with
the quotient map $G\rightarrow P_- \backslash G$
for a maximal parabolic subgroup $P_-$.

For subsets $I,J$ of $\{1,\ldots ,n\}$ of the same cardinality, let $\Delta_{I,J}$ denote the corresponding minor, as a function on $SL_n(\mathbb{C})$,
with row-set $I$ and column-set $J$.
We first check that the image of $q$ is
contained in the open positroid variety $\Gr_{k,n}^*$
in $\Gr_{k,n}$ defined by the non-vanishing of the
Pl\"{u}cker coordinates $\mcoeff{i}$ for $i\in \{1,\ldots ,n\}$.

\begin{lemma} \label{l:qtarget}
Let $x\in N^w$. Then $q(x)\in \Gr_{k,n}^*$.
\end{lemma}
\begin{proof}
Let $x\in N^w$. We can write $x$ in the form $y_1wy_2$,
where $y_1,y_2\in B_-$.
We must show that $\Delta_{\{1,\ldots ,k\},\coeff{r}}(x)\not=0$,
for $1\leq r\leq n$.
Suppose first that $1\leq r\leq k-1$.
Then, first using the fact that $x\in N$ and then using
the Cauchy-Binet formula, we have:
\begin{align*}
\Delta_{\{1,\ldots ,k\},\{1,\ldots ,r\}\cup \{r-k+1+n,\ldots ,n\}]}(x)
&=
\Delta_{\{r+1,\ldots ,k\},\{r-k+1+n,\ldots ,n\}]}(x) \\
&=\sum_{I,J} \Delta_{\{r+1,\ldots ,k\}],I}(y_1)\Delta_{I,J}(w)
\Delta_{J,\{r-k+1+n,\ldots ,n\}]}(y_2),
\end{align*}
where the sum is over all $(k-r)$-subsets $I,J$ of $\{1,\ldots ,n\}$.
Since $y_2\in B_-$, $\Delta_{J,\{r-k+1+n,\ldots n\}}(y_2)$ is only
non-zero if $J=\{r-k+1+n,\ldots ,n\}$. Furthermore, $\Delta_{I,J}(w)$
is non-zero if and only if $I=w(J)=\{r+1,\ldots ,k\}$. We also have
$\Delta_{\{r+1,\ldots ,k\},\{r+1,\ldots ,k\}}(y_1)\not=0$, since $y_1\in B_-$.
It follows that
$\Delta_{\{1,\ldots ,k\},\{1,\ldots ,r\}\cup \{r-k+1+n,\ldots ,n\}}(x)\not=0$, as
required.

Suppose next that $k\leq r\leq n$.
Then, first using the fact that $x\in N$ and then using
the Cauchy-Binet formula, we have:
\begin{align*}
\Delta_{\{1,\ldots k\},\{r-k+1,\ldots r\}}(x)
&=
\Delta_{\{1,\ldots ,k\}\cup \{r+1,\ldots ,n\},\{r-k+1,\ldots ,n\}}(x) \\
&=\sum_{I,J} \Delta_{\{1,\ldots ,k\}\cup \{r+1,\ldots ,n\},I}(y_1)\Delta_{I,J}(w)
\Delta_{J,\{r-k+1,\ldots ,n\}}(y_2),
\end{align*}
where the sum is over all $(k+n-r)$-subsets $I,J$ of $\{1,\ldots ,n\}$.
Since $y_2\in B_-$, $\Delta_{J,\{r-k+1,\ldots ,n\}}(y_2)$ is only
non-zero if $J=\{r-k+1,\ldots ,n\}$. Furthermore, $\Delta_{I,J}(w)$
is non-zero if and only if $I=w(J)=\{1,\ldots ,k\}\cup \{r+1,\ldots ,n\}$.
We also have
$\Delta_{\{1,\ldots ,k\}\cup \{r+1,\ldots ,n\},\{1,\ldots ,k\}\cup \{r+1,\ldots ,n\}}(y_1)\not=0$, since $y_1\in B_-$.
It follows that
$\Delta_{\{1,\ldots ,k\},\{r-k+1,\ldots ,r\}}(x)\not=0$, as
required. We are done.
\end{proof}

Let $\varphi:\Gr_{k,n}^*\rightarrow N$ be the map defined by
$$\varphi(\mat)_{ij} \ = \
\begin{dcases}
\frac{\big[\coeff{i} - \{i\} \cup \{j\} \big](\mat)}{\vphantom{2^{2^{2}}} \mcoeff{i}(\mat) }, & i\leq j+k-1; \\
0, & i\geq j+k.
\end{dcases}
$$
We will show that $\varphi$ is the inverse
of $q$ restricted to $N^w$. We first need the
following.

\begin{lemma} \label{l:phiminor}
Let $J=\{j_1<j_2<\cdots <j_k\}$ be a $k$-subset of $\{1,\ldots ,n\}$ and $0\leq r\leq j_1-1$. Let $p\in \Gr_{k,n}^*$. Then, we have:
\begin{enumerate}[(a)]
\item
$$\Delta_{\{r+1,\ldots ,r+k\},J}\left(\vphantom{{2^2}^2}\varphi(p)\right)=\frac{[J](p)}{\mcoeff{s}(p)},$$
where $s=r+k$, and
\item
$$\Delta_{\{r,\ldots ,r+k\},J}\left(\vphantom{{2^2}^2}\varphi(p)\right)=0.$$
\end{enumerate}
\end{lemma}
\begin{proof}
Note that, by the assumptions, $0\leq r\leq n-k$.
We first consider the proof of (a).
Since $r\leq j_1-1$, we have
$r+a\leq r+k\leq j_1+k-1<j_b+k-1$
for $1\leq a\leq k$ and $1\leq b\leq k$. It follows that for all $(i,j)\in
\{r+1,\ldots ,r+k\}\times J$,
we have
$$\varphi(p)_{ij}=\frac{\big[\coeff{i} - \{i\} \cup \{j\} \big](\mat)}{ \mcoeff{i}(\mat) }.$$
Hence, $\Delta_{\{r+1,\ldots ,r+k\},J}\left(\vphantom{{2^2}^2}\varphi(p)\right)$
is equal to the following (dropping the notation
$(p)$ to save space):
\begingroup
\def\arraystretch{4}
\begin{equation*}
\begin{array}{cc}
\displaystyle\frac{1}{\prod_{s=r+1}^{r+k}\mcoeff{s}}
\left|
\begin{smallmatrix}
{[1,\ldots ,r,j_1,r-k+2+n,\ldots ,n]} & {[1,\ldots ,r,j_2,r-k+2+n,\ldots ,n]} & \cdots & {[1,\ldots ,r,j_k,r-k+2+n,\ldots ,n]} \\
{[1,\ldots ,r+1,j_1,r-k+3+n,\ldots ,n]} & {[1,\ldots ,r+1,j_2,r-k+3+n,\ldots ,n]} & \cdots & {[1,\ldots ,r+1,j_k,r-k+3+n,\ldots ,n]} \\
\vdots & \vdots & & \vdots \\
{[1,\ldots ,k-1,j_1]} & {[1,\ldots ,k-1,j_2]} & \cdots & {[1,\ldots ,k-1,j_k]} \\
\vdots & \vdots & & \vdots \\
{[r+1,\ldots ,r+k-1,j_1]} & {[r+1,\ldots ,r+k-1,j_2]} & \cdots & {[r+1,\ldots ,r+k-1,j_k]}
\end{smallmatrix}
\right|, & r<k-1 \\
\displaystyle\frac{1}{\prod_{s=r+1}^{r+k}\mcoeff{s}}
\left|
\begin{smallmatrix}
{[r-k+2,\ldots ,r,j_1]} & {[r-k+2,\ldots ,r,j_2]} & \cdots & {[r-k+2,\ldots ,r,j_k]} \\
{[r-k+3,\ldots ,r+1,j_1]} & {[r-k+3,\ldots ,r+1,j_2]} & \cdots & {[r-k+3,\ldots ,r+1,j_k]} \\
\vdots & \vdots & & \vdots \\
{[r+1,\ldots ,r+k-1,j_1]} & {[r+1,\ldots ,r+k-1,j_2]} & \cdots & {[r+1,\ldots ,r+k-1,j_k]}
\end{smallmatrix}
\right|, & r\geq k-1.
\end{array}
\end{equation*}
\endgroup
We can rewrite this (by Proposition~\ref{p:compound}) as:
\begingroup
\def\arraystretch{6}
\begin{equation*}
\Delta_{\{r+1,\ldots ,r+k\},J}\left(\vphantom{{2^2}^2}\varphi(p)\right)=
\left\{
\begin{array}{cc}
\displaystyle\frac{1}{\prod_{s=r+1}^{r+k}\mcoeff{s}}\,\,

\boxed{
\begin{smallmatrix}
1 & \cdots & r & \boxed{\scriptstyle j_1} & r-k+2+n                  & r-k+3+n & \cdots & n\\
1 & \cdots & r & r+1                      & \boxed{\scriptstyle j_2} & r-k+3+n & \cdots & n\\
\vdots &   &   &                          &                          &         &        & \vdots \\
1 &        &   &                          & \cdots                   &         & k-1    & \boxed{\scriptstyle j_{k-r}} \\
\vdots &   &   &                          &                          &         &        & \vdots \\
r+1 &        & &                          & \cdots                 &         & r+k-1   & \boxed{\scriptstyle j_k}
\end{smallmatrix}} \, ,
& r<k-1; \\
\displaystyle\frac{1}{\prod_{s=r+1}^{r+k}\mcoeff{s}}\,\,
\boxed{
\begin{smallmatrix}
r-k+2 & r-k+3 & \cdots & r & \boxed{\scriptstyle 
j_1} \\
r-k+3 & r-k+4 & \cdots & r+1 & 
\boxed{\scriptstyle j_2} \\
\vdots & \vdots & & \vdots \\
r+1 & r+2 & \cdots & r+k-1 & \boxed{\scriptstyle 
j_k}
\end{smallmatrix}}\, ,
& r \geq k-1.
\end{array}
\right.
\end{equation*}
\endgroup

We apply Turnbull's identity for the first row with $B$
given by the set of all unboxed entries in the first
row, obtaining:
\begin{equation*}
\Delta_{\{r+1,\ldots ,r+k\},J}\left(\vphantom{{2^2}^2}\varphi(p)\right)=\left\{
\def\arraystretch{6}
\begin{array}{cc}
\displaystyle\frac{1}{\prod_{s=r+1}^{r+k}\mcoeff{s}}\,\,
\boxed{
\begin{smallmatrix}
j_{k+1-r} & \cdots & j_k & j_1 & j_2 & \cdots & & j_{k-r} \\
1 & \cdots & & r+1 & \boxed{\scriptstyle r-k+2+n} & r-k+3+n & \cdots & n \\
1 & \cdots & & & r+2 & \boxed{\scriptstyle r-k+3+n} & \cdots & n \\
\vdots &&&&&&& \vdots \\
1 &&&& \cdots && k-1 & \boxed{\scriptstyle n} \\
2 &&&& \cdots && k & \boxed{\scriptstyle 1} \\
\vdots &&&&&&& \vdots \\
r+1 &&&& \cdots && r+k-1 & \boxed{\scriptstyle r}
\end{smallmatrix}}\, , & r<k-1; \\
\displaystyle\frac{1}{\prod_{s=r+1}^{r+k}\mcoeff{s}}\,\,
\boxed{
\begin{smallmatrix}
\scriptstyle j_2 & \scriptstyle j_3 & \cdots & 
\scriptstyle j_k & \scriptstyle j_1 \\
r-k+3 & r-k+4 & \cdots & r+1 & 
\boxed{\scriptstyle r-k+2} \\
\vdots & \vdots & & \vdots & \vdots \\
r+1 & r+2 & \cdots & r+k-1 & \boxed{\scriptstyle 
r}
\end{smallmatrix}}\, , & r\geq k-1.
\end{array}
\right.
\end{equation*}
We can remove the boxes in the tableau (in
either case), since every other ordering of the
boxed elements evaluates to zero. Hence,
\begingroup
\addtolength{\jot}{1em}
\begin{align*}
\begin{split}
\Delta_{\{r+1,\ldots ,r+k\},J}\left(\vphantom{{2^2}^2}\varphi(p)\right)
&=
\left\{
\def\arraystretch{3}
\begin{array}{cc}
\displaystyle\frac{(-1)^{(k-1)(k-r)+(k-1)r} \minor{J} \prod_{s=r+1}^{r+k-1}\mcoeff{s}}{\prod_{s=r+1}^{r+k}\mcoeff{s}}
, & r<k-1; \\
\displaystyle \frac{(-1)^{k(k-1)} \minor{J} \prod_{s=r+1}^{r+k-1}\mcoeff{s}}{\prod_{s=r+1}^{r+k}\mcoeff{s}}
, & r\geq k-1;
\end{array}
\right.
\\
&= \frac{\minor{J} \prod_{s=r+1}^{r+k-1}\mcoeff{s}}
{\prod_{s=r+1}^{r+k}\mcoeff{s}},
\end{split}
\end{align*}
\endgroup
and the proof of (a) is complete.

The proof of (b) is similar: $\Delta_{\{r,\ldots ,r+k\},J}(\varphi(p))$ can be written in terms of a
$(k+1)\times (k+1)$ tableau in which $k+1$ elements are boxed. Since this is
greater than $k$, it follows that this is zero by~\cite[Prop.\ 1.2.2(i)]{leclerc93} (see Proposition~\ref{p:turnbull}).
\end{proof}

By definition, $\varphi(\Gr_{k,n}^*)\subseteq N$.
In order to check that the image of $\varphi$ is
contained in $N^w$, we use the following result,
which can easily be deduced
from~\cite[4.1]{fominzelevinsky99}.

\begin{prop} \label{p:criterionminus}
Let $x\in SL_n(\mathbb{C})$ and $w\in W$. Then $x\in B_-wB_-$ if and only if the
following hold:
\begin{itemize}
\item[(a)] $\Delta_{\{1,\ldots ,i\},w^{-1}(\{1,\ldots ,i\})}(x)\not=0$ for $i=1,2,\ldots ,n-1$.
\item[(b)] $\Delta_{\{1,\ldots ,i\},w^{-1}(\{1,\ldots ,i-1\}\cup \{j\})}(x)=0$ for all $1\leq i<j\leq n$ satisfying $w^{-1}(i)<w^{-1}(j)$.
\end{itemize}
\end{prop}

\begin{lemma} \label{l:phitarget}
Let $p\in \Gr_{k,n}^*$ and $x=\varphi(p)$. Then $x\in N^w$.
\end{lemma}
\begin{proof}
Let $p\in \Gr_{k,n}^*$ and consider $x=\varphi(p)$. As indicated above, $x\in N$. We must
show that (a) and (b) in Proposition~\ref{p:criterionminus} both hold, so that we can conclude
that $x\in B_-wB_-$ also.
We first claim that, for $1\leq i<j\leq k$, the following hold:
\begin{align}
x_{i,i+n-k}\not=0,\text{ if }1\leq i\leq k; \label{e:nonzero}\\
x_{i,j+n-k}=0,\text{ if }1\leq i<j\leq k. \label{e:zero}
\end{align}
Fix $1\leq i\leq k$. Then we have:
$$x_{i,i+n-k}=\frac{ [\coeff{i}\setminus \{i\}\cup \{i+n-k\}](p)}{\mcoeff{i}(p)}=\frac{\mscoeff{}{i}(p)}{\mcoeff{i}(p)}\not=0,$$
since $p\in \Gr_{k,n}^*$. This shows~\eqref{e:nonzero}.
Fix $1\leq i<j\leq k$. Then we have:
$$x_{i,j+n-k}=\frac{[\coeff{i}\setminus \{i\}\cup \{j+n-k\}](p)}{\mcoeff{i}(p)}=0,$$
since $j+n-k\in \coeff{i}\setminus\{i\}$. This shows~\eqref{e:zero}.

It follows that the submatrix of $x$ with rows $\{1,\ldots ,k\}$ and columns $\{n-k+1,\ldots ,n\}$
has non-zero entries along its diagonal and zero entries above its diagonal.
Hence, for $1\leq i\leq k$, the submatrix of $x$ with rows $\{1\ldots ,i\}$ and columns $w^{-1}(\{1,\ldots ,i\})=\{n-k+1,\ldots ,n-k+i\}$ has the same property and thus non-zero determinant,
showing (a) for $1\leq i\leq k$. And, for $1\leq i<j\leq k$, the submatrix of $x$
with rows $\{1,\ldots ,i\}$ and columns $w^{-1}(\{1,\ldots ,i-1\}\cup \{j\})$ has non-zero entries
along its diagonal, except the bottom right entry with is zero, and zero entries
above the diagonal. Hence it has zero determinant, and (b) is shown for $1\leq i<j\leq k$.

It remains to show (a) for $k+1\leq i\leq n$ and (b) for $k+1\leq i<j\leq n$.

We first fix $k+1\leq i\leq n$ and write $i=k+r$ where $1\leq r\leq n-k$.
We have
\begin{align*}
\begin{split}
\Delta_{\{1,\ldots ,i\},w^{-1}(\{1,\ldots ,i\})}(x) &= \Delta_{\{1,\ldots ,k+r\},\{n-k+1,\ldots ,n+r\}}(x) \\
&=\Delta_{\{1,\ldots ,k+r\},\{1,\ldots ,r\}\cup \{n-k+1,\ldots ,n\}}(x) \\
&=\Delta_{\{r+1,\ldots ,r+k\},\{n-k+1,\ldots ,n\}}(x),
\end{split}
\end{align*}
where in the last step we have used the fact that $x\in N$.

Since $r\leq n-k$, we have by Lemma~\ref{l:phiminor}(a) that
$$\Delta_{\{r+1,\ldots ,r+k\},\{n-k+1,\ldots ,n\}}(x)=\frac{\mcoeff{n}(p)}{\mcoeff{s}(p)},$$
where $s=r+k$. This is non-zero, since $p\in \Gr_{k,n}^*$.

Next we suppose that $i=k+r$ and $j=k+t$, where $1\leq r<t\leq n-k$.
We have that
\begin{align*}
\begin{split}
\Delta_{\{1,\ldots ,i\},w^{-1}(\{1,\ldots ,i-1\}\cup \{j\})}(x) &=
\Delta_{\{1,\ldots ,k+r\},\{n-k+1,\ldots ,n\}\cup \{1,\ldots ,r-1\}\cup \{t\}}(x) \\
&= \Delta_{\{r,\ldots k+r\},\{t\}\cup \{n-k+1,\ldots ,n\}}(x),
\end{split}
\end{align*}
using the fact that $x\in N$.
We have assumed that $r\leq t-1$, so this is zero by Lemma~\ref{l:phiminor}(b).
\end{proof}

Let $G_0=N_-HN$, where $N^-$ denotes the set of
lower unitriangular matrices in $SL_n(\mathbb{C})$.
The \emph{Gaussian decomposition} of an element $x\in G_0$
is
$$x=[x]_- [x]_0 [x]_+,$$
where $[x]_-\in N_-$, $[x]_0\in H$ and $[x]_+\in N$.

We set
\begin{align}
N_+(w) &= N\cap wN_-w^{-1}; \\
N_-(w) &= N_-\cap w^{-1}Nw.
\end{align}
Then $N_+(w)$ coincides with the unipotent radical of the maximal parabolic subgroup $P$ associated
to the complement of the root $\boldsymbol{\alpha}_k$.

By~\cite[Props.\ 2.10, 2.17]{fominzelevinsky99},
for $x\in N^w$, $xw^{-1}\in G_0$ and, moreover, the map
$x\mapsto [xw^{-1}]_+$ defines a biregular isomorphism
$$\beta:N^w\rightarrow N_+(w)\cap G_0w.$$

\begin{theorem} \label{t:inverse}
The maps $q:N^w\rightarrow \Gr_{k,n}^*$ and $\varphi:\Gr_{k,n}^*\rightarrow N^w$ are
mutual inverses (and hence biregular).
\end{theorem}

\begin{proof}
By Lemmas~\ref{l:qtarget} and~\ref{l:phitarget}, 
the image of $q$ is contained in $\Gr_{k,n}^*$ and 
the image of $\varphi$ is contained in $N^w$.
For $x\in N^w$,
$$q(x)=P_-x=P_-xw^{-1}w=P_-[xw^{-1}]_+w=P_-\beta(x)w,$$
using the definition of the Gaussian decomposition.
So if $q(x)=q(y)$, then $P_-\beta(x)w=P_-\beta(y)w$, so $P_-\beta(x)=P_-\beta(y)$.
Since $\beta(x),\beta(y)\in N_+(w)$, it follows that $\beta(x)=\beta(y)$,
as the projection map $q$ restricted to $N_+(w)$
is injective by~\cite[Prop.\ 14.21]{borel91}
(as reformulated in~\cite[\S2.2]{GLS08}),
so $x=y$, since $\beta$ is a bijection.
Therefore $q$ is injective.

By Lemma~\ref{l:phiminor}(a) (in the case $r=0$), the composition $q\varphi$ is equal to the identity on
$\Gr_{k,n}^*$. Hence $q$ is surjective and therefore a bijection. It follows that
$\varphi$ is its inverse and we are done.
\end{proof}

\begin{remark}
It is easy to see that the image of $N^w$ under
$q$ is the same as the image of
$B_-B_+\cap B_-wB_-$ and is thus an open
Richardson variety, part of the stratification
of $P_-\backslash G$ introduced by Lusztig~\cite{lusztig98}.
The natural map $\pi:G\rightarrow B_-\backslash G$ is injective
on restriction to $N$, and hence on restriction
to $N^w$, and the
projection $B_-\backslash G\rightarrow P_-\backslash G$ is injective on $\pi(N^w)$
(see~\cite[7.1]{rietschwilliams08}). This gives
an alternative way to see the injectivity of
$q$ on $N^w$.
\end{remark}

The BFZ-twist for the open positroid variety $\Gr^*_{k,n}$ is defined by 
transporting the BFZ-twist $\eta$ for $N^w$ to
$\Gr^*_{k,n}$ by setting:
$\BFZcheck{p} := \varphi^{-1} \circ \eta \circ \varphi (p) $
for $ p \in \Gr^*_{k,n}$, so that we have the following
commutative diagram:
$$
\xymatrix{
N^w \ar[r]^q \ar[d]_{\eta} & \Gr_{k,n}^* \ar[d]^{p\mapsto \BFZcheck{p}} \\
N^w \ar[r]^q & \Gr_{k,n}^*
}
$$
For a $k$-subset $I$, we denote by $\BFZcheck{\minor{I}}$ the BFZ-twisted Pl\"ucker coordinate defined by the composition $p \mapsto \minor{I}(\BFZcheck{p})$. We recall that
the domain of the
twist $p\mapsto \twist{p}$ contains
$\Gr^*_{k,n}$ (see Remark~\ref{r:coefficientcase}).

\begin{prop} \label{p:relationship}
For any $k$-subset $I$ the twisted Pl\"{u}cker coordinate
and the BFZ-twisted Pl\"{u}cker coordinate are related by the following formula:
$$ \frac{\BFZcheck{\minor{I}}}{\vphantom{2^{2^{2^e}}} \BFZcheck{\mcoeff{k}}}
= \frac{ \twist{\minor{I}} \cdot \mcoeff{k}} {\vphantom{2^{2^{2}}}\prod_{i \in I} \mcoeff{i}}$$
as elements of the coordinate ring $\mathbb{C}[\Gr^*_{k,n}]$.
\end{prop}
\begin{proof}
Firstly, let $\overline{w}$ be the matrix whose entries are $\overline{w}_{i+k,i}=(-1)^k$ for $1\leq i\leq n-k$, $\overline{w}_{i+k-n,i}=1$ for $n-k+1\leq i\leq n$ and zero otherwise. Then $\overline{w}$ is a particular choice of representative for $w$, and we can write our arbitrary choice of representative in the form $w=t\overline{w}$, where $t=\text{diag}(t_1,\ldots ,t_n)\in T$.

Notice that $q(bg) = q(g)$ for any $g \in \text{SL}_n\big( \mathbb{C} \big)$ and any
$b \in B_{-}$.
By definition the twist $\eta(g)$ for $g \in N^w$ can be written in the form $bwg^T$ for some (unique) element $b \in B_{-}$ and therefore $q(\eta(g))$ coincides with
$q(wg^T)$. Hence
$$q\left(\vphantom{{2^2}^2}w\varphi(p)^T\right)=q\left(\vphantom{{2^2}^2}\eta\left(\vphantom{{2^2}^2}\varphi(p)\right)\right)=q\varphi(\BFZcheck{p})=\BFZcheck{p}$$
for any point $p \in \Gr^*_{k,n}$.
It is easy to see that, for
$1\leq i\leq k$, the $(i,j)$-matrix entry of
$w\varphi(p)^T=t\overline{w}\varphi(p)^T$ is
$$t_i\cdot \frac{\minor{ \coeff{j} - \{j\} \cup \sigma^k(i)}(p)}{\mcoeff{j}(p)}. $$
On the other hand, for $1\leq i\leq k$, the $(i,j)$-entry of $\varphi\!\left( \vphantom{A^b}
\raisebox{-1pt}{$\tmat$} \right)$ is
$$ \frac{\twist{ \minor{ \coeff{i} - \{i \} \cup \{j\}}}(p)}{\vphantom{2^{2^{2^2}}}\twist{\mcoeff{i}}(p)}, $$
which after employing Proposition~\ref{p:twistcomputation} and performing
cancellations may be re-expressed as
$$ \frac{ \minor{ \coeff{j} - \{j \} \cup \{\sigma^k(i)\}}(p)} {\vphantom{{{2^2}^2}^a}\minor{\sigma(\coeff{i})}(p)}. $$
So we may deduce that
\begin{equation} \label{e:minoreqn1}
 \Delta_{\{1,\ldots,k\},I} \left( \vphantom{{2^2}^2}\varphi\!\left( \vphantom{A^b} \raisebox{-1pt}{$\tmat$} \right)
\right) \ = \ 
t_1^{-1}\cdots t_k^{-1}\cdot 
\frac{\prod_{i \in I} \mcoeff{i}(p)}{\prod_{j=1}^k \minor{\sigma(\coeff{j})}(p)}
\cdot \Delta_{\{1,\ldots k\},I} \left( \vphantom{{2^2}^2} w\varphi(p)^T \right).
\end{equation}
By the definition of $q$, we have that
\begin{equation}
\label{e:minoreqn2}
\frac{\Delta_{\{1,\ldots k\},I} \left( \vphantom{{2^2}^2} w\varphi(p)^T \right)}
{\vphantom{{{2^2}^2}^2}\Delta_{\{1,\ldots k\},\{1,\ldots k\}}\left( \vphantom{{2^2}^2} w\varphi(p)^T \right)}
\ =\ 
\frac{ \minor{I}\left( \vphantom{{2^2}^2} q\left( \vphantom{{2^2}^2} w\varphi(p)^T\right) \right)}
{\vphantom{{{2^2}^2}^2}\mcoeff{k}\left( \vphantom{{2^2}^2} q \left( \vphantom{{2^2}^2} w\varphi(p)^T\right)\right)}
\ =\ 
\frac{ \BFZcheck{\minor{I}}(p) }{\vphantom{2^{2^{2^2}}}\BFZcheck{\mcoeff{k}}(p)}.
\end{equation}
Using the Cauchy-Binet formula, we have (summing
over all $k$-subsets $J$ of $\{1,\ldots ,n\}$):
\begin{align*}
\begin{split}
\Delta_{\{1,\ldots ,k\},\{1,\ldots k\}}\left(\vphantom{{2^2}^2} w\varphi(p)^T\right) &=
\sum_{J} \Delta_{\{1,\ldots ,k\},J}(w) \Delta_{J,\{1,\ldots ,k\}}\left( \vphantom{{2^2}^2} \varphi(p)^T\right) \\
&= \Delta_{\{1,\ldots ,k\},\{n-k+1,\ldots ,n\}}(w) \Delta_{\{n-k+1,\ldots n\},\{1,\ldots ,k\}}\left( \vphantom{{2^2}^2} \varphi(p)^T \right) \\
&= \Delta_{\{1,\ldots ,k\},\{n-k+1,\ldots ,n\}}(t\overline{w}) \Delta_{\{n-k+1,\ldots n\},\{1,\ldots ,k\}}\left( \vphantom{{2^2}^2} \varphi(p)^T \right) \\
&= t_1\cdots t_k\cdot \Delta_{\{1,\ldots k\},\{n-k+1,\ldots ,n\}}(\varphi(p)).
\end{split}
\end{align*}
Hence, by Lemma~\ref{l:phiminor}(a),
\begin{equation}
\label{e:minoreqn3}
\Delta_{\{1,\ldots k\},\{1,\ldots k\}}\left(\vphantom{{2^2}^2} w\varphi(p)^T\right)\ =\ t_1\cdots t_k \cdot
\frac{\mcoeff{n}(p)}{\vphantom{{2^2}^2}\mcoeff{k}(p)}.
\end{equation}

Combining equations~\eqref{e:minoreqn1},~\eqref{e:minoreqn2} and~\eqref{e:minoreqn3},
we may conclude that:
$$ \Delta_{\{1,\ldots ,k\},I} \left( \vphantom{{2^2}^2} \varphi\!\left( \vphantom{A^b} \raisebox{-1pt}{$\tmat$} \right)
 \right) \ = \
\frac{\prod_{i \in I} \mcoeff{i}(p)}{\vphantom{{{2^2}^2}^2} \prod_{j=1}^k \minor{\sigma(\coeff{j})}(p)}
\cdot \frac{ \BFZcheck{\minor{I}}(p) }{\vphantom{2^{2^{2^2}}}\BFZcheck{\mcoeff{k}}(p)}\cdot \frac{\mcoeff{n}(p)}{\vphantom{{{2^2}^2}^2}\mcoeff{k}(p)}.$$
Since we have:
$$\Delta_{\{1,\ldots ,k\},I} \left( \vphantom{{2^2}^2} \varphi\!\left( \vphantom{A^b} \raisebox{-1pt}{$\tmat$} \right)
 \right) = \frac{\twist{\minor{I}}(p)}{\vphantom{2^{2^{2^2}}}\twist{\mcoeff{k}}(p)}\quad \text{and} \quad 
\frac{\prod_{j=1}^{k-1} \mcoeff{j}(p)}{\vphantom{{{2^2}^2}^2} \twist{\mcoeff{k}}(p)} = 1,$$
we may perform cancellations and obtain the asserted formula.
\end{proof}

\section{Cluster structure of the Grassmannian}
\label{s:clusterstructure}
In this section we recall the description of the
cluster structure on the Grassmannian from~\cite{scott06}.
We firstly recall the definition of a skew-symmetric cluster algebra of geometric type.
Fix $l,m\in \mathbb{N}$ and let $\mathbb{F}$ denote the field of rational functions in indeterminates $u_1,\ldots ,u_l,\ldots ,u_{l+m}$ over $\mathbb{Q}$.
We consider \emph{seeds} $(\xx,\Q)$ consisting of a free generating set $\xx=\{x_1,x_2,\ldots ,x_{l+m}\}$ (known as an
\emph{extended cluster}) of $\mathbb{F}$ over $\mathbb{Q}$ and a quiver $\Q$ on
vertices $1,2,\ldots ,l+m$ (known as the exchange quiver), with no arrows between vertices labelled $l+1,l+2,\ldots ,l+m$ (the coefficients) and no two-cycles or loops. The tuple $\mathbf{x}=(x_1,x_2,\ldots ,x_l)$ is called a \emph{cluster},
with the remaining elements of $\xx$ known as \emph{coefficients}. The subquiver $Q$ of $\Q$ on vertices $1,2,\ldots ,l$ is called the \emph{principal part} of $\Q$. Vertices in $\widetilde{Q}$ but not
$Q$ are referred to as \emph{frozen vertices}.

Fix $r\in \{1,\ldots ,l\}$. The \emph{mutation} of $(\xx,\Q)$ at $r$ is the pair $(\xx',\Q')$, where $$\xx'=(x_1,\ldots ,x_{r-1},x'_r,x_{r+1},\ldots ,x_{l+m}),$$
with $x'_r$ defined by the
\emph{exchange relation}:
$$x_rx'_r=\prod_{j\rightarrow r} x_j+\prod_{r\rightarrow j} x_j,$$
with the first product taken over all arrows in $\Q$ ending at $r$, and the
second over all arrows in $\Q$ starting at $r$.
The quiver $\Q'$ is obtained from $\Q$ via quiver mutation. Firstly, an arrow $i\rightarrow j$ is added to $\Q$ for every path of length two from $i$ to $j$ passing through $r$, then the
arrows incident with $r$ are reversed, and finally a maximal collection of two-cycles is removed.

The \emph{cluster algebra} associated to $(\xx,\Q)$ is the $\mathbb{C}$-subalgebra of $\mathbb{F}$ generated by the free generating sets occurring in the seeds
obtained from $(\xx,\Q)$ by iterated mutation.

In~\cite{scott06} (see also~\cite{GSV03,GSV10}) it is shown that $\mathbb{C}[\Gr_{k,n}]$
is a cluster algebra, using certain diagrams in a disk known as Postnikov
diagrams~\cite{postnikov} (in~\cite{postnikov} they are referred to as alternating strand diagrams), which we now recall.

\begin{defn} \label{d:arrangement}
Label the vertices of a convex $2n$-sided polygon $1',1,2',2,\ldots ,n',n$
clockwise around its boundary. A \emph{Postnikov diagram} consists of $n$ oriented strands inside the polygon, satisfying the conditions:
\begin{enumerate}[(a)]
\item No strand intersects itself.
\item All intersections of strands are transversal.
\item There are finitely many such intersections.
\item Strand $i$ starts at $i$ and ends at $(i+k)'$.
\item Travelling along strand $i$, the directions of the strands crossing it alternate, starting with left to right and ending with right to left.
\item No two strands, regardless of the rest of diagram, are allowed to form an unoriented lens, as depicted in Figure~\ref{fig:lens}.
\end{enumerate}

\begin{figure}
\includegraphics[width=4cm]{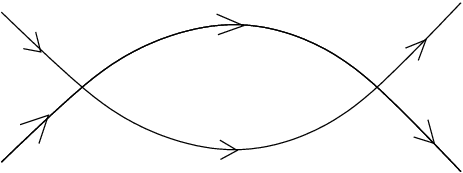}
\caption{Forbidden configuration}
\label{fig:lens}
\end{figure}

Postnikov diagrams are considered up to isotopy and the creation/annihilation
of a local oriented lens, as in Figure~\ref{fig:locallens} (or with the orientation in this figure reversed). Note that we specify the conditions on the boundary precisely.

\begin{figure}
\includegraphics[width=10cm]{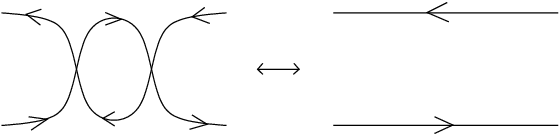}
\caption{Creation or annihilation of a local oriented lens}
\label{fig:locallens}
\end{figure}
\end{defn}

For an example of a Postnikov diagram, see Figure~\ref{fig:postnikovexample}.

\begin{figure}
\begin{center}
\includegraphics[width=8cm]{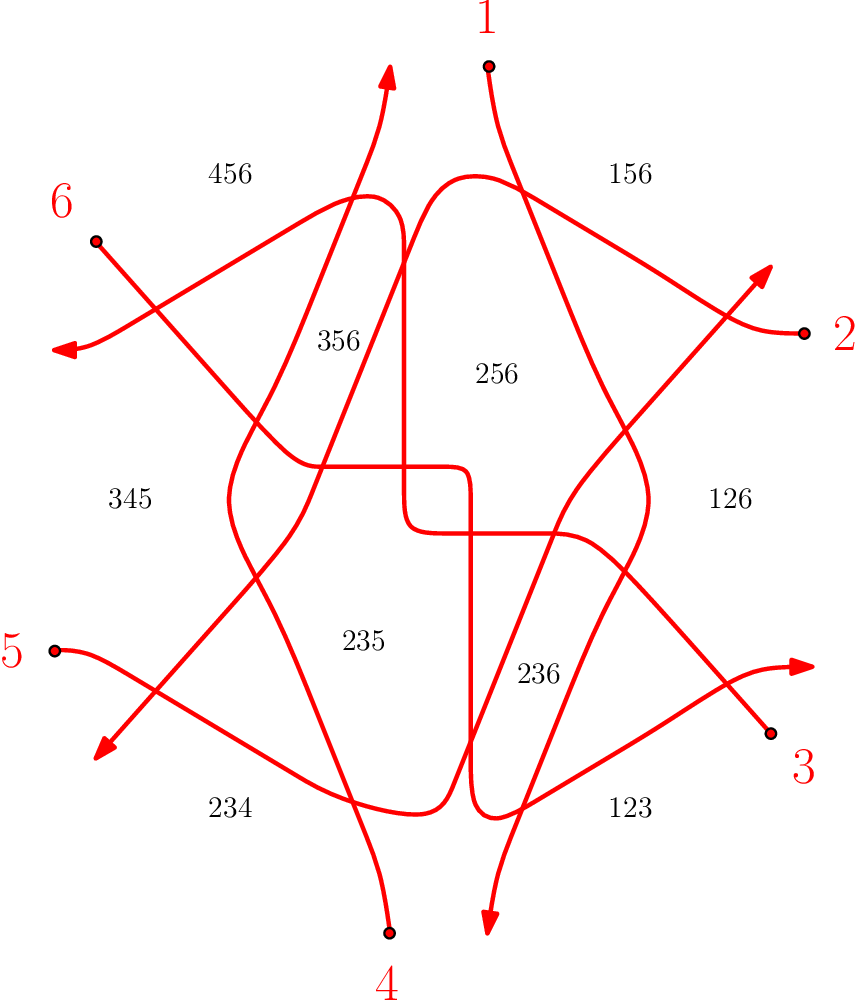} 
\end{center}
\caption{A Postnikov diagram (in the case $k=3,n=6$).}
\label{fig:postnikovexample}
\end{figure}

We orient the boundary of the polygon clockwise. Then
each face (i.e.\ connected component of the complement of the strands inside the polygon) of a Postnikov diagram either has an alternating boundary
or an oriented boundary. We refer to the former as \emph{alternating faces} and the latter as \emph{oriented faces}. Each alternating face $F$ is labelled with the subset $I_F$ of $\{1,\ldots ,n\}$ consisting of the strands which have $F$ on their left (as they are traversed).

\begin{prop} (Postnikov; see~\cite[Prop.\ 5]{scott06})
Let ${\postdiag}$ be a Postnikov diagram. Then we have:
\begin{enumerate}
\item For any alternating face $F$ of ${\postdiag}$, the cardinality of its labelling
set $I_F$ is exactly $k$.
\item There are exactly $k(n-k)+1$ alternating faces in ${\postdiag}$, with $n$ of them situated along the boundary. The remaining $(k-1)(n-k-1)$ alternating faces are
internal.
\item The labels of the alternating boundary faces are the cyclic intervals $\mcoeff{1},\ldots ,\mcoeff{n}$.
\end{enumerate}
\end{prop}

\begin{remark} \label{r:coefficientlocation}
The label on the alternating boundary face immediately clockwise of $i$ is $\coeff{i}$.
\end{remark}

\begin{remark}
Note that the set of labels on a Postnikov diagram is not changed by the creation
or annihilation of oriented lenses.
\end{remark}

\begin{defn} \label{d:quadrilateralmove}
The \emph{quadrilateral move} on a Postnikov diagram (called geometric exchange in~\cite{scott06}) is the local move depicted in Figure~\ref{fig:quadrilateralmove}.
\end{defn}

\begin{figure}
\includegraphics[width=8cm]{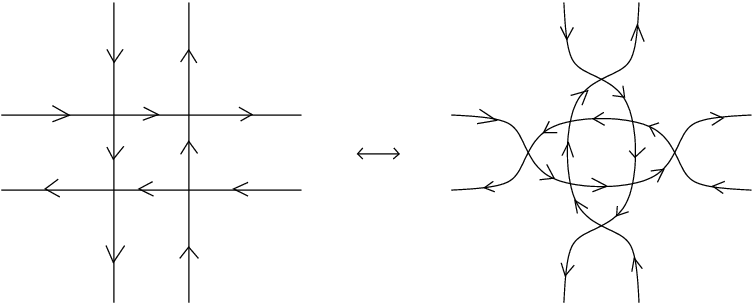}
\caption{The quadrilateral move}
\label{fig:quadrilateralmove}
\end{figure}

We recall the following result of Postnikov, referred to
in~\cite[Prop. 6]{scott06}.

\begin{prop} \label{p:connected}
Any two Postnikov diagrams are connected by a sequence of quadrilateral moves.
\end{prop}

Scott~\cite[Sect.\ 5]{scott06}, associates a quiver $\Q({\postdiag})$ to a Postnikov diagram ${\postdiag}$. The vertices correspond to the alternating faces and the arrows are as shown in Figure~\ref{fig:quiver} (with the thick arrow indicating the arrow in the quiver), cancelling any two-cycles. Note that cancelling two-cycles
corresponds to applying the maximal annihilation of oriented lenses as in Figure~\ref{fig:locallens} (from left to right).

\begin{figure}
\includegraphics[width=3cm]{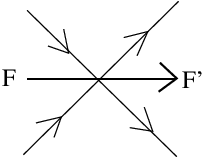}
\caption{Arrow in the quiver of a Postnikov diagram}
\label{fig:quiver}
\end{figure}

Figure~\ref{fig:postnikovquiver} shows an example of the quiver of a Postnikov
diagram.

\begin{figure}
\begin{center}
\includegraphics[width=15cm]{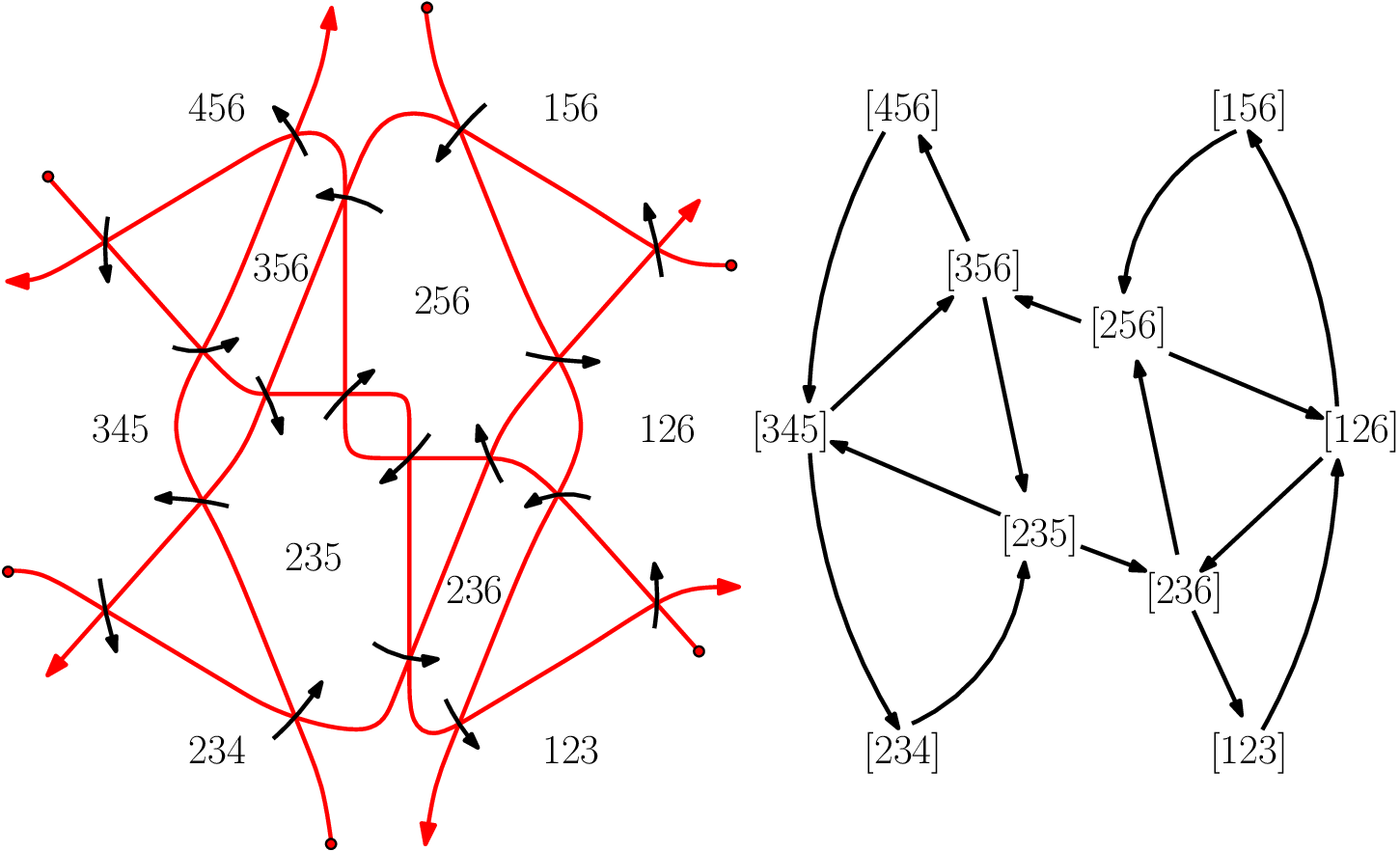}
\end{center}
\caption{The quiver associated to the Postnikov diagram in Figure~\ref{fig:postnikovexample}.}
\label{fig:postnikovquiver}
\end{figure}

\begin{theorem} \cite{scott06} \label{t:scott}
Let ${\postdiag}$ be a Postnikov diagram.
Let
$$\xx({\postdiag})=\{\minor{I_F}\,:\,F\text{ an alternating face of ${\postdiag}$}\},$$
$$\mathbf{x}({\postdiag})=\{\minor{I_F}\,:\,F \text{ an internal alternating face of ${\postdiag}$}\}.$$
Then:
\begin{enumerate}[(a)]
\item The pair $(\xx({\postdiag}),\Q(\postdiag))$ is a seed in the rational function field
$\mathbb{C}(\Gr_{k,n})$, with coefficients $\mcoeff{1},\ldots ,\mcoeff{n}$.
\item If ${\postdiag},{\postdiag}'$ are related by a single quadrilateral move, then $(\xx({\postdiag}'),\Q(\postdiag'))$ can be obtained from $(\xx({\postdiag}),\Q(\postdiag))$ by a single mutation. The exchange relation in this case is a short Pl\"{u}cker relation (see Remark~\ref{r:shortPluecker}, below).
\item The cluster algebra determined by $(\xx({\postdiag}),\Q(\postdiag))$ (for any ${\postdiag}$) is
$\mathbb{C}[\Gr_{k,n}]\subseteq \mathbb{C}(\Gr_{k,n})$.
\end{enumerate}
\end{theorem}

\begin{remark} \label{r:shortPluecker}
If $F$ is an alternating internal face about which a quadrilateral move can be
performed, then there are four indices $a<b<c<d$ in $\{1,\ldots ,n\}$ and a
$(k-2)$-subset $J$ with $J\cap \{a,b,c,d\}=\phi$ such that the exchange relation
corresponding to the quadrilateral move is the short Pl\"{u}cker relation:
$$\minor{Jac}\minor{Jbd}=\minor{Jab}\minor{Jcd}+\minor{Jad}\minor{Jbc},$$
where $Jac=J\cup \{a,c\}$, etc.
\end{remark}

\section{Dimer partition functions}
\label{s:dimer}
In this section, we consider the dual bipartite graph
of a Postnikov diagram. We give a weighting to the edges
of this graph and define a dimer partition function,
for each $k$-subset $I$ of $\{1,\ldots ,n\}$, in
terms of perfect matchings of a subgraph.
We show that this function is invariant under blow-up or
blow-down moves. We then show that a scaled version 
of each dimer partition function is, in addition,
independent of the choice of Postnikov diagram ${\postdiag}$. We then use this to associate a polynomial to each $k$-subset $I$.

Given a Postnikov diagram, ${\postdiag}$, a dual graph $G=G_{\postdiag}$ is defined as follows. The vertices of $G_{\postdiag}$ are in bijection with the oriented faces of ${\postdiag}$, and the edges correspond to points of intersection of the boundaries of the corresponding faces. The internal alternating faces of ${\postdiag}$ correspond to
internal faces of $G_{\postdiag}$, and the latter inherits a
$k$-subset label from $\postdiag$, which we replace with the associated minor.
For each $i\in \{1,\ldots ,n\}$, the boundary vertices $i$ and $i'$
of $\postdiag$ lie on the boundary of an oriented face of ${\postdiag}$
that corresponds to a boundary vertex of $G_{\postdiag}$, which
we label with $i$. Thus the graph $G_{\postdiag}$ lies inside a polygon with vertices $1,2,\ldots ,n$.
Furthermore, for each $i$, the boundary vertices $i$ and $(i+1)'$
of $\postdiag$ lie on the boundary of an alternating face of
$\postdiag$ which corresponds to part of the boundary of
$G_{\postdiag}$. We regard the part of the plane bounded by this
and the side of the above polygon between vertices $i$ and $i+1$ as a boundary face of $G_{\postdiag}$. Thus $G_{\postdiag}$ has
$n$ boundary faces in total.

The graph $G_{\postdiag}$ is bipartite: a vertex corresponding to a clockwise face (respectively, anticlockwise face) of ${\postdiag}$ is coloured black (respectively, white). The graph $G_{\postdiag}$ is considered up to local moves, shown in Figure~\ref{fig:blowdown}, corresponding to the creation or annihilation of
oriented lenses as in Figure~\ref{fig:locallens} (we also allow blow-ups and blow-downs for black vertices).
A blow-up at a vertex $v$ involves a partition of the edges incident with the vertex into two subsets; we restrict to the case where both subsets are nonempty, unless
$v$ is a boundary black vertex (see Figure~\ref{fig:blowdown} for an example).

We refer to the move corresponding to
the annihilation of an oriented lens as a \emph{blow-down}, and the move corresponding
to the creation of an oriented lens as a \emph{blow-up}.

\begin{figure}
\includegraphics[width=8cm]{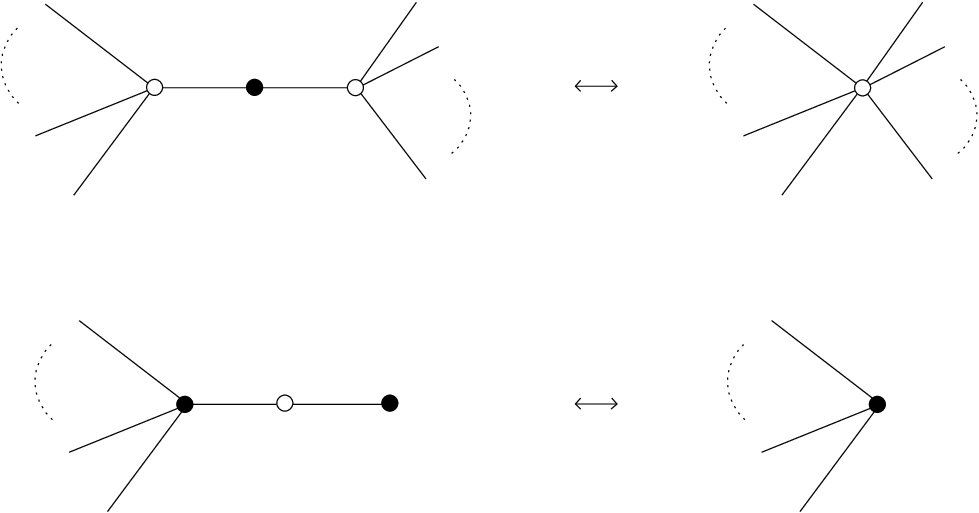}
\caption{The blow-down/blow-up move on the bipartite graph. In the lower figure, the black vertices on the right hand
side of the diagrams must be boundary
vertices.}
\label{fig:blowdown}
\end{figure}

For an example of the dual bipartite graph of the Postnikov diagram in Figure~\ref{fig:postnikovexample}, see Figure~\ref{fig:postnikovdual}. The
labels of the faces are shown in Figure~\ref{fig:postnikovdual2}.

\begin{figure}
\begin{center}
\includegraphics[width=10cm]{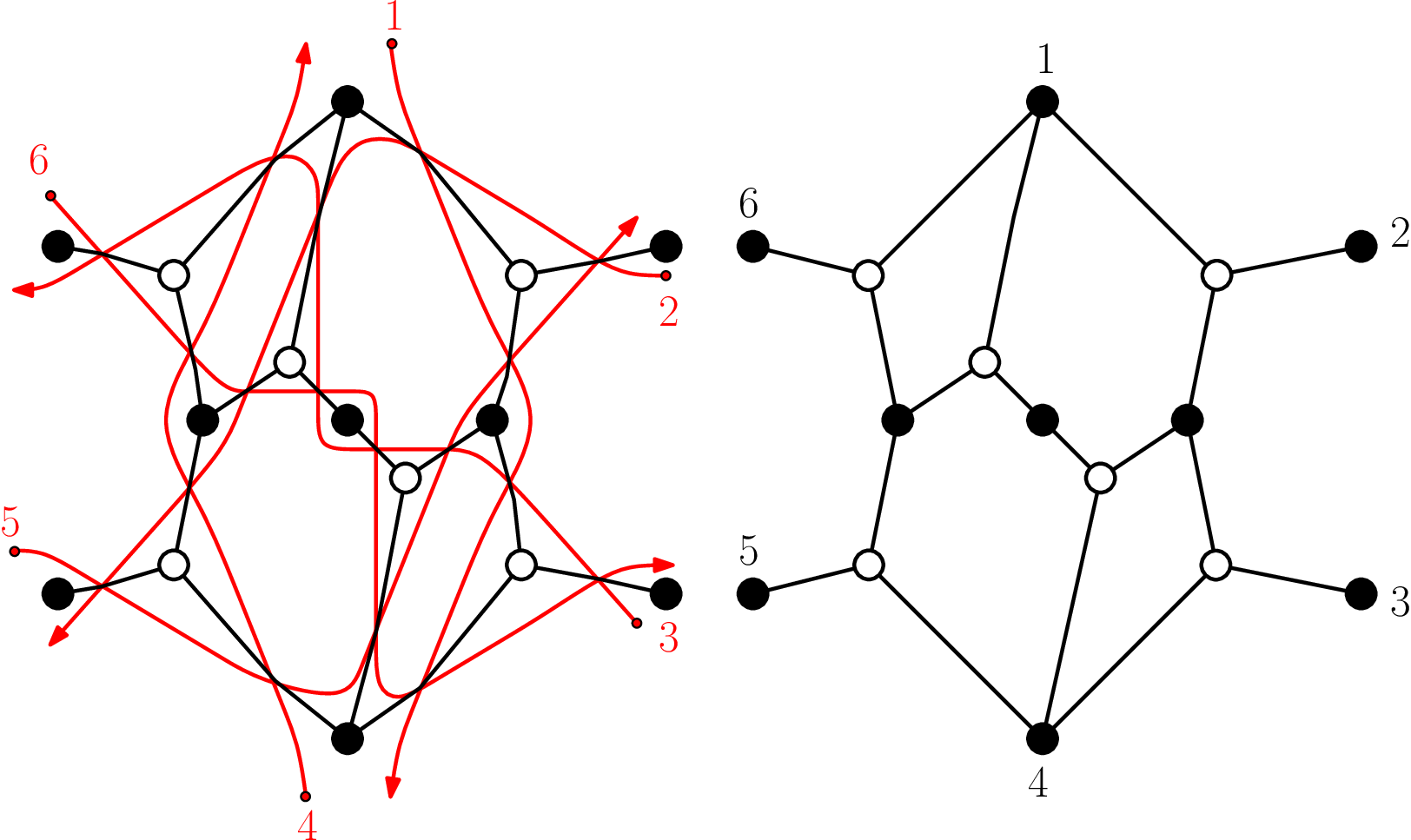}
\end{center}
\caption{The dual bipartite graph associated to the Postnikov diagram in Figure~\ref{fig:postnikovexample}.}
\label{fig:postnikovdual}
\end{figure}

\begin{figure}
\begin{center}
\includegraphics[width=6cm]{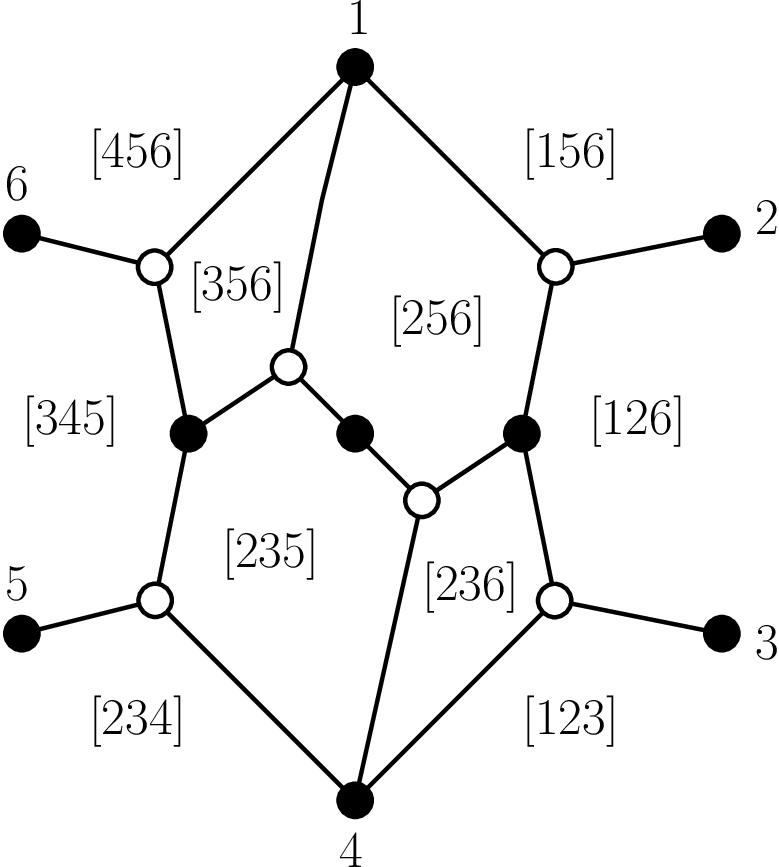}
\end{center}
\caption{The face labels on dual bipartite graph associated to the Postnikov diagram in Figure~\ref{fig:postnikovexample}.}
\label{fig:postnikovdual2}
\end{figure}

\begin{defn} \label{d:weights}
We give weights to the edges of $G_{\postdiag}$ as follows.
Let $e$ be an edge of $G_P$. Then we label
$e$ with the product $\weight_e$ of the Pl\"{u}cker coordinates labelling the faces of $G_P$ which are incident with the white vertex incident with $e$ but not with $e$ itself. See Figure~\ref{fig:weighting}.
\end{defn}

\begin{figure}
\psfragscanon
\psfrag{I1}{$\minor{I_1}$}
\psfrag{I2}{$\minor{I_2}$}
\psfrag{Id}{$\minor{I_d}$}
\psfrag{we}{$\weight_e$}
\includegraphics[width=4cm]{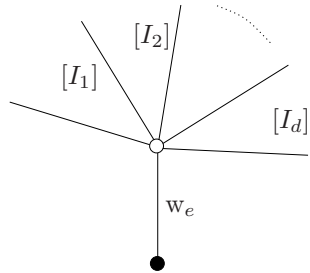}
\caption{Weighting of an edge in the bipartite graph: $\weight_e=\minor{I_1}\cdot \minor{I_2}\cdots \minor{I_d}$.}
\label{fig:weighting}
\end{figure}

\begin{remark}
Given an arbitary edge weighting $(\kappa_e)$ of
$G_P$ define the shear weight $y_F$ of a face $F$ of $G_P$ to be the following ratio of edge weights:
$$y_F=\prod_{e\in \partial F} m_e^{\varepsilon_F(e)},$$
where $\varepsilon_F(e)$ is $+1$ if the
orientation of $e$ induced by the
counter-clockwise orientation of $\partial F$
points from black to white, and is $-1$
otherwise. In the case where $\kappa_e=w_e$ for all edges $e$, it is
easy to check that the shear weights coincide
with the $y$-variables (as defined in~\cite{fominzelevinsky07})
$$y_F=\prod_{\stackrel{\scriptstyle F'\rightarrow F}{\vphantom{A^{2^b}} \text{ in }\widetilde{Q}(P)}} [I_F] \ \cdot
\prod_{\stackrel{\scriptstyle F\rightarrow F'}{\vphantom{A^{2^b}} \text{ in }\widetilde{Q}(P)}} [I_F]^{-1}.
$$
\end{remark}

Let ${\postdiag}$ be a Postnikov diagram with dual graph $G_{\postdiag}$.
We write $G_{\postdiag}(I)$ for the graph $G$ with the vertices labelled
by elements of $I$ removed from the boundary. We recall the following (as
discussed in the introduction).

The \emph{dimer partition function} of $G_{\postdiag}(I)$ is the sum:
\begin{equation*}
\dimerpart_{G_{\postdiag}(I)}=
\sum_{\dimerconfig} \weight_{\dimerconfig}=\sum_{{\dimerconfig}} \prod_{e\in
{\dimerconfig}} \weight_e,
\label{e:dimerpartitionfunction}
\end{equation*}
where ${\dimerconfig}$ varies over all dimer configurations of $G_{\postdiag}(I)$.
Since $G_{\postdiag}(I)$ is only defined up to blow-up and blow-down moves,
we need to check that the dimer partition function is invariant under these moves.

\begin{lemma} \label{l:blowdown}
Let $P$ be a Postnikov diagram and $I$ a $k$-subset of $\{1,\ldots ,n\}$.
Suppose that $G'$ is a graph obtained from $G=G_P(I)$ by applying the blow-down at the top Figure~\ref{fig:blowdown}.
Then a dimer configuration $\dimerconfig$ on $G$ induces a dimer
configuration $\dimerconfig'$ on $G'$ by removing the edge $e$ incident with the
black vertex in the middle of the diagram. Then the weights of $\dimerconfig$
and $\dimerconfig'$ coincide. A similar result holds if the vertex in
the middle is white.
\end{lemma}
\begin{proof}
Let $f$ denote the unique edge incident with one of the white vertices in the part of $G$ shown, and let $f'$ be the corresponding edge in $G'$. Then it is easy to see that
$\weight_e\weight_f=\weight_{f'}$ (see Figure~\ref{fig:blowdownweights}), and it follows that
$\weight_{\dimerconfig}=\weight_{\dimerconfig'}$. A similar argument applies in the case where the
vertex in the middle is white.
\end{proof}

\begin{figure}
\psfragscanon
\psfrag{[I1]}{$[I_1]$}
\psfrag{[I1]e}{$\scriptstyle \mathbf{[I_1]}$}
\psfrag{[I2]}{$[I_2]$}
\psfrag{[I3]}{$[I_3]$}
\psfrag{[I1][I2][I3]e}{$\scriptstyle \mathbf{[I_1][I_2][I_3]}$}
\psfrag{[I2][I3]e}{$\scriptstyle \mathbf{[I_2][I_3]}$}
\begin{center}
\includegraphics[width=15cm]{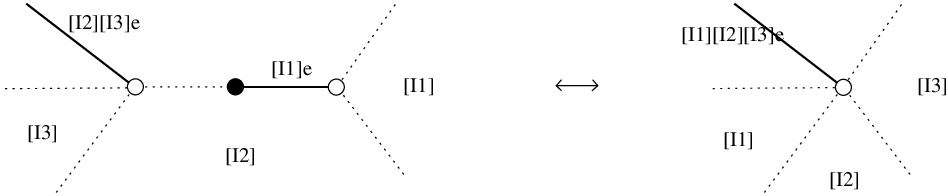}
\end{center}
\caption{Proof of Lemma~\ref{l:blowdown}.}
\label{fig:blowdownweights}
\end{figure}

As a corollary, we see that the dimer partition function $\dimerpart_{G_{\postdiag}}(I)$ is invariant under blow-ups and blow-downs:

\begin{corollary} \label{c:blowdown}
Let ${\postdiag}$ be a Postnikov diagram and $I$ a $k$-subset of $\{1,\ldots ,n\}$.
Then $\dimerpart_{G_{\postdiag}}(I)$ is invariant under a blow-up or blow-down move
applied to $G_P$ (see Figure~\ref{fig:blowdown}).
\end{corollary}
\begin{proof}
This follows from Lemma~\ref{l:blowdown}, noting that the map
$\dimerconfig\mapsto \dimerconfig'$ defined there is a bijection between the set of dimer configurations on $G$ and the set of dimer configurations on $G'$.
\end{proof}

\begin{defn}
Given a Postnikov diagram $P$ and a $k$-subset $I$ of $\{1,\ldots ,n\}$,
we set $$\uu_{G_{\postdiag}}(I):=\frac{\dimerpart_{G_{\postdiag}}(I)}{\prod_{x\in \mathbf{x} ({\postdiag})} x}.$$
Recall that $\mathbf{x}({\postdiag})$, the set of variables occuring in the denominator, is the set of cluster
variables corresponding to the internal alternating
faces of $\postdiag$ (see Theorem~\ref{t:scott}).
\end{defn}

By Corollary~\ref{c:blowdown}, $\uu_{G_{\postdiag}}(I)$ is also invariant under
blow-up and blow-down moves.
We shall next show that the scaled dimer partition
function $\uu_{G_{\postdiag}}(I)$ is invariant under the 
quadrilateral move (see Definition~\ref{d:quadrilateralmove}) 
as well.

Note that the effect of the quadrilateral move on the associated bipartite graph is sometimes known as \emph{urban renewal}; see Figure~\ref{fig:flipbipartite}
(see~\cite{ciucu98,KPW00}; the latter reference also mentioning G. Kuperberg). The proof of Proposition~\ref{prop:flipmoveinvariance} below uses the condensation principle discussed in~\cite{ciucu98,kuo04}.

\begin{prop} \label{prop:flipmoveinvariance}
Let ${\postdiag}$ be a Postnikov diagram, and $I$ a $k$-subset of $\{1,\ldots ,n\}$.
Suppose that ${\postdiag}'$ is obtained from ${\postdiag}$ by applying a quadrilateral move to ${\postdiag}$. Then $\uu_{G_{\postdiag}}(I)=\uu_{G_{\postdiag'}}(I)$.
\end{prop}

\begin{proof}
Suppose that ${\postdiag}'$ is obtained from ${\postdiag}$ by applying a quadrilateral move. Then $G'=G_{\postdiag'}(I)$ is obtained from $G=G_{\postdiag}(I)$ by applying
the urban renewal move dual to a quadrilateral move, as illustrated in Figure~\ref{fig:flipbipartite}.

\begin{figure}
\psfragscanon
\psfrag{AnA}{$\scriptstyle A_{a+1}$}
\psfrag{BnB}{$\scriptstyle B_{b+1}$}
\psfrag{CnC}{$\scriptstyle C_{c+1}$}
\psfrag{DnD}{$\scriptstyle D_{d+1}$}
\psfrag{P1}{$\scriptstyle [P_1]$}
\psfrag{Q1}{$\scriptstyle [Q_1]$}
\psfrag{R1}{$\scriptstyle [R_1]$}
\psfrag{S1}{$\scriptstyle [S_1]$}
\psfrag{Pn}{$\scriptstyle [P_{a}]$}
\psfrag{Qn}{$\scriptstyle [Q_{b}]$}
\psfrag{Rn}{$\scriptstyle [R_{c}]$}
\psfrag{Sn}{$\scriptstyle [S_{d}]$}
\psfrag{I1}{$\scriptstyle [I_1]$}
\psfrag{I2}{$\scriptstyle [I_2]$}
\psfrag{I3}{$\scriptstyle [I_3]$}
\psfrag{I4}{$\scriptstyle [I_4]$}
\psfrag{J}{$\scriptstyle [J]$}
\psfrag{A}{$\scriptstyle A$}
\psfrag{B}{$\scriptstyle B$}
\psfrag{C}{$\scriptstyle C$}
\psfrag{D}{$\scriptstyle D$}
\psfrag{J'}{$\scriptstyle [J']$}
\psfrag{A'}{$\scriptstyle A'$}
\psfrag{B'}{$\scriptstyle B'$}
\psfrag{C'}{$\scriptstyle C'$}
\psfrag{D'}{$\scriptstyle D'$}
\psfrag{A1}{$\scriptstyle A_1$}
\psfrag{A2}{$\scriptstyle A_2$}
\psfrag{B1}{$\scriptstyle B_1$}
\psfrag{B2}{$\scriptstyle B_2$}
\psfrag{C1}{$\scriptstyle C_1$}
\psfrag{C2}{$\scriptstyle C_2$}
\psfrag{D1}{$\scriptstyle D_1$}
\psfrag{D2}{$\scriptstyle D_2$}
\includegraphics[width=15cm]{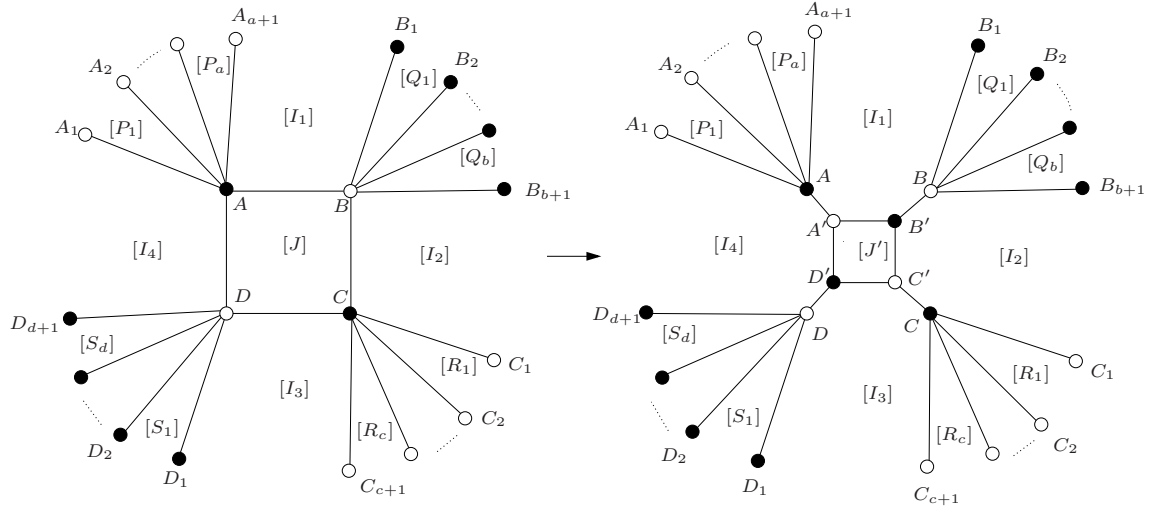}
\caption{The effect of a quadrilateral move on the associated bipartite graph.}
\label{fig:flipbipartite}
\end{figure}

For a dimer configuration ${\dimerconfig}$ of $G$, we shall write, as above, $\weight_{\dimerconfig}$ for the weight of $\dimerconfig$, and similarly $\weight'_{{\dimerconfig}'}$ for a dimer configuration ${\dimerconfig}'$ of $G'$.

Let $G_0$ be the full subgraph of $G$ on vertices $A,B,C,D$
and let $G'_0$ be the full subgraph of $G'$ on vertices $A,B,C,D,A',B',C',D'$.
For a dimer configuration ${\dimerconfig}$ of $G$, let ${\dimerconfig}_0$ be its restriction to $G_0$, and
for a dimer configuration ${\dimerconfig}'$ of $G'$, let ${\dimerconfig}'_0$ be its restriction to $G'_0$.
Since the edges in $G\setminus G_0$ can be identified with those of $G'\setminus G'_0$, a dimer configuration on $G\setminus G_0$ can be regarded as a dimer
configuration on $G'\setminus G'_0$.

In order to carry out the proof, we temporarily define an equivalence relation on dimer configurations of $G$ by stating that two dimer configurations in which the edges $AB,CD$ are replaced with $AD,BC$ are equivalent (i.e. the difference
between the dimer configurations, when regarded as elements in the $\mathbb{Z}_2$-vector space on the edges of $G$,
is the cycle ABCD in the graph $G$).

Similarly, on ${\dimerconfig}'$, we stipulate that two dimer configurations are equivalent when the edges $A'B',C'D'$
are replaced by $A'D',B'C'$. The equivalence class of a dimer configuration
${\dimerconfig}$ on $G$ (respectively, ${\dimerconfig}'$ on $G'$) will be denoted $\overline{{\dimerconfig}}$
(respectively, $\overline{{\dimerconfig}'}$).

We define a map $\varphi$ from equivalence classes of dimer configurations on $G$ to equivalence classes of dimer configurations
on $G'$ as follows. Let $\overline{{\dimerconfig}}$ be an equivalence class, where ${\dimerconfig}$ is a
dimer configuration on $G$.

\begin{enumerate}[(a)]
\item
If ${\dimerconfig}_0=\{AB,CD\}$ or $\{AD,BC\}$, let ${\dimerconfig}'$ be ${\dimerconfig}$ with ${\dimerconfig}_0$ replaced by
${\dimerconfig}'_0=\{AA',BB',CC',DD'\}$ and set $\varphi(\overline{\dimerconfig})=\overline{\dimerconfig}'$.
\item
If ${\dimerconfig}_0$ consists of a single edge $XY$ where $X,Y\in \{A,B,C,D\}$, then
let ${\dimerconfig}'$ be ${\dimerconfig}$ with ${\dimerconfig}_0$ replaced by the edge between the vertices in
$\{A',B',C',D'\}\setminus \{X',Y'\}$, together with the two edges $XX'$ and $YY'$. Set $\varphi(\overline{\dimerconfig})=\overline{\dimerconfig}'$.
\item
If ${\dimerconfig}_0$ is empty, then set $\varphi(\overline{\dimerconfig})=\overline{\dimerconfig}'$ where ${\dimerconfig}'_0=\{A'B',C'D'\}$.
Note that the other element of $\overline{\dimerconfig}'$ contains $\{A'D',B'C'\}$ instead.
\end{enumerate}

Then the following holds:

\textbf{Claim}:
Let $C$ be an equivalence class of dimer configurations on $G$.
Then
$$\frac{\sum_{{\dimerconfig}\in C} \weight_{\dimerconfig}}
{\prod_{x\in \mathbf{x}({\postdiag})}x}
=\frac{\sum_{{\dimerconfig}'\in \varphi(C)} \weight'_{{\dimerconfig}'}}{\prod_{x\in \mathbf{x}({\postdiag}')}x}.$$

\textbf{Proof of claim}:
By the definition of $\varphi$, it is enough to consider the contribution
to $\weight_{\dimerconfig}$ from edges incident with vertices in $G_0$ (respectively, $G'_0$).
Also, since $\mathbf{x}({\postdiag})$ and $\mathbf{x}({\postdiag}')$ coincide
apart from $[J]$ and $[J']$, it is enough to replace the denominator on the left (respectively, right)
hand side with $[J]$ (respectively, $[J']$).

\noindent \textbf{Case (a)}: On the left hand side we have:
\begin{align*}
\begin{split}
&\frac{\minor{Q_1}\cdots \minor{Q_b}\minor{I_2}\cdot \minor{S_1}\cdots \minor{S_d}\minor{I_4}+\minor{Q_1}\cdots \minor{Q_b}\minor{I_1}\cdot \minor{S_1}\cdots \minor{S_d}\minor{I_3}}{\minor{J}} \\
&=\frac{\minor{Q_1}\cdots \minor{Q_b}\minor{S_1}\cdots \minor{S_d}(\minor{I_2}\minor{I_4}+\minor{I_1}\minor{I_3})}{\minor{J}} \\
&= \frac{\minor{Q_1}\cdots \minor{Q_b}\minor{S_1}\cdots \minor{S_d}(\minor{J}\minor{J'})}{\minor{J}} \\
&= \minor{Q_1}\cdots \minor{Q_b}\minor{S_1}\cdots \minor{S_d}\minor{J'}.
\end{split}
\end{align*}
On the right hand side we have:
$$\frac{\minor{J'}\minor{Q_1}\cdots \minor{Q_b}\minor{J'}\minor{S_1}\cdots \minor{S_d}}{\minor{J'}} = \minor{Q_1}\cdots \minor{Q_b}\minor{S_1}\cdots \minor{S_d}\minor{J'},$$
as required.

\noindent \textbf{Case (b)}: Suppose that ${\dimerconfig}_0=\{CD\}$ so that ${\dimerconfig}'_0=\{A'B',CC',DD'\}$.
Assume that the edge in ${\dimerconfig}$ (and thus also in ${\dimerconfig}'$) incident with $A$ is
$e_A$ and the edge in ${\dimerconfig}$ incident with $B$ is $e_B$.
On the left hand side we have:
$$\frac{\weight_{e_A}\weight_{e_B}\weight_{CD}}{\minor{J}} = \frac{\weight_{e_A}\weight_{e_B}\minor{S_1}\cdots \minor{S_d}\minor{I_4}}{\minor{J}}.$$
Noting that $\weight'_{e_A}=\weight_{e_A}$ and $\weight'_{e_B}=\weight_{e_B}/\minor{J}$, on the right hand
side we obtain:
\begin{align*}
\frac{\weight'_{e_A}\weight'_{e_B}\weight'_{A'B'}\weight'_{CC'}\weight'_{DD'}}{\minor{J'}}
&= \frac{\weight_{e_A}\weight_{e_B}\minor{I_4}\minor{J'}\minor{S_1}\cdots \minor{S_d}}{\minor{J}\minor{J'}}.
\end{align*}
The other possibilities are similar.

\noindent \textbf{Case (c)}:
We suppose that the edges incident with $A,B,C,D$ in ${\dimerconfig}$ (and thus in either
element of $\varphi(\overline{\dimerconfig})$) are $e_A,e_B,e_C$ and $e_D$ respectively.
On the left hand side we have
$$\frac{\weight_{e_A}\weight_{e_B}\weight_{e_C}\weight_{e_D}}{\minor{J}}.$$
Note that for either dimer configuration ${\dimerconfig}'$ in $\varphi(\overline{\dimerconfig})$
we have $\weight'_{e_A}=\weight_{e_A}$, $\weight'_{e_B}=\weight_{e_B}/\minor{J}$, $\weight'_{e_C}=\weight_{e_C}$ and
$\weight'_{e_D}=\weight_{e_D}/\minor{J}$, so on the right hand side we have:
\begin{align*}
\frac{\weight_{e_A}\weight_{e_B}\weight_{e_C}\weight_{e_D}
(\minor{I_4}\minor{I_2}+\minor{I_3}\minor{I_1})}{\minor{J}^2\minor{J'}} &=
\frac{\weight_{e_A}\weight_{e_B}\weight_{e_C}\weight_{e_D}\minor{J}\minor{J'}}{\minor{J}^2\minor{J'}} \\
&=\frac{\weight_{e_A}\weight_{e_B}\weight_{e_C}\weight_{e_D}}{\minor{J}}.
\end{align*}
The claim is proved and the proposition follows.
\end{proof}

We note that a similar correspondence between perfect matchings under mutation is used independently in~\cite[Thm. 4.7]{GK13}.

\begin{corollary}
Let $I$ be a $k$-subset of $\{1,\ldots ,n\}$.
Then $\uu_{G_{\postdiag}}(I)$ does not depend on the choice of Postnikov diagram ${\postdiag}$.
\end{corollary}

\begin{proof}
This follows from Proposition~\ref{prop:flipmoveinvariance} and Proposition~\ref{p:connected}.
\end{proof}

We may therefore write $\uu(I)$ for the polynomial $\uu_{G_{\postdiag}}(I)$ for any choice of Postnikov diagram ${\postdiag}$.

\section{Some regular Postnikov diagrams}
\label{s:regular}

In this section, we introduce a regular Postnikov diagram, $\R_{k,n}$, for which the dual bipartite graph is,
apart from a few extra edges, part of a hexagonal tiling of the plane.
We assume in this section that $k\not=1,n-1$.
Reversing the strands in the corresponding diagram for $\Gr(n-k,n)$ gives
rise to another regular Postnikov diagram for $\Gr_{k,n}$, which we
denote by $\R^*_{n-k,n}$. The $k$-subsets labelling $\R_{k,n}$ are disjoint
unions of two cyclic intervals (or coefficients, which consist of just one
cyclic interval),
and we show that the $k$-subsets labelling $\R^*_{n-k,n}$ also have this
form, and in fact that the corresponding Pl\"{u}cker coordinates are exactly the twists
of the Pl\"{u}cker coordinates corresponding to the labels of $\R_{k,n}$, up to a product
of coefficient Pl\"{u}cker coordinates.

The regular form of $\R^*_{n-k,n}$ means that, given any $k$-subset
$I$ labelling $\R_{k,n}$, there is a unique dimer configuration on the
bipartite graph $G_{\R^*_{n-k,n}}(I)$.
This allows us to compute the corresponding scaled dimer partition function $\uu_{\R^*_{n-k,n}}(I)$ explicitly and thus
to show that for $k$-subsets $I$ labelling $\R_{k,n}$, the twist of $\minor{I}$ coincides with $\uu_{\R^*_{n-k,n}}(I)$.
The main result will then be shown in Section~\ref{s:mainresult},
using the fact that the twists of Pl\"{u}cker coordinates and the scaled dimer partition functions both satisfy the Pl\"{u}cker relations.

Let $\R_{k,n}$ be the diagram defined as follows. We take a tiling
of the plane by regular hexagons and equilateral triangles in which each
hexagon has 6 triangles adjacent to it and each triangle has 3 adjacent
hexagons. The edges in the tiling are assumed to be horizontal or at
an angle of $\pm \pi/3$ to the horizontal.

We consider the subset of the tiling obtained by taking $n-k-1$ rows of
$k-1$ hexagons, with each row above and to the right of the previous row.
We also include all of the triangles in the tiling adjacent to the
hexagons.

The boundary triangles on the left of the diagram are labelled
$T_{k+1},T_{k+2},\ldots ,T_n$ from bottom to top, and the boundary triangles
at the top are labelled $T_1,T_2,\ldots ,T_k$ from left to right.
We label the hexagon which is $i$ hexagons across in the $j$th row from
the bottom of the diagram by $H_{k,n}(i,j)$, for $1\leq i\leq k-1$ and
$1\leq j\leq n-k-1$.
We label the incomplete hexagons along the left of the diagram by
$H_{k,n}(0,j)$, with $H_{k,n}(0,j)$ to the left of
$H_{k,n}(1,j)$.
We label the incomplete hexagons along the top of the diagram by
$H_{k,n}(i,n-k)$, with $H_{k,n}(i,n-k)$ above $H_{k,n}(i,n-k-1)$.

For an example ($k=4,n=9$), see Figure~\ref{fig:R49hexagons}.

\begin{figure}
\psfragscanon
\psfrag{H(0,1)}{$\scriptstyle H_{4,9}(0,1)$}
\psfrag{H(0,2)}{$\scriptstyle H_{4,9}(0,2)$}
\psfrag{H(0,3)}{$\scriptstyle H_{4,9}(0,3)$}
\psfrag{H(0,4)}{$\scriptstyle H_{4,9}(0,4)$}
\psfrag{H(0,5)}{$\scriptstyle H_{4,9}(0,5)$}
\psfrag{H(1,5)}{$\scriptstyle H_{4,9}(1,5)$}
\psfrag{H(2,5)}{$\scriptstyle H_{4,9}(2,5)$}
\psfrag{H(3,5)}{$\scriptstyle H_{4,9}(3,5)$}
\psfrag{H(1,1)}{$\scriptstyle H_{4,9}(1,1)$}
\psfrag{H(2,1)}{$\scriptstyle H_{4,9}(2,1)$}
\psfrag{H(3,1)}{$\scriptstyle H_{4,9}(3,1)$}
\psfrag{H(1,2)}{$\scriptstyle H_{4,9}(1,2)$}
\psfrag{H(2,2)}{$\scriptstyle H_{4,9}(2,2)$}
\psfrag{H(3,2)}{$\scriptstyle H_{4,9}(3,2)$}
\psfrag{H(1,3)}{$\scriptstyle H_{4,9}(1,3)$}
\psfrag{H(2,3)}{$\scriptstyle H_{4,9}(2,3)$}
\psfrag{H(3,3)}{$\scriptstyle H_{4,9}(3,3)$}
\psfrag{H(1,4)}{$\scriptstyle H_{4,9}(1,4)$}
\psfrag{H(2,4)}{$\scriptstyle H_{4,9}(2,4)$}
\psfrag{H(3,4)}{$\scriptstyle H_{4,9}(3,4)$}
\psfrag{T1}{$\scriptstyle T_1$}
\psfrag{T2}{$\scriptstyle T_2$}
\psfrag{T3}{$\scriptstyle T_3$}
\psfrag{T4}{$\scriptstyle T_4$}
\psfrag{T5}{$\scriptstyle T_5$}
\psfrag{T6}{$\scriptstyle T_6$}
\psfrag{T7}{$\scriptstyle T_7$}
\psfrag{T8}{$\scriptstyle T_8$}
\psfrag{T9}{$\scriptstyle T_9$}
\includegraphics[width=9cm]{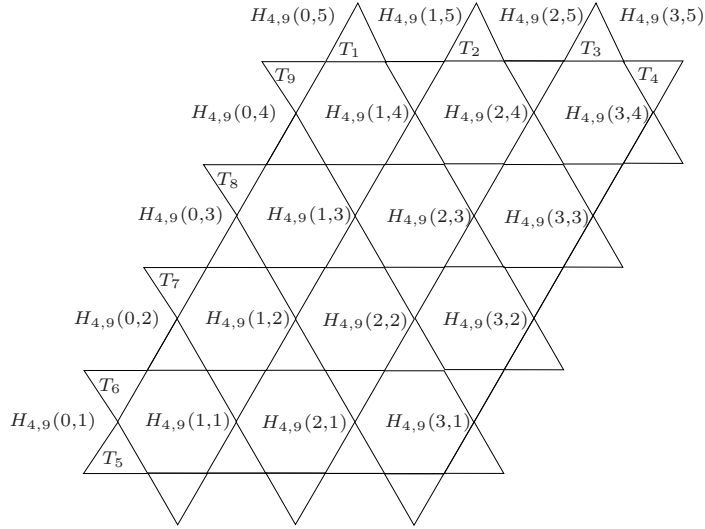}
\caption{Subset of the planar tiling for $k=4$, $n=9$, used in construction of
$\R_{4,9}$.}
\label{fig:R49hexagons}
\end{figure}

Next, the edges in the tiling are oriented by stipulating that
each horizontal edge is oriented left to right, that each triangle
is oriented (clockwise or anticlockwise), and that each hexagon has an
alternating orientation.

Boundary triangles which are oriented clockwise
(i.e.\ $T_k,T_{k+2},T_{k+3},\ldots ,T_n$) are split at the boundary vertex which
is not incident with any hexagon, while for the remaining boundary triangles the
two edges meeting at the boundary vertex are extended beyond it. The strand
starting at triangle $T_i$ is labelled $i$, for $i\in \{1,\ldots ,n\}$.
See Figure~\ref{fig:R49} for $\R_{4,9}$.
The following is easy to check.

\begin{lemma} \label{l:rknpostnikov}
The diagram $\R_{k,n}$ is a Postnikov diagram.
\end{lemma}

We use $H_{k,n}(0,0)$ to denote the bottom right boundary face.

\begin{figure}
\psfragscanon
\psfrag{1234}{$\scriptstyle 1234$}
\psfrag{2345}{$\scriptstyle 2345$}
\psfrag{3456}{$\scriptstyle 3456$}
\psfrag{4567}{$\scriptstyle 4567$}
\psfrag{5678}{$\scriptstyle 5678$}
\psfrag{6789}{$\scriptstyle 6789$}
\psfrag{1789}{$\scriptstyle 1789$}
\psfrag{1289}{$\scriptstyle 1289$}
\psfrag{1239}{$\scriptstyle 1239$}
\psfrag{1789}{$\scriptstyle 1789$}
\psfrag{1345}{$\scriptstyle 1345$}
\psfrag{1456}{$\scriptstyle 1456$}
\psfrag{1567}{$\scriptstyle 1567$}
\psfrag{1678}{$\scriptstyle 1678$}
\psfrag{1245}{$\scriptstyle 1245$}
\psfrag{1256}{$\scriptstyle 1256$}
\psfrag{1267}{$\scriptstyle 1267$}
\psfrag{1235}{$\scriptstyle 1235$}
\psfrag{1236}{$\scriptstyle 1236$}
\psfrag{1237}{$\scriptstyle 1237$}
\psfrag{1238}{$\scriptstyle 1238$}
\begin{center}
\includegraphics[width=10cm]{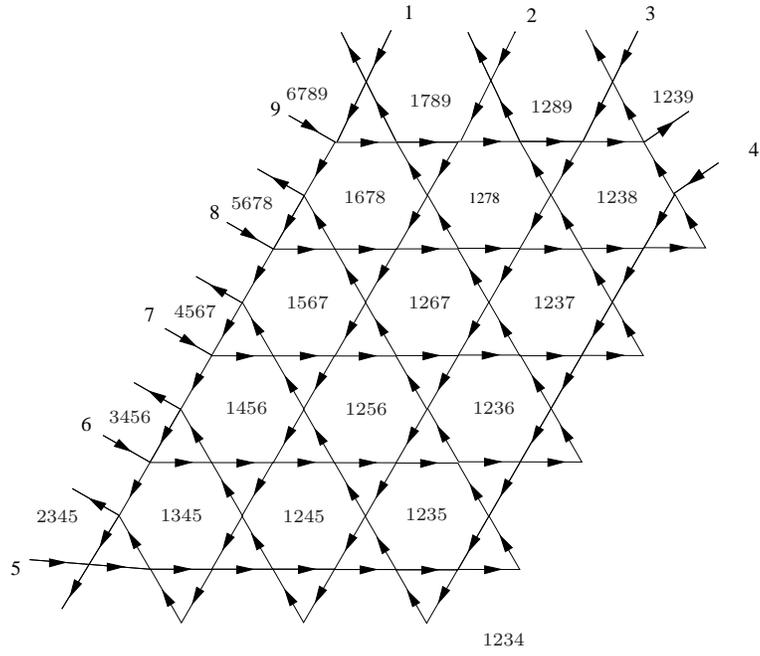}
\end{center}
\caption{The Postnikov diagram $\R_{4,9}$.}
\label{fig:R49}
\end{figure}

The bipartite dual $G_{k,n}$ of $\R_{k,n}$ consists of a subset of a
hexagonal tiling of the plane in which the edges are either vertical
or at $\pm \pi/3$ to the horizontal. There are $n-k-1$ adjacent horizontal
rows of $k-1$ hexagons, together with extra vertical edges attached
to the topmost $k-1$ vertices and a single extra edge at an angle of
$\pi/3$ to the horizontal attached to the bottom left hexagon.
See Figure~\ref{fig:G49} for $G_{4,9}$.

\begin{figure}
\begin{center}
\includegraphics[width=10cm]{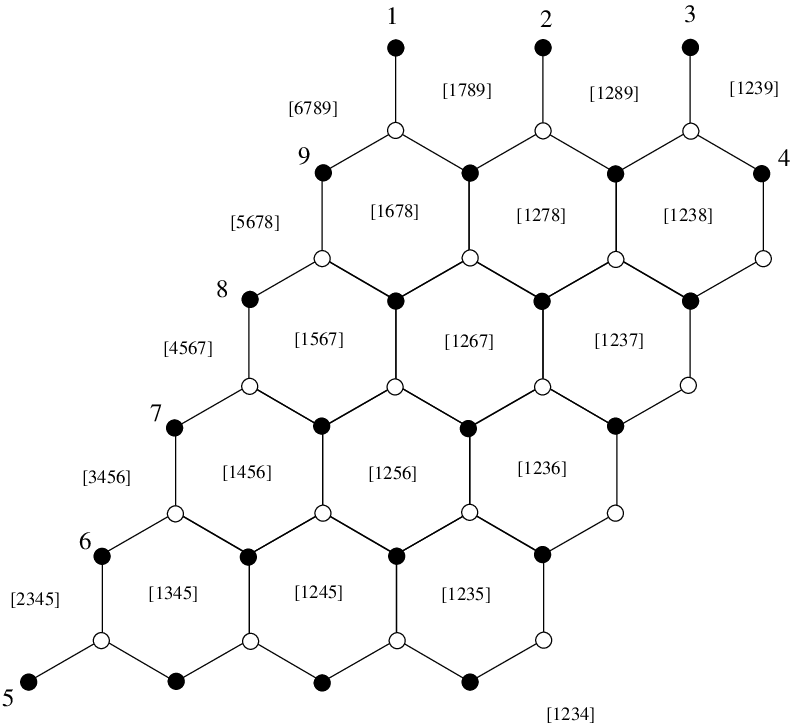}
\caption{The bipartite graph $G_{4,9}$.}
\label{fig:G49}
\end{center}
\end{figure}

\begin{lemma} \label{l:Hij}
Fix $0\leq i\leq k-1$ and $1\leq j\leq n-k$.
Then the strands (which exist) on the boundary of hexagon $H_{k,n}(i,j)$ in $\R_{k,n}$ are labelled as in Figure~\ref{fig:Hij}.
\end{lemma}

\begin{figure}
\begin{center}
\psfragscanon
\psfrag{i}{$\scriptstyle i$}
\psfrag{i+j}{$\scriptstyle i+j$}
\psfrag{j+k}{$\scriptstyle j+k$}
\psfrag{i+1}{$\scriptstyle i+1$}
\psfrag{i+j+1}{$\scriptstyle i+j+1$}
\psfrag{j+k+1}{$\scriptstyle j+k+1$}
\psfrag{X}{$\scriptstyle X$}
\psfrag{Y}{$\scriptstyle Y$}
\includegraphics[width=4cm]{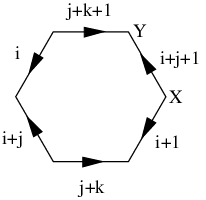}
\caption{A hexagon $H_{k,n}(i,j)$ in $\R_{k,n}$.}
\label{fig:Hij}
\end{center}
\end{figure}

\begin{lemma} \label{l:rknsubsets}
Fix $0\leq i\leq k-1$ and $1\leq j\leq n-k$.
Then the $k$-subset labelling $H_{k,n}(i,j)$ in $\R_{k,n}$ is
$$M_{k,n}(i,j)=\{1,2,\ldots, i\} \cup \{i+j+1, i+j+2, ... ,j+k\}.$$
The $k$-subset labelling $H_{k,n}(0,0)$ is $\{1,2,\ldots ,k\}$.
\end{lemma}

\begin{proof}
The statement for $i=0$ or $j=n-k$ follows from Remark~\ref{r:coefficientlocation}.
The result for all $i,j$ then follows by induction on $i$ using
Lemma~\ref{l:Hij}. Assume that the result is true for some $H_{k,n}(i,j)$.
Then, since strands $i+j+1$ and $i+1$ cross between $H_{k,n}(i,j)$ and $H_{k,n}(i+1,j)$
to its right (point $X$ in Figure~\ref{fig:Hij}),
with $i+1$ going down and $i+j+1$ going up, it follows that
$$M_{k,n}(i+1,j)=M_{k,n}(i,j)\setminus \{i+j+1\}) \cup \{i+1\},$$
as required.
\end{proof}

For a Postnikov diagram ${\postdiag}$ for $\Gr_{k,n}$ let ${\postdiag}^*$ denote the diagram obtained from ${\postdiag}$ by reversing the orientation of each of the strands in ${\postdiag}$. Each strand retains the same label. We also add an extra
crossing at the boundary between the strand starting at $i$ and the strand ending at $i'$, for each $i$. We then annihilate any local oriented lenses produced by this procedure.
The following can be seen by checking that the appropriate conditions in Definition~\ref{d:arrangement} are satisfied.

\begin{lemma}
For any Postnikov diagram ${\postdiag}$ for $\Gr_{k,n}$, the diagram ${\postdiag}^*$ is a Postnikov diagram for $\Gr_{n-k,n}$.
\end{lemma}

It follows that $\R_{n-k,n}^*$ is a Postnikov diagram for $\Gr_{k,n}$.
The labels of its alternating faces are the complements (in $\{1,2,\ldots ,n\}$) of the
labels of the corresponding faces in $\R_{n-k,n}$. For example, see $\R^*_{5,9}$
in Figure~\ref{fig:R59dual}. The bipartite graph $G_{\R^*_{5,9}}$ is shown in Figure~\ref{fig:G59dual}.

\begin{figure}
\psfragscanon
\psfrag{1234}{$\scriptstyle 1234$}
\psfrag{2345}{$\scriptstyle 2345$}
\psfrag{3456}{$\scriptstyle 3456$}
\psfrag{4567}{$\scriptstyle 4567$}
\psfrag{5678}{$\scriptstyle 5678$}
\psfrag{6789}{$\scriptstyle 6789$}
\psfrag{1789}{$\scriptstyle 1789$}
\psfrag{1289}{$\scriptstyle 1289$}
\psfrag{1239}{$\scriptstyle 1239$}
\psfrag{2349}{$\scriptstyle 2349$}
\psfrag{3459}{$\scriptstyle 3459$}
\psfrag{4569}{$\scriptstyle 4569$}
\psfrag{5679}{$\scriptstyle 5679$}
\psfrag{2389}{$\scriptstyle 2389$}
\psfrag{3489}{$\scriptstyle 3489$}
\psfrag{4589}{$\scriptstyle 4589$}
\psfrag{5689}{$\scriptstyle 5689$}
\psfrag{2789}{$\scriptstyle 2789$}
\psfrag{3789}{$\scriptstyle 3789$}
\psfrag{4789}{$\scriptstyle 4789$}
\psfrag{5789}{$\scriptstyle 5789$}
\begin{center}
\includegraphics[width=10cm]{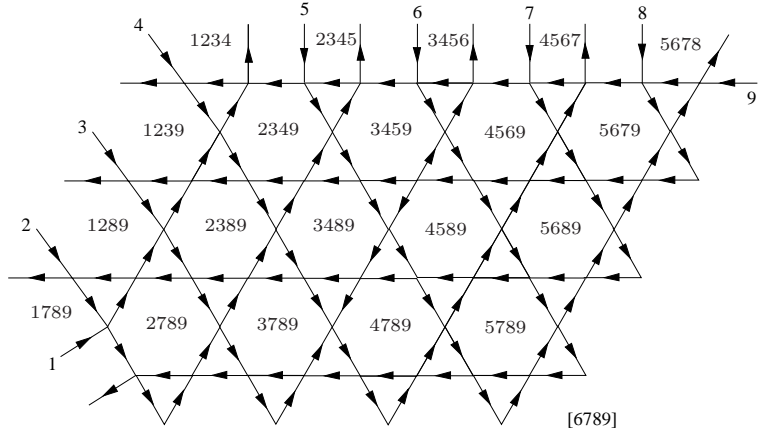}
\end{center}
\caption{The Postnikov diagram $\R^*_{5,9}$.}
\label{fig:R59dual}
\end{figure}

\begin{figure}
\begin{center}
\includegraphics[width=10cm]{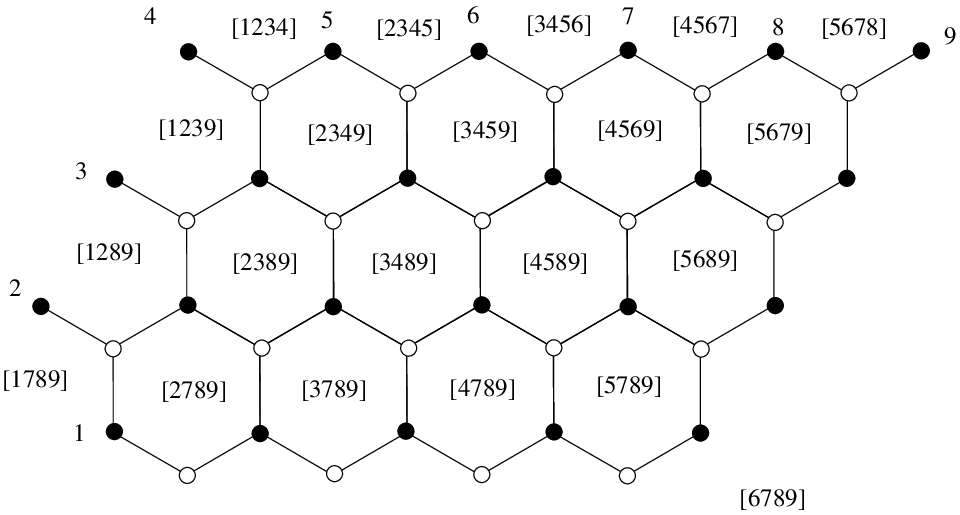}
\end{center}
\caption{The bipartite graph $G_{\R^*_{5,9}}$.}
\label{fig:G59dual}
\end{figure}

The diagram $\R_{n-k,n}$ has faces $H_{n-k,n}(j,i)$, with
$1\leq i\leq k$ and $0\leq j\leq n-k-1$ (and $i=j=0$).
The label of $H_{n-k,n}(j,i)$ is:
$$M_{n-k,n}(j,i)=\{1,2,\ldots ,j\}\cup \{i+j+1,i+j+2,\ldots ,i+n-k\}.$$
We denote the corresponding faces in $\R^*_{n-k,n}$ by $H^*_{n-k,n}(j,i)$.
The label of $H^*_{n-k,n}(j,i)$ is the complement of $M_{n-k,n}(j,i)$ in
$\{1,\ldots ,n\}$, which is:
$$M^*_{n-k,n}(j,i)=\{j+1,j+2,\ldots ,i+j\}\cup \{i+n-k+1,i+n-k+2,\ldots ,n\}.$$

\begin{lemma} \label{l:rkntwist}
Fix $0\leq i\leq k-1$ and $1\leq j\leq n-k$, or $(i,j)=(0,0)$.
Then we have
$$\twist{\minor{M_{k,n}(i,j)}}=\minor{M^*_{n-k,n}(j,i)}\cdot
\prod_{r=1}^{i-1} \mcoeff{r} \cdot \prod_{r=i+j+1}^{j+k-1} \mcoeff{r}.$$
Hence the Pl\"{u}cker coordinates whose $k$-subsets label the alternating faces of $\R_{n-k,n}^*$ are (up to multiplication by coefficients) exactly the twists of the Pl\"{u}cker coordinates whose $k$-subsets label the alternating faces of $\R_{k,n}$.
\end{lemma}

\begin{proof}
Fix $0\leq i\leq k-1$ and $1\leq j\leq n-k$, or $(i,j)=(0,0)$.
Recall that, by Lemma~\ref{l:rknsubsets},
$$M_{k,n}(i,j):=\{1,2,\ldots, i\} \cup \{i+j+1, i+j+2, ... ,j+k\}.$$
Hence, by Proposition~\ref{p:twistcomputation},
$$\twist{\minor{M_{k,n}(i,j)}}=\minor{J}\cdot
\prod_{r=1}^{i-1} \mcoeff{r} \cdot \prod_{r=i+j+1}^{j+k-1} \mcoeff{r},$$
where
$$J=\{n+1+i-k,n+2+i-k,\ldots ,n\} \cup \{j+1,\ldots ,i+j\}=M^*_{n-k,n}(j,i),$$
and the result follows.
\end{proof}

By Lemmas~\ref{l:rknpostnikov} and~\ref{l:rknsubsets}, the
pair $(\xx(\R_{k,n}),\Q(\R_{k,n}))$ is a seed of
$\mathbb{C}[\Gr_{k,n}]$, where
$$\xx(\R_{k,n})=\{\minor{M_{k,n}(i,j)}\,:\,0\leq i\leq k-1,1\leq j\leq n-k\text{ or }i=j=0\}.$$

Recall that, for a quiver $\Q$ in a seed of $\mathbb{C}[\Gr_{k,n}]$, $Q$ denotes the principal part of $\Q$. In particular,
$Q(\R_{k,n})$ is the principal part of $\Q(\R_{k,n})$.

\begin{defn} \label{d:specializedcoefficients}
We define $\mathbb{C}[\Gr_{k,n}]_1$ to be the cluster algebra in $\mathbb{C}(\mathbf{x}_{\R_{k,n}})$ with initial seed \\
$(\mathbf{x}_{\R_{k,n}},Q(\R_{k,n}))$.
\end{defn}

Note that using any other seed of $\mathbb{C}[\Gr_{k,n}]$  in 
Definition~\ref{d:specializedcoefficients} would give the same result.

\begin{remark} \label{r:specialize}
By~\cite[\S12]{fominzelevinsky07}
(see also~\cite[\S6]{ADS14}) there is a surjective ring
homomorphism $\pi:\mathbb{C}[\Gr_{k,n}]\rightarrow \mathbb{C}[\Gr_{k,n}]_1$ mapping the coefficients in $\mathbb{C}[\Gr_{k,n}]$ to $1$
and cluster variables to cluster variables (surjectively) such
that whenever $S=(\xx,\Q)$ is a seed in $\mathbb{C}[\Gr_{k,n}]$, $\pi(S)=(\pi(\mathbf{x}),Q)$ is a seed in $\mathbb{C}[\Gr_{k,n}]_1$. This map respects mutation in the following sense. If $S=(\xx,\Q)$ and
$S'=(\xx',\Q')$ are seeds in $\mathbb{C}[\Gr_{k,n}]$ and $S'$ is obtained from $S$ by mutation at $x\in\mathbf{x}$, then
$\pi(S')$ can be obtained from $\pi(S)$ by mutation at $\pi(x)$.
\end{remark}

\begin{defn} \label{d:uptocoefficients}
We shall say that two elements $f,g\in \mathbb{C}[\Gr_{k,n}]$
are \emph{equal up to coefficients} if $\pi(f)=\pi(g)$, and that
two homomorphisms $\varphi_1,\varphi_2:\mathbb{C}[Gr_{k,n}]\rightarrow \mathbb{C}[\Gr_{k,n}]$ are \emph{equal
up to coefficients} if $\pi\circ \varphi_1=\pi\circ \varphi_2$.
\end{defn}

We have the following:

\begin{prop} \label{p:periodicitycoefficients}
Up to coefficients, the $(2n)$th power of the twist on $\mathbb{C}[\Gr_{k,n}]$ is equal to the identity.
\end{prop}
\begin{proof}
Suppose that $k>1$ and $I$ is a $k$-subset of $\{1,\ldots ,n\}$.
By Proposition~\ref{p:periodicitymonomial}, applying the twist $2n$ times to $[I]\in \mathbb{C}[\Gr_{k,n}]$ gives $[I]$ multiplied by a
monomial in the Pl\"{u}cker coordinates $\mcoeff{i}$,
$i\in \{1,\ldots ,n\}$. The images of this element and $[I]$ under
$\pi$ are equal. The result then follows from the fact that the
Pl\"{u}cker coordinates generate $\mathbb{C}[\Gr_{k,n}]$.
If $k=1$, then the twist of any Pl\"{u}cker coordinate is $1$ by
Remark~\ref{r:coefficientcase}, giving the result in this case.
\end{proof}

We can also show that the twist of a cluster variable is again
a cluster variable, up to coefficients. We thank David Speyer
for communicating this proof to us.

\begin{prop} \label{p:twistofclustervariable}
Let $(\xx,\Q)$ be a seed in $\mathbb{C}[\Gr_{k,n}]$. Then
$\left(\pi\twistbracket{\mathbf{x}},Q\right)$ is a seed in $\mathbb{C}[\Gr_{k,n}]_1$.
In particular, if $x$ is a cluster variable in $\mathbb{C}[\Gr_{k,n}]$ then $\pi\twistbracket{x}$ is a cluster variable
in $\mathbb{C}[\Gr_{k,n}]_1$, and hence of the form $\pi(y)$ for
some cluster variable $y$ in $\mathbb{C}[\Gr_{k,n}]$.
\end{prop}
\begin{proof}
Since the principal parts of $\R_{k,n}$ and $\R^*_{n-k,n}$
are isomorphic, the statement holds for the seed
$(\xx(\R_{k,n}),\Q(\R_{k,n}))$ by Lemma~\ref{l:rkntwist}.
Let $(\xx,\Q)$ be a seed of $\mathbb{C}[\Gr_{k,n}]$.
We shall prove the result by induction on the length of the shortest path in the exchange graph from
$(\xx_{\R_{k,n}},\Q(\R_{k,n}))$ to $(\xx,\Q)$.

Suppose that the result holds for $(\xx,\Q)$ and that the seed $S'=(\xx',\Q')$ is obtained from $S=(\xx,\Q)$ by mutating at a cluster variable $x_r$, replacing it with $x_r'$. The corresponding exchange relation is:
\begin{equation}
\label{e:xrexchange}
x_rx'_r=\prod_{j\rightarrow r} x_j+\prod_{r\rightarrow j} x_j,
\end{equation}
with the first product taken over all arrows in $\Q$ ending at $r$, and the second over all arrows in $\Q$ starting at $r$.
We apply the twist and then $\pi$ to~\ref{e:xrexchange}.
Note that both of these maps are ring homomorphisms and the 
twist of a coefficient is a product of coefficients by
Remark~\ref{r:coefficientcase}. We therefore obtain:
\begin{equation}
\label{e:xrexchange2}
\pi \twistbracket{x_r} \pi\twistbracket{x'_r}=\prod_{j\rightarrow r}
\pi\twistbracket{x_j}+\prod_{r\rightarrow j} \pi\twistbracket{x_j},
\end{equation}
with the first product taken over all arrows in $Q$ ending at $r$, and the second over all arrows in $Q$ starting at $r$.

By our assumption, $\left(\pi\twistbracket{\mathbf{x}},Q\right)$ is a
seed in $\mathbb{C}[\Gr_{k,n}]_1$. Hence,~\eqref{e:xrexchange2} is the exchange relation for $\pi\twistbracket{x_r}$ in the seed
$\left(\pi\twistbracket{\mathbf{x}},Q\right)$ of $\mathbb{C}[\Gr_{k,n}]_1$. It follows that the seed in $\mathbb{C}[\Gr_{k,n}]_1$ obtained from 
$\pi\twistbracket{\mathbf{x}}$ via~\eqref{e:xrexchange2} is
$\left(\pi\twistbracket{\mathbf{x}'},Q'\right)$, since the mutation of
$Q$ at $r$ is $Q'$, the principal part of $\Q'$.
In particular, this pair is a seed in $\mathbb{C}[\Gr_{k,n}]_1$ as required. The result follows by induction.
\end{proof}

Next, we consider seven different choices of edges of a hexagon, $O,A,B,C,X,Y,Z$,
displayed in Figure~\ref{fig:hexagontypes}. Dashed lines are not in the dimer
configuration, while full lines are in the dimer configuration.

\begin{figure}
\begin{center}
\psfragscanon
\psfrag{O}{$O$}
\psfrag{A}{$A$}
\psfrag{B}{$B$}
\psfrag{C}{$C$}
\psfrag{X}{$X$}
\psfrag{Y}{$Y$}
\psfrag{Z}{$Z$}
\includegraphics[width=10cm]{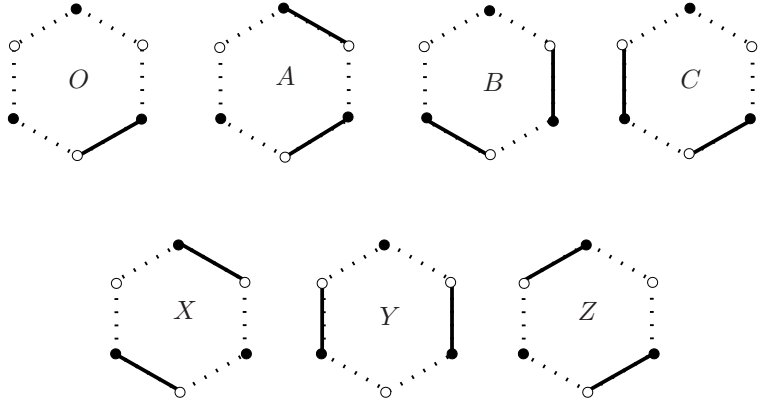}
\end{center}
\caption{Edge choices in a hexagon}
\label{fig:hexagontypes}
\end{figure}

If $i$ is a boundary vertex which is
incident to a unique edge in the dual
bipartite graph $G_P$ of a Postnikov diagram
$P$, we say that $i$ is on a \emph{stalk}.
For any $k$-subset $I$ of $\{1,\ldots ,n\}$,
we modify $G_{\R^*_{n-k,n}}$ by applying blow-up moves to all boundary vertices in
$I$ which are not on stalks.
This ensures that each boundary vertex
in $G_{\R^*_{n-k,n}}(I)$ is on a stalk.
We draw all edges on the boundary of the diagram in such a way as to continue
the hexagonal tiling.

\begin{prop} \label{p:uniquematching}
Fix $0\leq i\leq k-1$ and $1\leq j\leq n-k$ or $(i,j)=(0,0)$. Then
the bipartite graph $G_{\R^*_{n-k,n}}(M_{k,n}(i,j))$ has a unique dimer configuration,
in which a hexagon $H^*_{n-k,n}(a,b)$ for $0\leq a\leq n-k-1$ and $1\leq b\leq k$
has type given by the following table:
\vskip 0.2cm
\begin{center}
\begin{tabular}{c|c}
Type of hexagon & Restriction on $a,b$ \\
\hline
$O$ & $(a,b)=(j,i)$ \\
$A$ & $a<j$, $b=i$ \\
$B$ & $a+b=i+j$, $a<j$ (equivalent to $b>i$) \\
$C$ & $a=j$, $b>i$ \\
$X$ & $a+b<i+j$, $b>i$ \\
$Y$ & $a+b>i+j$, $a<j$ \\
$Z$ & otherwise.
\end{tabular}
\end{center}
\vskip 0.2cm
Here we regard an incomplete hexagon as having the appropriate type if,
for the edges that do appear, the edges in the dimer configuration correspond
to the choice of edges in the type (note that this may not be unique).
\end{prop}

\begin{proof}
Let ${\dimerconfig}$ be a dimer configuration on $G=G_{\R^*_{n-k,n}}(M_{k,n}(i,j))$.
Denote the set of hexagons appearing in the table next to letter $A$ by
$\HH(A)$, and similarly for $A,B,C,X,Y$ and $Z$. We may write $M_{k,n}(i,j)$ in the
form $I_1\cup I_2$ where $I_1$ and $I_2$ are disjoint integer intervals (i.e.\
in the usual sense of intervals), where $I_2$ may be empty.

Note that any edge in a stalk in
$G_{\R^*_{n-k,n}}(M_{k,n}(i,j))$ must be in a perfect matching on this graph.
We use this fact repeatedly in the following.

\noindent \textbf{Case I}:
Firstly, assume that $1\leq i\leq k-1$, $1\leq j\leq n-k-1$.
Then $I_1=\{1,\ldots, i\}$ and $I_2=\{i+j+1,\ldots ,j+k\}$, and both are
nonempty. Write $\{1,\ldots ,n\}=I_1\cup J_1\cup I_2\cup J_2$ as disjoint
union of intervals (not cyclic). Let $a_1$ (respectively, $a_2,b_1,b_2$)
be the largest element of $I_1$ (respectively, $I_2,J_1,J_2$). Note that
$b_2=n$.

\noindent \textbf{Case I(a)}:
Suppose first that $i+j\leq k$.
In this case, the boundary face labelled
$\mcoeff{a}$ in $G$ must be of type $Z$ if $a\in I_1\setminus \{a_1\}$, of type $A$ if $a=a_1$, of type $X$ if $a\in J_1\setminus \{b_1\}$, of type $B$ if $a=b_1$, of type $Y$ if $a\in I_2\setminus \{a_2\}$, of type $C$ if $a=a_2$ and of type
$Z$ if $a\in J_2\setminus \{b_2\}$. See Figure~\ref{fig:large1} for an example.

\noindent \textbf{Case I(b)}:
Secondly, assume that $i+j>k$.
In this case, the boundary face labelled $\mcoeff{a}$ in $G$ must be of type $Z$ if $a\in I_1\setminus \{a_1\}$, of type $A$ if $a=a_1$, of
type $X$ if $a\in J_1\setminus \{b_1\}$, of type $B$ if $a=b_1$, of type
$Y$ if $a\in I_2\setminus \{a_2\}$, of type $C$ if $a=a_2$, and of type $Z$
if $a\in J_2\setminus \{b_2\}$. See Figure~\ref{fig:large2} for an example.

\noindent \textbf{Case II}:
Next, we assume that $i=0$ and $1\leq j\leq n-k$. Then
$I_1=\{j+1,\ldots ,j+k\}$ and $I_2$ is empty.
Write $\{1,\ldots ,n\}$ as
a disjoint union $J_1\cup I_1\cup J_2$ of intervals, defining $a_1,b_1,b_2$
as above. Note that $b_2=n$.
Then a boundary face labelled $\mcoeff{a}$ in $G$
must be of type $X$ if $a\in J_1\setminus \{b_1\}$, of type $B$ if
$a=b_1$, of type $Y$ if $a\in I_1\setminus \{a_1\}$, of type $C$ if $a=a_1$,
and of type $Z$ if $a\in J_2\setminus \{b_2\}$.
See Figure~\ref{fig:large3} for an example.

\noindent \textbf{Case III}:
Next, we assume that $1\leq i\leq k-1$ and $j=n-k$. Then
$I_1=\{1,\ldots ,i\}$ and $I_2=\{i+n-k+1,\ldots ,n\}$, and both are non-empty.
Write $\{1,\ldots ,n\}=I_1\cup J_1\cup I_2$ as a disjoint union of
intervals, defining $a_1,a_2,b_1$ as before.
Note that $a_2=n$. Then
a boundary face labelled $\mcoeff{a}$ in $G$ 
must be of type $Z$ if $a\in I_1\setminus
\{a_1\}$, of type $A$ if $a=a_1$, of type
$X$ if $a\in J_1\setminus \{b_1\}$,
of type $B$ if $a=b_1$, of type $Y$ if
$a\in I_2\setminus \{a_2\}$.
See Figure~\ref{fig:large4} for an example.

\noindent \textbf{Case IV}:
The final case is $i=j=0$. Then
$I_1=\{1,\ldots ,k\}$ and $I_2$ is empty.
Then a boundary face labelled $\mcoeff{a}$ in $G$ for any $a\geq \{1,\ldots ,n\}$ must be of type $Z$.

In any of these cases, this forces the hexagons in $\HH(X)$, $\HH(Y)$ and $\HH(Z)$ to be of type $X$, $Y$ or $Z$, respectively, using an easy induction argument starting from the boundary. This also forces the hexagons in the boundary strips
$\HH(A)$, $\HH(B)$ and $\HH(C)$ inbetween to be of the appropriate types.
If $1\leq i\leq k-1$ and $1\leq j\leq n-k-1$, i.e.\ if $M_{k,n}(i,j)$ is not
a coefficient, a single hexagon remains, i.e.\ the unique element of $\HH(O)$,
which is forced to be of type $O$. Otherwise, there are no hexagons remaining
and type $O$ does not occur.

Note that the regions $\HH(X)$ and $\HH(Y)$ are roughly triangular
(or truncated triangular), with $\HH(Z)$ forming the remainder
(apart from $\HH(O)$ and the boundary strips $\HH(A),\HH(B)$ and $\HH(C)$).
\end{proof}

\begin{remark}
If $i=0$ or $j=n-k$, then the restrictions $a<j$ for $A$ and $B$ and $b>i$ for $C$ are automatically satisfied and thus can be omitted. Furthermore, if
if $i=0$ and $0\leq j\leq n-k-1$ there are no hexagons of type $A$, and if
if $1\leq i\leq k-1$ and $j=n-k$, there are no hexagons of type $C$.
If $i=j=0$ then every hexagon is of type $Z$.
\end{remark}

The example in Figure~\ref{fig:large1} shows a case where $\HH(Y)$ forms a truncated
triangle, while in Figure~\ref{fig:large2}, $\HH(X)$ forms a truncated triangle.
Figure~\ref{fig:large3} is a case where $I_2$ is empty (and $\HH(Y)$ is a truncated
triangle), while Figure~\ref{fig:large4} illustrates a case where $\HH(X)$ is
a truncated triangle. In the examples in Figures~\ref{fig:large3} and~\ref{fig:large4}, $\minor{M_{k,n}(i,j)}$ is a coefficient.

In each case, the edges in the dimer configuration are drawn as full lines, and the other edges as dotted lines). The type of each hexagon is also indicated.

\begin{figure}
\begin{center}
\includegraphics[width=14cm]{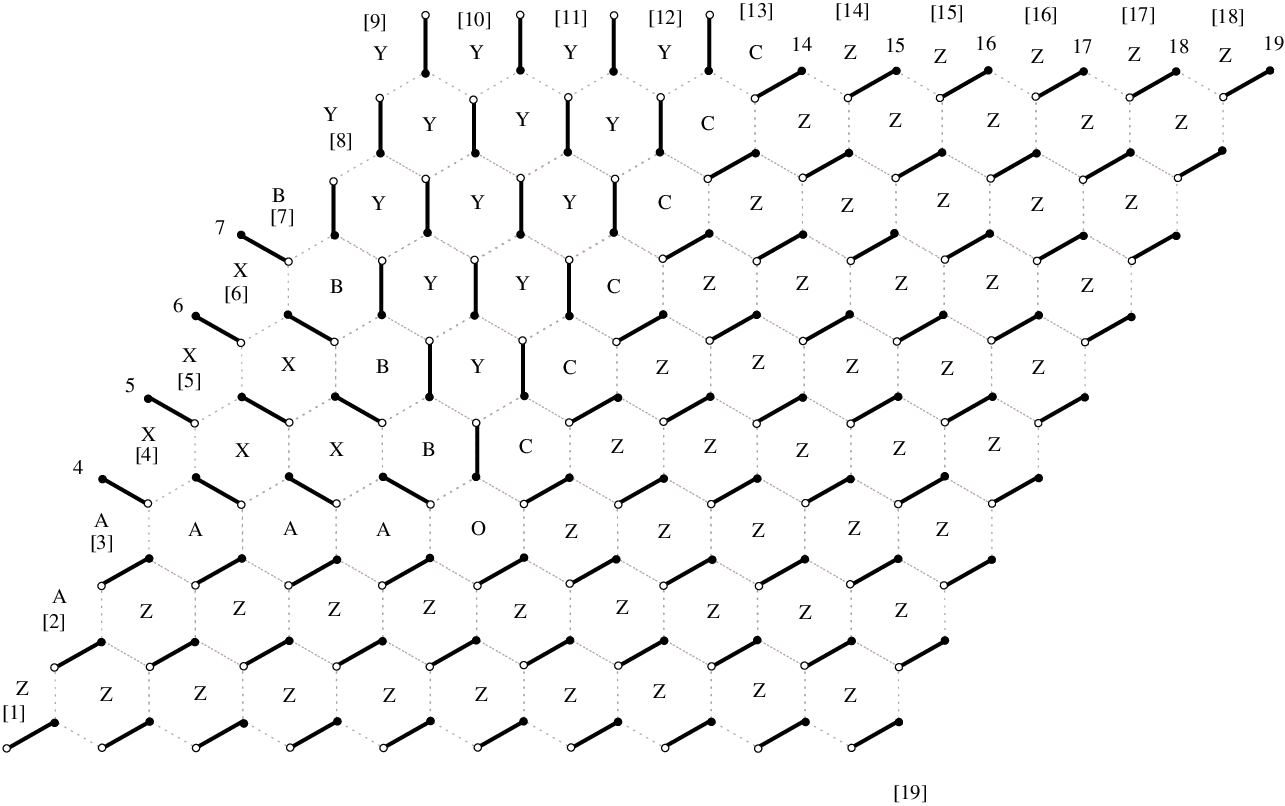}
\end{center}
\caption{The unique dimer configuration in $G_{\R^*_{10,19}}(M_{9,19}(3,4))=G_{\R^*_{10,19}}(\{1,2,3\}\cup \{8,9,10,11,12,13\})$. As per
definition,
$G_{\R^*_{10,19}}(M_{9,19}(3,4))$ is
constructed from $G_{\R^*_{10,19}}$ by
deleting the boundary vertices
in $M_{9,19}(3,4)=\{1,2,3\}\cup \{8,9,10,11,12,13\}$ (and their labels).}
\label{fig:large1}
\end{figure}

\begin{figure}
\begin{center}
\includegraphics[width=14cm]{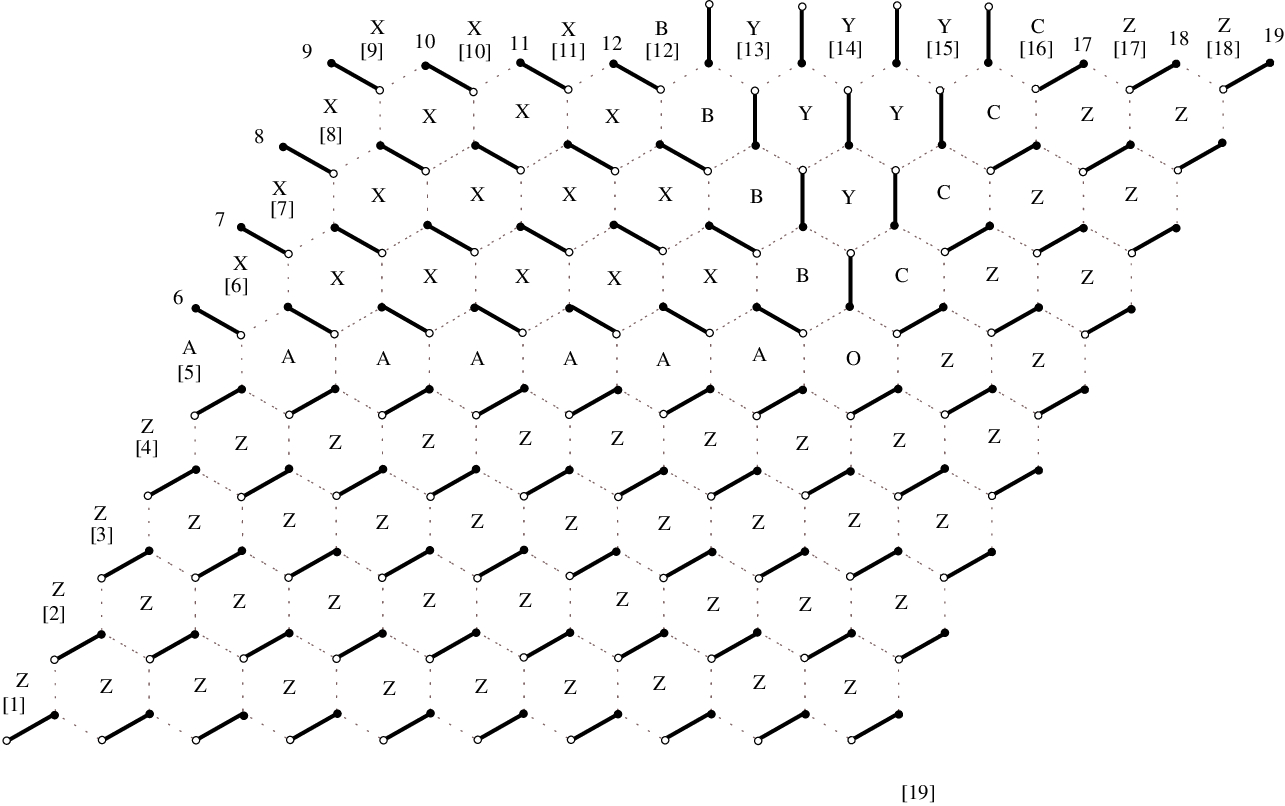}
\end{center}
\caption{The unique dimer configuration in $G_{\R^*_{10,19}}(M_{9,19}(5,7))=G_{\R^*_{10,19}}(\{1,2,3,4,5\}\cup \{13,14,15,16\})$.}
\label{fig:large2}
\end{figure}

\begin{figure}
\begin{center}
\includegraphics[width=14cm]{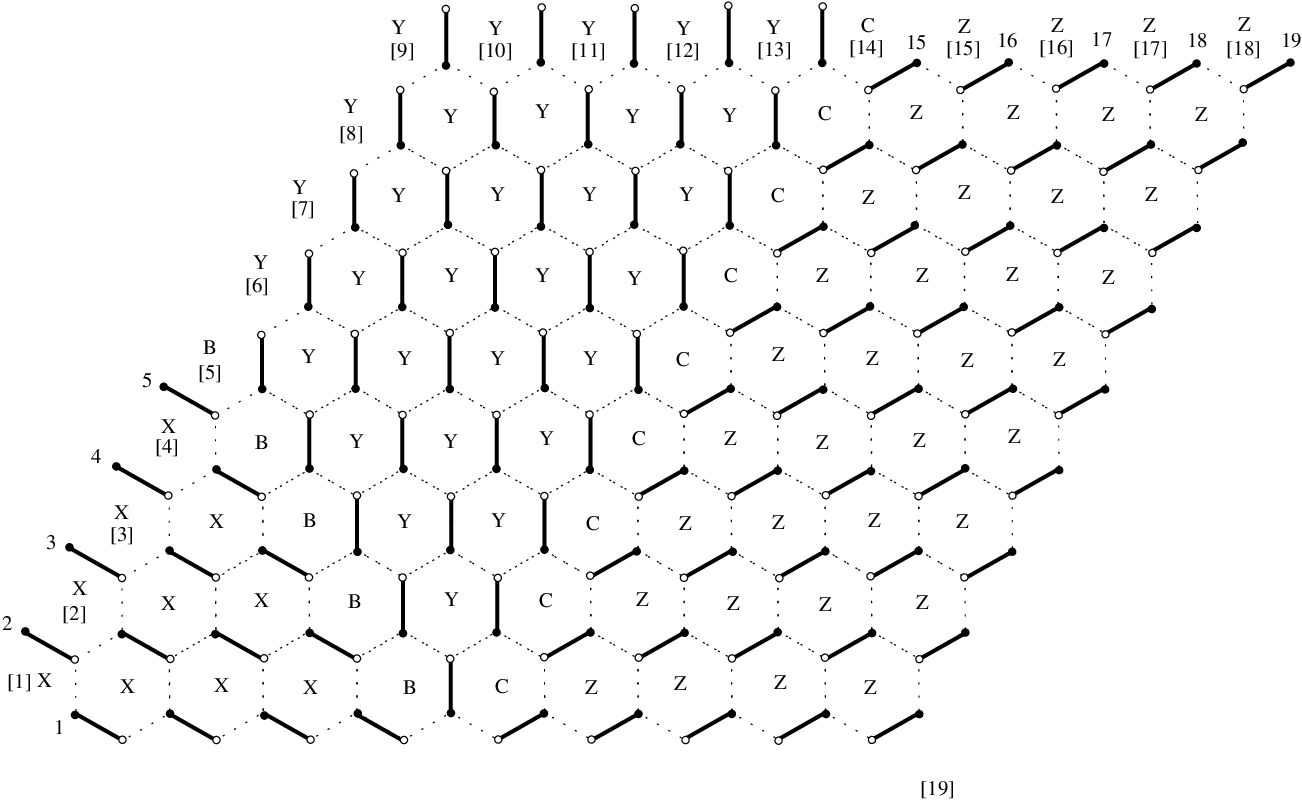}
\end{center}
\caption{The unique dimer configuration in $G_{\R^*_{10,19}}(M_{9,19}(0,5))=G_{\R^*_{10,19}}(\{6,7,8,9,10,11,12,13,14\})$.}
\label{fig:large3}
\end{figure}

\begin{figure}
\begin{center}
\includegraphics[width=14cm]{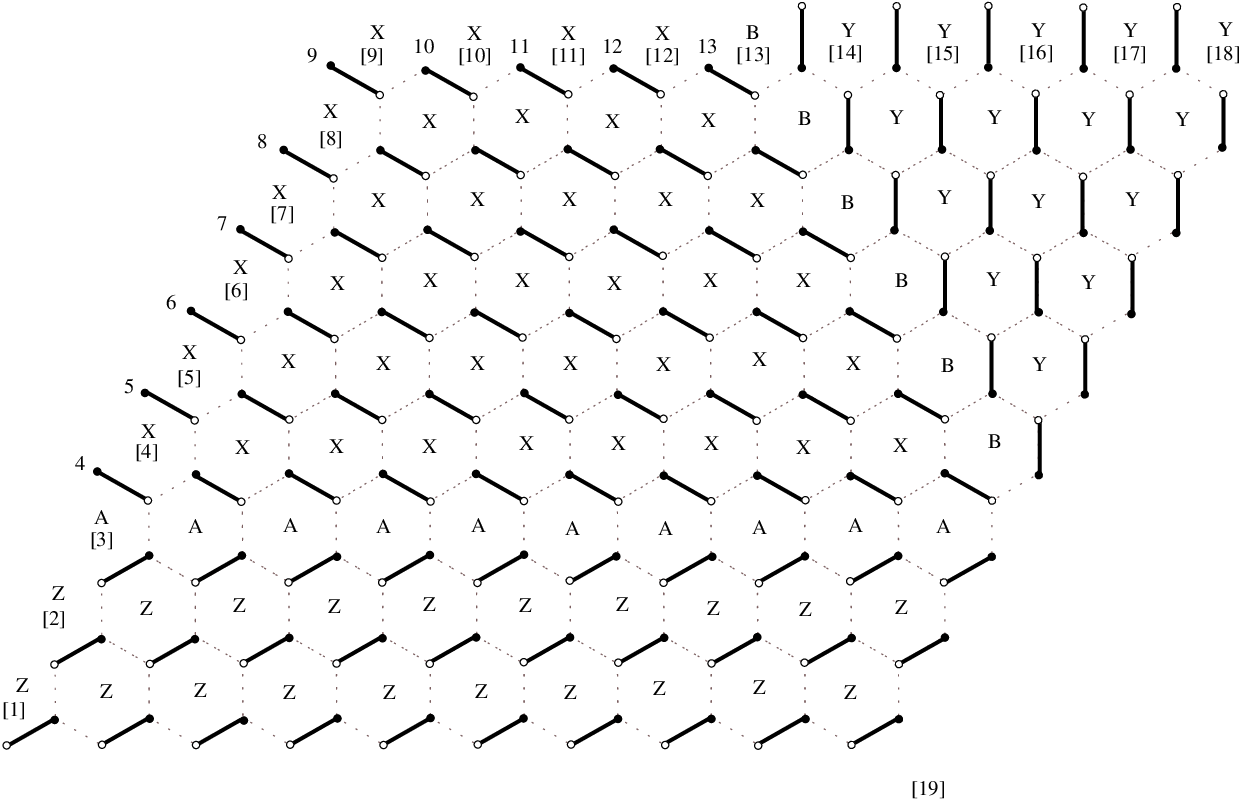}
\end{center}
\caption{The unique dimer configuration in $G_{\R^*_{10,19}}(M_{9,19}(3,10))=G_{\R^*_{10,19}}(\{1,2,3\}\cup \{14,15,16,17,18,19\})$.}
\label{fig:large4}
\end{figure}

\begin{corollary}
For any Postnikov diagram $P$ and $k$-subset $I$, the dual bipartite graph
$G_P(I)$ is balanced.
\end{corollary}
\begin{proof}
By Proposition~\ref{p:uniquematching}, $G_{\R^*_{n-k,n}}(I)$ has a dimer configuration
for any $k$-subset $I$ labelling $\R_{k,n}$, so it is balanced. It follows that $G_{\R^*_{n-k,n}}(I)$ is balanced for any $k$-subset $I$, since all the boundary
vertices are black. Since the blow-up, blow-down and quadrilateral moves preserve
the difference between the number of white vertices and the number of black vertices,
the result follows, using Proposition~\ref{p:connected}.
\end{proof}

Fix $0\leq i\leq k-1$ and $1\leq j\leq n-k$ or $(i,j)=(0,0)$. Then we have the scaled dimer partition function $\uu_{\R^*_{n-k,n}}({M_{k,n}(i,j)})$, which is the
dimer partition function of $G_{\R^*_{n-k,n}}(M_{k,n}(i,j))$ divided by the product of Pl\"{u}cker coordinates labelling the interior faces of $G_{\R^*_{n-k,n}}(M_{k,n}(i,j))$ (i.e.\ the Pl\"{u}cker coordinates lying in the corresponding cluster which are not coefficients).

\begin{prop} \label{p:rkncase}
Fix $0\leq i\leq k-1$ and $1\leq j\leq n-k$ or $(i,j)=(0,0)$.
Then $$\twist{\minor{M_{k,n}(i,j)}}=\uu_{\R_{n-k,n}^*}({M_{k,n}(i,j)}).$$
\end{prop}

\begin{proof}
Recall that, by Lemma~\ref{l:rkntwist}, we have:
$$\twist{\minor{M_{k,n}(i,j)}}=\minor{M^*_{n-k,n}(j,i)}\cdot
\prod_{r=1}^{i-1} \mcoeff{r} \cdot \prod_{r=i+j+1}^{j+k-1} \mcoeff{r}.$$

By Proposition~\ref{p:uniquematching}, $\uu_{\R_{n-k,n}^*}(M_{k,n}(i,j))$ is the weight
$\weight_{\dimerconfig}$ associated to the unique dimer configuration ${\dimerconfig}$ on $G_{\R^*_{n-k,n}}(M_{k,n}(i,j))$ given in the lemma, divided by the product of the Pl\"{u}cker coordinates labelling non-boundary faces in ${\R}_{n-k,n}^*(M_{k,n}(i,j))$.
By the definition of the weighting on $G_{\R^*_{n-k,n}}(M_{k,n}(i,j))$,
the exponent of a Pl\"{u}cker coordinate in $\weight_{\dimerconfig}$ associated to a face of
$G_{\R^*_{n-k,n}}(M_{k,n}(i,j))$ is equal to the number of edges
in ${\dimerconfig}$ for which only the white vertex of the edge is incident with the face.
Therefore (since ${\dimerconfig}$ is a dimer configuration), the exponent in $\weight_{\dimerconfig}$ of a Pl\"{u}cker coordinate
corresponding to a face of $G_{\R^*_{n-k,n}}(M_{k,n}(i,j))$ coincides with the number of white vertices
on the boundary of the face which are not incident with an edge incident with two
vertices on the boundary.

It follows (see the description of the hexagon types in Figure~\ref{fig:hexagontypes})
that the exponent of a Pl\"{u}cker coordinate corresponding to a non-boundary face is
exactly one for all cases except the Pl\"{u}cker coordinate corresponding to the hexagon labelled $O$ (if one exists)
in which case the exponent is $2$. Note that this hexagon, if it exists, is labelled with the Pl\"{u}cker coordinate $\minor{M^*_{n-k,n}(j,i)}$.

We can now check that, in each of the cases of the proof of Proposition~\ref{p:uniquematching}, and using the notation there, the exponent of a coefficient $\mcoeff{a}$ in $\uu_{\R_{n-k,n}^*}(M_{k,n}(i,j))$ is $1$ if $a\in I_1\setminus \{a_1\}$ or $a\in I_2$, and zero otherwise. For example, in Case I(a), if $a\in I_1\setminus \{a_1\}$, then the part of $G_{\R^*_{n-k,n}}(M_{k,n}(i,j))$ adjacent to the boundary face
labelled $\mcoeff{a}$ must be as in Figure~\ref{fig:HZ}, so the exponent of $\mcoeff{a}$ must be $1$. The other cases are similar.
The result follows.
\end{proof}

\begin{figure}
\psfragscanon
\psfrag{a}{$a$}
\psfrag{a+1}{$a+1$}
\psfrag{[a]}{$\mcoeff{a}$}
\includegraphics[width=2.5cm]{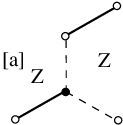}
\caption{Proof of Proposition~\ref{p:rkncase}.}
\label{fig:HZ}
\end{figure}

\section{Proof of the main result}
\label{s:mainresult}
In this section we prove our main result. Recall that the short Pl\"{u}cker
relations are the relations of the form
$$\minor{Jac}\minor{Jbd}=\minor{Jab}\minor{Jcd}+\minor{Jad}\minor{Jbc},$$
where $J$ is a $(k-2)$-subset with $J\cap \{a,b,c,d\}=\phi$.

\begin{prop} \label{p:uIshort}
Let $P$ be a Postnikov diagram. Then the elements $\uu_{G_P}(I)$, for $I$ a $k$-subset of $\{1,\ldots ,n\}$, satisfy the short Pl\"{u}cker relations.
\end{prop}
\begin{proof}
The proposition for the $\dimerpart_{G_P}(I)$
follows from~\cite[Thm.\ 2.1]{kuo04}, noting that the proof given there
carries over to weighted graphs as well as the enumeration result stated.
But then the result follows, as the $\uu_{G_P}(I)$ are obtained from the $\dimerpart_{G_P}(I)$ by scaling by a constant monomial.
\end{proof}

\begin{lemma} \label{l:tIshort}
The elements $\twist{\minor{I}}$, for $I$ a $k$-subset of $\{1,\ldots ,n\}$, satisfy the short Pl\"{u}cker relations.
\end{lemma}

\begin{proof}
This follows from the facts that the Pl\"{u}cker coordinates $\minor{I}$ satisfy the short Pl\"{u}cker relations and that the twist is a homomorphism of algebras.
\end{proof}

\begin{proof}[Proof of Theorem~\ref{t:mainresult}]
If $k=1$ or $n-1$ there is, up to blow-up and blow-down moves, a unique Postnikov
diagram, and it is easy to check that the
result holds directly. Otherwise,
if $I$ is the label of an alternating face of $\R_{k,n}$, then the theorem holds
for $I$ by Proposition~\ref{p:rkncase}.
For the general case, we prove the result by induction on the number $e$ of quadrilateral moves required to get from $\R_{k,n}$ to a Postnikov diagram ${\postdiag}$ which has $I$ labelling one of its alternating faces. The above deals with the case $e=0$.
Suppose the result is known for all $k$-subsets labelling Postnikov diagrams which can be obtained by a sequence of fewer than $e$ quadrilateral moves starting from $\R_{k,n}$. Then there is a Postnikov diagram ${\postdiag}'$, related to ${\postdiag}$ by a quadrilateral move, for which the result is known for all $k$-subsets labelling ${\postdiag}'$. So if $I$ labels an alternating face of ${\postdiag}'$, we are done. But if not, $\minor{I}$ is related to the Pl\"{u}cker coordinates of ${\postdiag}'$ by a short
Pl\"{u}cker relation. That the result then holds for $I$ follows from
Proposition~\ref{p:uIshort} and Lemma~\ref{l:tIshort}. The result then follows by induction,
using Proposition~\ref{p:connected}.
\end{proof}

\begin{remark} \label{r:twistk2}
If $k=2$ then, by Proposition~\ref{p:twistcomputation}, the twist of an arbitrary Pl\"{u}cker coordinate $\minor{ab}$ is
$\twist{\minor{ab}}=\minor{\sigma(a)\sigma(b)}$ and we see that every Pl\"{u}cker coordinate is the twist of a Pl\"{u}cker coordinate, so Theorem~\ref{t:mainresult} gives a formula for any Pl\"{u}cker coordinate as a positive Laurent polynomial in terms of any initial cluster. We note that in this case, such a formula has already been given in~\cite{schiffler08}.
\end{remark}

\begin{remark} \label{r:factorizable}
Note also that it follows from Theorem~\ref{t:mainresult} that for every Postnikov diagram and $k$-subset $I$, the bipartite graph $G_P(I)$ admits at least one dimer
configuration, i.e.\ it is \emph{factorizable}.
\end{remark}

Recall that the \emph{totally positive} part
of the real Grassmannian $Gr_{k,n}(\mathbb{R})$ is the subset
$$(\Gr_{k,n})_{>0}=\{p\in Gr_{k,n}(\mathbb{R})\,:\,\minor{I}(p)>0\text{ for all $k$-subsets $I$ of $\{1,\ldots ,k\}$}\},$$
while the totally nonnegative part is
$$(\Gr_{k,n})_{\geq 0}=\{p\in Gr_{k,n}(\mathbb{R})\,:\,\minor{I}(p)\geq 0\text{ for all $k$-subsets $I$ of $\{1,\ldots ,k\}$}\}.$$
 
\begin{corollary} \label{c:preservespositive}
The twist preserves the totally positive part of the Grassmannian and the totally nonnegative part of the Grassmannian.
\end{corollary}
\begin{proof}
This follows from Theorem~\ref{t:mainresult}
and Remark~\ref{r:factorizable}.
\end{proof}

\section{An example}
\label{s:example}
We give an example of the main result, taking $k=3$ and $n=6$. Consider the
Postnikov diagram ${\postdiag}$ for $\Gr_{3,6}$ shown in Figure~\ref{fig:ex36}. The weighted bipartite graph $G_{\postdiag}$ is shown in Figure~\ref{fig:plabic36}, and
$G_{\postdiag}(\{2,5,6\})$ can be obtained from
$G_{\postdiag}$
by removing the black boundary vertices
labelled $2,5$ and $6$.
There are six dimer configurations
$\dimerconfig_1,\ldots ,\dimerconfig_6$
on $G_{\postdiag}(\{2,5,6\})$, and 
the corresponding monomials $
\weight_{\dimerconfig_i}$ are shown in Figure~\ref{fig:matchings36}.

\begin{figure}
\psfragscanon
\psfrag{123}{$\scriptstyle 123$}
\psfrag{234}{$\scriptstyle 234$}
\psfrag{345}{$\scriptstyle 345$}
\psfrag{456}{$\scriptstyle 456$}
\psfrag{156}{$\scriptstyle 156$}
\psfrag{126}{$\scriptstyle 126$}
\psfrag{235}{$\scriptstyle 235$}
\psfrag{135}{$\scriptstyle 135$}
\psfrag{356}{$\scriptstyle 356$}
\psfrag{125}{$\scriptstyle 125$}
\psfrag{1'}{$\scriptstyle 1'$}
\psfrag{1}{$\scriptstyle 1$}
\psfrag{2'}{$\scriptstyle 2'$}
\psfrag{2}{$\scriptstyle 2$}
\psfrag{3'}{$\scriptstyle 3'$}
\psfrag{3}{$\scriptstyle 3$}
\psfrag{4'}{$\scriptstyle 4'$}
\psfrag{4}{$\scriptstyle 4$}
\psfrag{5'}{$\scriptstyle 5'$}
\psfrag{5}{$\scriptstyle 5$}
\psfrag{6'}{$\scriptstyle 6'$}
\psfrag{6}{$\scriptstyle 6$}
\includegraphics[width=5cm]{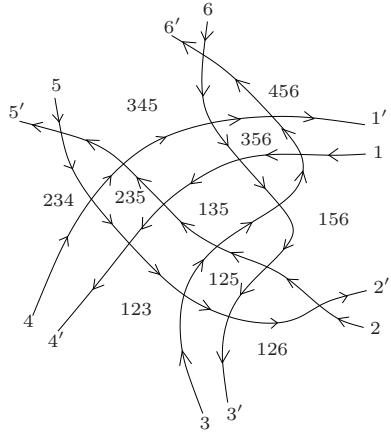}
\caption{A Postnikov diagram for $\Gr_{3,6}$.}
\label{fig:ex36}
\end{figure}

\begin{figure}
\includegraphics[width=8cm]{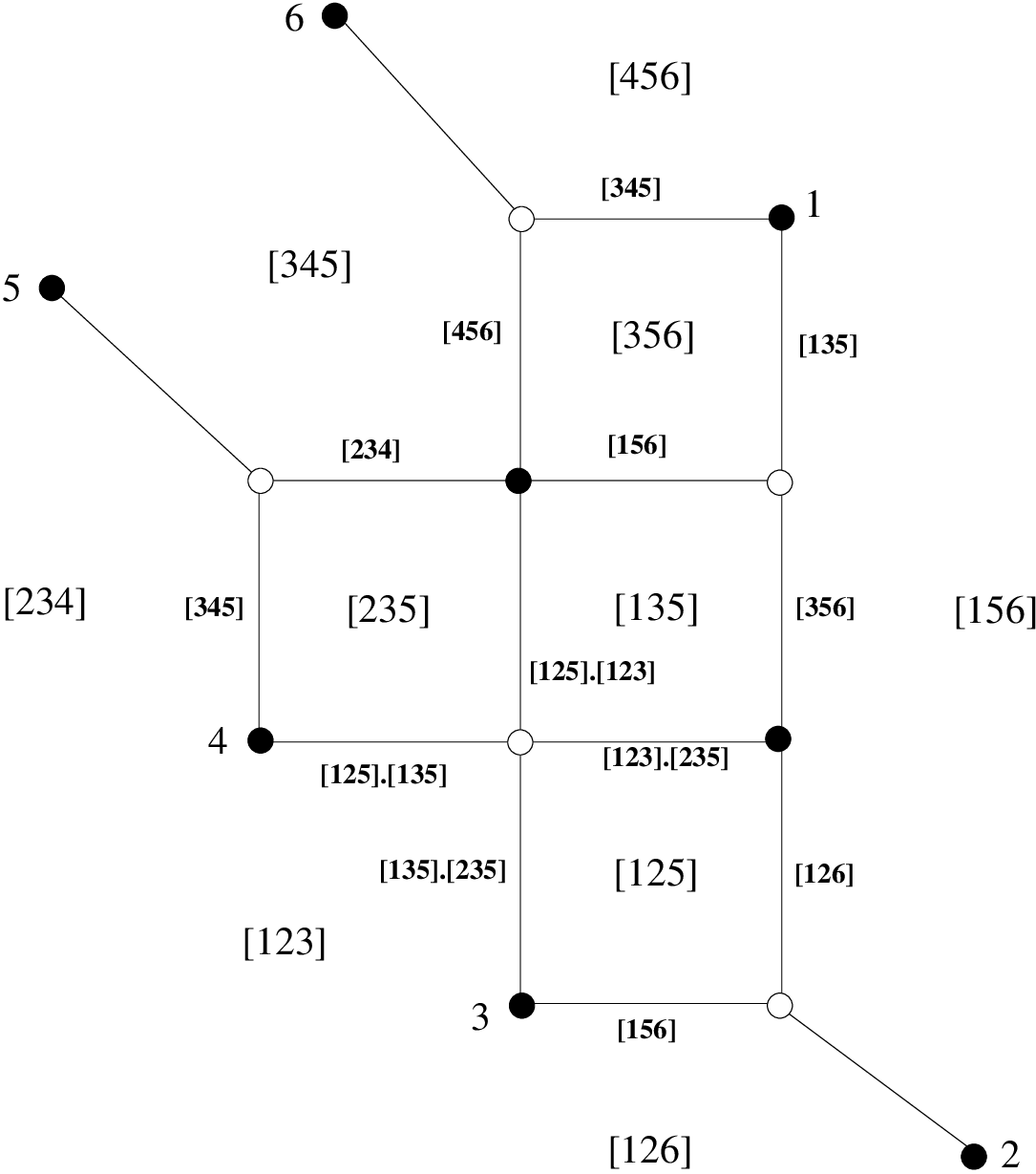}
\caption{The weighted bipartite graph of the arrangement in Figure~\ref{fig:ex36}.}
\label{fig:plabic36}
\end{figure}

\begin{figure}
\psfragscanon
\psfrag{m1}{$\scriptstyle \weight_{\dimerconfig_1}=\minor{123}\minor{156}\minor{156}\minor{235}\minor{345}\minor{345}$}
\psfrag{m2}{$\scriptstyle \weight_{\dimerconfig_2}=\minor{123}\minor{135}\minor{156}\minor{235}\minor{345}\minor{456}$}
\psfrag{m3}{$\scriptstyle \weight_{\dimerconfig_3}=\minor{123}\minor{125}\minor{156}\minor{345}\minor{345}\minor{356}$}
\psfrag{m4}{$\scriptstyle \weight_{\dimerconfig_4}=\minor{125}\minor{135}\minor{156}\minor{234}\minor{345}\minor{356}$}
\psfrag{m5}{$\scriptstyle \weight_{\dimerconfig_5}=\minor{126}\minor{135}\minor{156}\minor{235}\minor{345}\minor{345}$}
\psfrag{m6}{$\scriptstyle \weight_{\dimerconfig_6}=\minor{126}\minor{135}\minor{135}\minor{235}\minor{345}\minor{456}$}
\includegraphics[width=15.5cm]{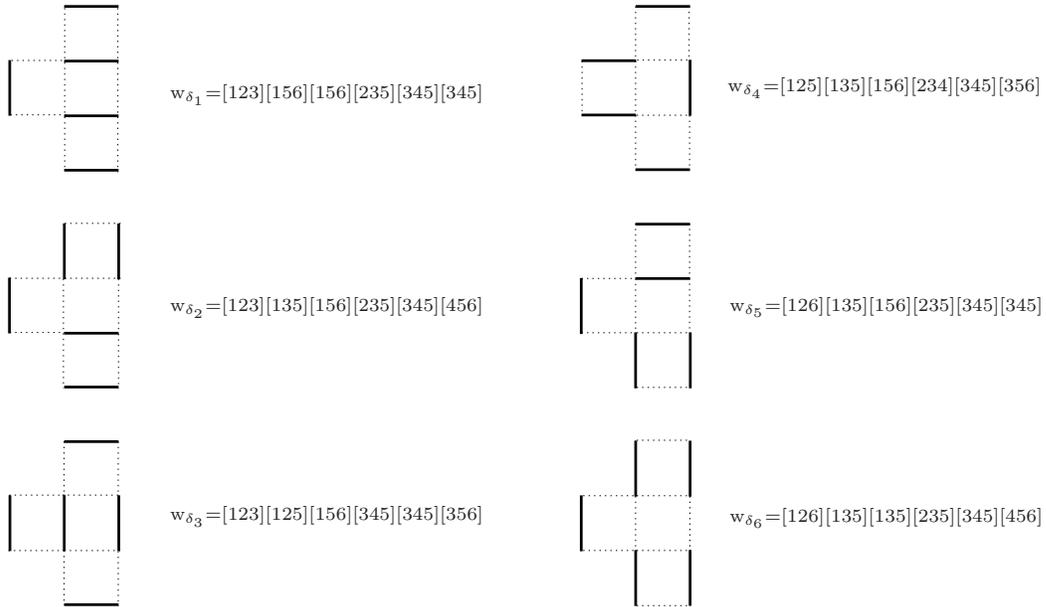}
\caption{The dimer configurations and corresponding monomials.}
\label{fig:matchings36}
\end{figure}

The corresponding dimer partition function, $\dimerpart_{G_{\postdiag}}(\{2,5,6\})$, is the sum of these.
Applying short Pl\"{u}cker relations we obtain:
\begin{align*}
\weight_{\dimerconfig_1}+\weight_{\dimerconfig_2} &= \minor{123}\minor{145}\minor{156}\minor{235}\minor{345}\minor{356}; \\
\weight_{\dimerconfig_3}+\weight_{\dimerconfig_4} &= \minor{125}\minor{134}\minor{156}\minor{235}\minor{345}\minor{356}; \\
\weight_{\dimerconfig_5}+\weight_{\dimerconfig_6} &= \minor{126}\minor{135}\minor{145}\minor{235}\minor{345}\minor{356}.
\end{align*}
Applying a short Pl\"{u}cker relation to $(\weight_{\dimerconfig_1}+\weight_{\dimerconfig_2})+(\weight_{\dimerconfig_3}+\weight_{\dimerconfig_4})$ and
to the sum of this and $\weight_{\dimerconfig_5}+\weight_{\dimerconfig_6}$, we obtain:
\begin{align*}
\weight_{\dimerconfig_1}+\weight_{\dimerconfig_2}+\weight_{\dimerconfig_3}+\weight_{\dimerconfig_4}=\minor{124}\minor{135}\minor{156}\minor{235}\minor{345}\minor{356}; \\
\dimerpart_{G_{\postdiag}}(\{2,5,6\})=\weight_{\dimerconfig_1}+\weight_{\dimerconfig_2}+\weight_{\dimerconfig_3}+
\weight_{\dimerconfig_4}+\weight_{\dimerconfig_5}+\weight_{\dimerconfig_6}=\minor{125}\minor{135}\minor{146}\minor{235}\minor{345}\minor{356}.
\end{align*}
and hence
$$\uu_{G_{\postdiag}}(\{2,5,6\})=\frac{\dimerpart_{G_{\postdiag}}(\{2,5,6\})}{\minor{125}\minor{135}\minor{235}\minor{356}}=\minor{146}\minor{345},$$
which coincides with $\overleftarrow{[256]}$ as stated by
Proposition~\ref{p:twistcomputation}. We have thus demonstrated
Theorem~\ref{t:mainresult} in this case.

\section{Maximal green sequences}
\label{s:green}
We have seen (Lemma~\ref{l:rkntwist}) that the twist of the cluster in the seed
corresponding to $\R_{k,n}$ coincides with the cluster in the seed corresponding to
$\R^*_{n-k,n}$, up to coefficients.
In this section we will show that up to coefficients and
a power of $\sigma$, this twisting can be realised via a maximal green
sequence. Such sequences arose independently in~\cite{keller11}, in the context of cluster algebras and
in~\cite{ACCERV14,CCV} (see also~\cite[\S3]{BDP14}),
in the context of the spectrum of Bogomol\'{n}yi--Prasad--Sommerfield (BPS) states.
We shall assume first that $k\not=1,2,n-1,n-2$;
see the proof of Theorem~\ref{t:greensummary} below
for a discussion of these cases.
For convenience, we shall denote the inverse of the map
$\sigma$ (see Section~\ref{s:twist}) by $\varrho$.

Our aim will be to define a maximal green sequence which starts at the seed corresponding to $\R_{k,n}$ and ends at a seed whose quiver is isomorphic to $\R^*_{n-k,n}$ and has the property that
applying $\sigma^k$ to its cluster gives the cluster
associated to $\R_{k,n}$. To do this, we will first 
introduce a compact notation for the quivers appearing in 
the maximal green sequence.

The lattice $\mathbb{Z}^2$ induces a tiling of
$\mathbb{R}^2$ by unit square tiles. We denote the tile with
$(x,y)\in \mathbb{Z}^2$ in its lower left corner by $U_{x,y}$.
Let $R$ be the rectangle in $\mathbb{R}^2$ with corners $(1,1)$, $(k-1,1)$, $(1,n-k-1)$ and $(k-1,n-k-1)$. Then $R$ contains the tiles $U_{x,y}$ for $1\leq x\leq k-2$
and $1\leq y\leq n-k-2$.
Let $\widetilde{R}$ be the larger rectangle whose vertices are
$(0,0)$, $(k,0)$, $(0,n-k)$ and $(k,n-k)$.
We call the tiles in $\widetilde{R}\setminus R$ \emph{boundary tiles}.
We define $R_{\mathbb{Z}}=R\cap \mathbb{Z}^2$
and $\widetilde{R}_{\mathbb{Z}}=\widetilde{R}\cap \mathbb{Z}^2$.

Next, we consider the four-vertex quivers shown in
Figure~\ref{fig:fourvertex}.

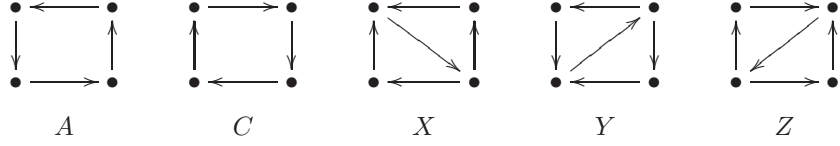
\begin{figure}
$$
\xymatrix@=6pt{
\bullet \ar[dd] && \bullet \ar[ll] \\ \\
\bullet \ar[rr] && \bullet \ar[uu] \\
& A
}
\quad\quad
\xymatrix@=6pt{
\bullet \ar[rr] && \bullet \ar[dd] \\ \\
\bullet \ar[uu] && \bullet \ar[ll] \\
& C
}
\quad\quad
\xymatrix@=6pt{
\bullet \ar[ddrr] && \bullet \ar[ll] \\ \\
\bullet \ar[uu] && \bullet \ar[ll] \ar[uu] \\
& X
}
\quad\quad
\xymatrix@=6pt{
\bullet \ar[dd] && \bullet \ar[ll] \ar[dd] \\ \\
\bullet \ar[uurr] && \bullet \ar[ll] \\
& Y
}
\quad\quad
\xymatrix@=6pt{
\bullet \ar[rr] && \bullet \ar[ddll] \\ \\
\bullet \ar[uu] \ar[rr] && \bullet \ar[uu] \\
& Z
}$$
\caption{Four-vertex quivers}
\label{fig:fourvertex}
\end{figure}

\begin{defn}
We will consider labellings $L$ of the unit square tiles
in $\widetilde{R}$ satisfying the following conditions:
\begin{enumerate}
\item[(a)] Each tile in $R$ is labelled by one $A,C,X,Y,Z$.
\item[(b)] For any pair of tiles in $R$ which
share an edge, the corresponding four-vertex quivers (with vertices at the corners of the tile) have arrows in the same direction along the common boundary.
\item[(c)] $n$ of the boundary tiles are labelled
$1,2,\ldots ,n$, in order clockwise around the boundary.
\item[(d)] The four corner boundary tiles are always labelled.
\end{enumerate}
\end{defn}

See Figure~\ref{fig:examplelabelling} for an example of
a labelling.

We will use the notation $\mathcal{T}_i$ for the boundary
tile labelled $i$.
For each $i\in \{1,\ldots ,n\}$ let $v_i$ be the external corner
of $\widetilde{R}$ on the boundary of $\mathcal{T}_i$ if
$\mathcal{T}_i$ is a corner tile of $\widetilde{R}$; otherwise we take $v_i$  to be the midpoint of the boundary edge of
$\mathcal{T}_i$ which is part of the boundary of $\widetilde{R}$.

\begin{defn}
If $L$ is a labelling as above, we set $Q(L)$ to be
the quiver with vertices $R_{\mathbb{Z}}\cup \{v_1,\ldots ,v_n\}$ and arrows given by the following:
\begin{enumerate}
\item[(a)] The full subquiver of $Q(L)$ on
the vertices of any tile in $R$ is given by its label
as in Figure~\ref{fig:fourvertex}.
\item[(b)] For each corner tile $\mathcal{T}_i$,
there is an arrow between $v_i$ and the unique
vertex in $\mathcal{T}_i\cap R$, oriented towards $v_i$ if
$v_i=(0,0)$ or $(k,n-k)$, and away from $v_i$ otherwise.
\item[(c)] For each non-corner boundary tile
$\mathcal{T}_i$, there are arrows between $v_i$ and the two vertices in $\mathcal{T}_i\cap R_{\mathbb{Z}}$, oriented in such a way as to create an oriented $3$-cycle on $v_i$ and these two vertices.
\end{enumerate}
\end{defn}

See Figure~\ref{fig:examplelabelling} for an example of
the quiver $Q(L)$ associated to a labelling.

\begin{figure}
$$
\vcenter{
\xymatrix@=12pt{
\ar@{}[dr] |{{\displaystyle 9}} & \ar@{}[dr] |{{\displaystyle 1}} & \ar@{}[dr] |{{\displaystyle 2}} & \ar@{}[dr] |{{\displaystyle 3}}\\
& \ar@{}[dr] |{{\displaystyle A}} \ar@{-}+0;[d]+0 \ar@{-}+0;[r]+0 &  \ar@{}[dr] |{{\displaystyle X}} \ar@{-}+0;[d]+0 \ar@{-}+0;[r]+0  & \ar@{-}+0;[d]+0 & \\
\ar@{}[dr] |{{\displaystyle 8}} & \ar@{}[dr] |{{\displaystyle C}} \ar@{-}+0;[d]+0 \ar@{-}+0;[r]+0 &  \ar@{}[dr] |{{\displaystyle Y}} \ar@{-}+0;[d]+0 \ar@{-}+0;[r]+0  & \ar@{-}+0;[d]+0 & \\
\ar@{}[dr] |{{\displaystyle 7}} & \ar@{}[dr] |{{\displaystyle X}} \ar@{-}+0;[d]+0 \ar@{-}+0;[r]+0 &  \ar@{}[dr] |{{\displaystyle X}} \ar@{-}+0;[d]+0 \ar@{-}+0;[r]+0  & \ar@{-}+0;[d]+0 \ar@{}[dr] |{{\displaystyle 4}} \\
\ar@{}[dr] |{{\displaystyle 6}} & \ar@{-}+0;[r]+0  & \ar@{-}+0;[r]+0 & \ar@{}[dr] |{{\displaystyle 5}} & & \\
& & & & }}
\quad\quad\quad
\vcenter{\xymatrix@=1pt{
v_9 \ar[ddrr] && & v_1 \ar[ddr] && v_2 \ar[ddr] &&& v_3 \\ \\
&& \bullet \ar[dd] \ar[uur] && \bullet \ar[uur] \ar[ll] \ar[ddrr] && \bullet \ar[ll] \ar[uurr] \\
& \\
&& \bullet \ar[lld] \ar[rr] && \bullet \ar[dd] \ar[uu] && \bullet \ar[ll] \ar[uu] \ar[dd] \\
v_8 \ar[rrd] \\
&& \bullet \ar[lld] \ar[ddrr] \ar[uu] && \bullet \ar[ll]\ar[uurr] \ar[ddrr] && \bullet \ar[ll] \ar[drr] \\
v_7 \ar[rrd] &&&&&&&& v_4 \ar[dll] \\
&& \bullet \ar[ddll] \ar[uu] && \bullet \ar[ll] \ar[uu] && \bullet \ar[ll] \ar[uu] \\ \\
v_6 && && && && v_5 \ar[uull]
}}$$
\caption{Example of a labelling, $L$ (left) and
the corresponding quiver $Q(L)$ (right), in the case $k=4$, $n=9$.}
\label{fig:examplelabelling}
\end{figure}
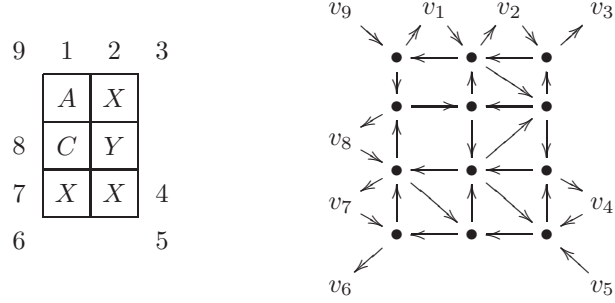

Next, we give a collection of local mutation rules which will allow us to
mutate quivers given by labellings directly, without decoding. 

\begin{lemma} \label{l:localrules}
Let $L$ be a labelling and $v\in R_{\mathbb{Z}}$ a vertex with the property
that the labels of the tiles around $v$ are as in one of the initial
diagrams in a rule in Figure~\ref{fig:localrules}. Then $\mu_v(Q(L))=Q(L')$ where $L'$ is the
labelling $L$ in which the tiles around $v$ have been changed according to the rule in Figure~\ref{fig:localrules}. Furthermore, the vertices $v_1,\ldots ,v_n$ in $Q(L)$ correspond to the vertices $v_1,\ldots ,v_n$ in $Q(L')$.
\end{lemma}
\begin{proof}
This is a simple case-by-case check, noting that a path
of length $2$ through $v$ in $Q(L)$ must start and end at vertices of tiles incident with $v$.
\end{proof}

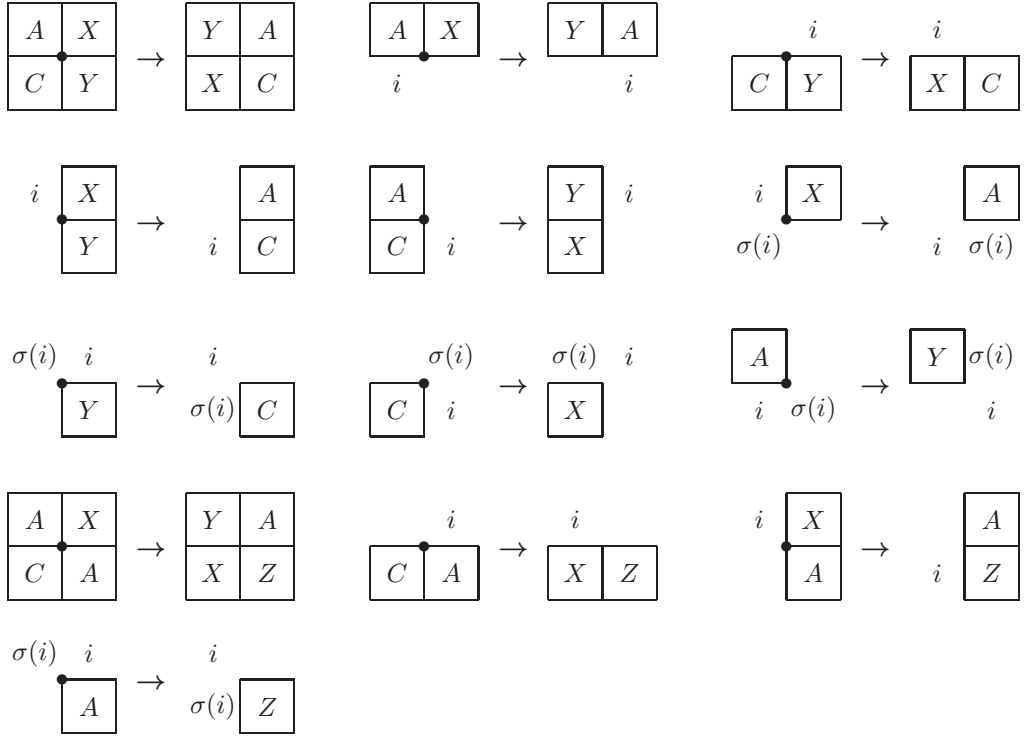
\begin{figure}
\begin{align*}
\vcenter{
\xymatrix@=12pt{
\ar@{}[dr]|{{\displaystyle A}} \ar@{-}+0;[r]+0 \ar@{-}+0;[d]+0
& \ar@{}[dr] |{{\displaystyle X}} \ar@{-}+0;[r]+0 \ar@{-}+0;[d]+0
& \ar@{-}+0;[d]+0 \\
\ar@{}[dr]|{{\displaystyle C}} \ar@{-}+0;[r]+0 \ar@{-}+0;[d]+0
& \bullet \ar@{}[dr]|{{\displaystyle Y}} \ar@{-}+0;[r]+0 \ar@{-}+0;[d]+0
& \ar@{-}+0;[d]+0
\\
\ar@{-}+0;[r]+0
& \ar@{-}+0;[r]+0
&
}}
\, &\boldsymbol{\rightarrow} \,
\vcenter{
\xymatrix@=12pt{
\ar@{}[dr]|{{\displaystyle Y}} \ar@{-}+0;[r]+0 \ar@{-}+0;[d]+0
& \ar@{}[dr] |{{\displaystyle A}} \ar@{-}+0;[r]+0 \ar@{-}+0;[d]+0
& \ar@{-}+0;[d]+0 \\
\ar@{}[dr]|{{\displaystyle X}} \ar@{-}+0;[r]+0 \ar@{-}+0;[d]+0
& {\phantom{\bullet}} \ar@{}[dr]|{{\displaystyle C}} \ar@{-}+0;[r]+0 \ar@{-}+0;[d]+0
& \ar@{-}+0;[d]+0
\\
\ar@{-}+0;[r]+0
& \ar@{-}+0;[r]+0
&
}}
&
\vcenter{
\xymatrix@=12pt{
\ar@{}[dr]|{{\displaystyle A}} \ar@{-}+0;[r]+0 \ar@{-}+0;[d]+0
& \ar@{}[dr] |{{\displaystyle X}} \ar@{-}+0;[r]+0 \ar@{-}+0;[d]+0
& \ar@{-}+0;[d]+0 \\
\ar@{}[dr]|{{\displaystyle i}} \ar@{-}+0;[r]+0 
& \bullet \ar@{-}+0;[r]+0 
& 
\\
& 
&
}}
\, &\boldsymbol{\rightarrow} \,
\vcenter{
\xymatrix@=12pt{
\ar@{}[dr]|{{\displaystyle Y}} \ar@{-}+0;[r]+0 \ar@{-}+0;[d]+0
& \ar@{}[dr] |{{\displaystyle A}} \ar@{-}+0;[r]+0 \ar@{-}+0;[d]+0
& \ar@{-}+0;[d]+0 \\
\ar@{-}+0;[r]+0 
& {\phantom{\bullet}} \ar@{}[dr]|{{\displaystyle i}} \ar@{-}+0;[r]+0 
&
\\
&
&
}}
&
\vcenter{
\xymatrix@=12pt{
& \ar@{}[dr] |{{\displaystyle i}} 
& \\
\ar@{}[dr]|{{\displaystyle C}} \ar@{-}+0;[r]+0 \ar@{-}+0;[d]+0
& \bullet \ar@{}[dr]|{{\displaystyle Y}} \ar@{-}+0;[r]+0 \ar@{-}+0;[d]+0
& \ar@{-}+0;[d]+0
\\
\ar@{-}+0;[r]+0
& \ar@{-}+0;[r]+0
&
}}
\, &\boldsymbol{\rightarrow} \,
\vcenter{
\xymatrix@=12pt{
\ar@{}[dr]|{{\displaystyle i}} 
& 
& \\
\ar@{}[dr]|{{\displaystyle X}} \ar@{-}+0;[r]+0 \ar@{-}+0;[d]+0
& {\phantom{\bullet}} \ar@{}[dr]|{{\displaystyle C}} \ar@{-}+0;[r]+0 \ar@{-}+0;[d]+0
& \ar@{-}+0;[d]+0
\\
\ar@{-}+0;[r]+0
& \ar@{-}+0;[r]+0
&
}}
\\[0.4cm]
\vcenter{
\xymatrix@=12pt{
\ar@{}[dr]|{{\displaystyle i}} 
& \ar@{}[dr] |{{\displaystyle X}} \ar@{-}+0;[r]+0 \ar@{-}+0;[d]+0
& \ar@{-}+0;[d]+0 \\
& \bullet \ar@{}[dr]|{{\displaystyle Y}} \ar@{-}+0;[r]+0 \ar@{-}+0;[d]+0
& \ar@{-}+0;[d]+0
\\
& \ar@{-}+0;[r]+0
&
}}
\, &\boldsymbol{\rightarrow} \,
\vcenter{
\xymatrix@=12pt{
& \ar@{}[dr] |{{\displaystyle A}} \ar@{-}+0;[r]+0 \ar@{-}+0;[d]+0
& \ar@{-}+0;[d]+0 \\
\ar@{}[dr]|{{\displaystyle i}} 
& {\phantom{\bullet}} \ar@{}[dr]|{{\displaystyle C}} \ar@{-}+0;[r]+0 \ar@{-}+0;[d]+0
& \ar@{-}+0;[d]+0
\\
& \ar@{-}+0;[r]+0
&
}}
&
\vcenter{
\xymatrix@=12pt{
\ar@{}[dr]|{{\displaystyle A}} \ar@{-}+0;[r]+0 \ar@{-}+0;[d]+0
& \ar@{-}+0;[d]+0
& \\
\ar@{}[dr]|{{\displaystyle C}} \ar@{-}+0;[r]+0 \ar@{-}+0;[d]+0
& \bullet \ar@{}[dr]|{{\displaystyle i}} \ar@{-}+0;[d]+0
& 
\\
\ar@{-}+0;[r]+0
& 
&
}}
\, &\boldsymbol{\rightarrow} \,
\vcenter{
\xymatrix@=12pt{
\ar@{}[dr]|{{\displaystyle Y}} \ar@{-}+0;[r]+0 \ar@{-}+0;[d]+0
& \ar@{}[dr] |{{\displaystyle i}} \ar@{-}+0;[d]+0
& \\
\ar@{}[dr]|{{\displaystyle X}} \ar@{-}+0;[r]+0 \ar@{-}+0;[d]+0
& {\phantom{\bullet}} \ar@{-}+0;[d]+0
& 
\\
\ar@{-}+0;[r]+0
& 
&
}}
&
\vcenter{
\xymatrix@=12pt{
\ar@{}[dr]|{{\displaystyle i}} 
& \ar@{}[dr] |{{\displaystyle X}} \ar@{-}+0;[r]+0 \ar@{-}+0;[d]+0
& \ar@{-}+0;[d]+0 \\
\ar@{}[dr]|{{\displaystyle \sigma(i)}} 
& \bullet \ar@{-}+0;[r]+0 
& 
\\
& 
&
}}
\, &\boldsymbol{\rightarrow} \,
\vcenter{
\xymatrix@=12pt{
& \ar@{}[dr] |{{\displaystyle A}} \ar@{-}+0;[r]+0 \ar@{-}+0;[d]+0
& \ar@{-}+0;[d]+0 \\
\ar@{}[dr]|{{\displaystyle i}} 
& {\phantom{\bullet}} \ar@{}[dr]|{{\displaystyle \sigma(i)}} \ar@{-}+0;[r]+0 
& 
\\
& 
&
}}
\\[0.4cm]
\vcenter{
\xymatrix@=12pt{
\ar@{}[dr]|{{\displaystyle \sigma(i)}}
& \ar@{}[dr] |{{\displaystyle i}}
& \\
& \bullet \ar@{}[dr]|{{\displaystyle Y}} \ar@{-}+0;[r]+0 \ar@{-}+0;[d]+0
& \ar@{-}+0;[d]+0
\\
& \ar@{-}+0;[r]+0
& 
}}
\, &\boldsymbol{\rightarrow} \,
\vcenter{
\xymatrix@=12pt{
\ar@{}[dr]|{{\displaystyle i}} 
& 
& \\
\ar@{}[dr]|{{\displaystyle \sigma(i)}}
& {\phantom{\bullet}} \ar@{}[dr]|{{\displaystyle C}} \ar@{-}+0;[r]+0 \ar@{-}+0;[d]+0
& \ar@{-}+0;[d]+0
\\
& \ar@{-}+0;[r]+0
&
}}
&
\vcenter{
\xymatrix@=12pt{
& \ar@{}[dr] |{{\displaystyle \sigma(i)}}
& \\
\ar@{}[dr]|{{\displaystyle C}} \ar@{-}+0;[r]+0 \ar@{-}+0;[d]+0
& \bullet \ar@{}[dr]|{{\displaystyle i}} \ar@{-}+0;[d]+0
&
\\
\ar@{-}+0;[r]+0
&
&
}}
\, &\boldsymbol{\rightarrow} \,
\vcenter{
\xymatrix@=12pt{
\ar@{}[dr]|{{\displaystyle \sigma(i)}}
& \ar@{}[dr] |{{\displaystyle i}}
& \\
\ar@{}[dr]|{{\displaystyle X}} \ar@{-}+0;[r]+0 \ar@{-}+0;[d]+0
& {\phantom{\bullet}} \ar@{-}+0;[d]+0
&
\\
\ar@{-}+0;[r]+0
& 
&
}}
&
\vcenter{
\xymatrix@=12pt{
\ar@{}[dr]|{{\displaystyle A}} \ar@{-}+0;[r]+0 \ar@{-}+0;[d]+0
& \ar@{-}+0;[d]+0
& \\
\ar@{}[dr]|{{\displaystyle i}} \ar@{-}+0;[r]+0
& \bullet \ar@{}[dr]|{{\displaystyle \sigma(i)}}
&
\\
&
&
}}
\, &\boldsymbol{\rightarrow} \,
\vcenter{
\xymatrix@=12pt{
\ar@{}[dr]|{{\displaystyle Y}} \ar@{-}+0;[r]+0 \ar@{-}+0;[d]+0
& \ar@{}[dr] |{{\displaystyle \sigma(i)}} \ar@{-}+0;[d]+0
& \\
\ar@{-}+0;[r]+0
& {\phantom{\bullet}} \ar@{}[dr]|{{\displaystyle i}}
&
\\
&
&
}}
\\[0.4cm]
\vcenter{
\xymatrix@=12pt{
\ar@{}[dr]|{{\displaystyle A}} \ar@{-}+0;[r]+0 \ar@{-}+0;[d]+0
& \ar@{}[dr] |{{\displaystyle X}} \ar@{-}+0;[r]+0 \ar@{-}+0;[d]+0
& \ar@{-}+0;[d]+0 \\
\ar@{}[dr]|{{\displaystyle C}} \ar@{-}+0;[r]+0 \ar@{-}+0;[d]+0
& \bullet \ar@{}[dr]|{{\displaystyle A}} \ar@{-}+0;[r]+0 \ar@{-}+0;[d]+0
& \ar@{-}+0;[d]+0
\\
\ar@{-}+0;[r]+0
& \ar@{-}+0;[r]+0
&
}}
\, &\boldsymbol{\rightarrow} \,
\vcenter{
\xymatrix@=12pt{
\ar@{}[dr]|{{\displaystyle Y}} \ar@{-}+0;[r]+0 \ar@{-}+0;[d]+0
& \ar@{}[dr] |{{\displaystyle A}} \ar@{-}+0;[r]+0 \ar@{-}+0;[d]+0
& \ar@{-}+0;[d]+0 \\
\ar@{}[dr]|{{\displaystyle X}} \ar@{-}+0;[r]+0 \ar@{-}+0;[d]+0
& {\phantom{\bullet}} \ar@{}[dr]|{{\displaystyle Z}} \ar@{-}+0;[r]+0 \ar@{-}+0;[d]+0
& \ar@{-}+0;[d]+0
\\
\ar@{-}+0;[r]+0
& \ar@{-}+0;[r]+0
&
}}
&
\vcenter{
\xymatrix@=12pt{
& \ar@{}[dr] |{{\displaystyle i}}
& \\
\ar@{}[dr]|{{\displaystyle C}} \ar@{-}+0;[r]+0 \ar@{-}+0;[d]+0
& \bullet \ar@{}[dr]|{{\displaystyle A}} \ar@{-}+0;[r]+0 \ar@{-}+0;[d]+0
& \ar@{-}+0;[d]+0
\\
\ar@{-}+0;[r]+0
& \ar@{-}+0;[r]+0
&
}}
\, &\boldsymbol{\rightarrow} \,
\vcenter{
\xymatrix@=12pt{
\ar@{}[dr]|{{\displaystyle i}}
&
& \\
\ar@{}[dr]|{{\displaystyle X}} \ar@{-}+0;[r]+0 \ar@{-}+0;[d]+0
& {\phantom{\bullet}} \ar@{}[dr]|{{\displaystyle Z}} \ar@{-}+0;[r]+0 \ar@{-}+0;[d]+0
& \ar@{-}+0;[d]+0
\\
\ar@{-}+0;[r]+0
& \ar@{-}+0;[r]+0
&
}}
&
\vcenter{
\xymatrix@=12pt{
\ar@{}[dr]|{{\displaystyle i}}
& \ar@{}[dr] |{{\displaystyle X}} \ar@{-}+0;[r]+0 \ar@{-}+0;[d]+0
& \ar@{-}+0;[d]+0 \\
& \bullet \ar@{}[dr]|{{\displaystyle A}} \ar@{-}+0;[r]+0 \ar@{-}+0;[d]+0
& \ar@{-}+0;[d]+0
\\
& \ar@{-}+0;[r]+0
&
}}
\, &\boldsymbol{\rightarrow} \,
\vcenter{
\xymatrix@=12pt{
& \ar@{}[dr] |{{\displaystyle A}} \ar@{-}+0;[r]+0 \ar@{-}+0;[d]+0
& \ar@{-}+0;[d]+0 \\
\ar@{}[dr]|{{\displaystyle i}}
& {\phantom{\bullet}} \ar@{}[dr]|{{\displaystyle Z}} \ar@{-}+0;[r]+0 \ar@{-}+0;[d]+0
& \ar@{-}+0;[d]+0
\\
& \ar@{-}+0;[r]+0
&
}}
\\
\vcenter{
\xymatrix@=12pt{
\ar@{}[dr]|{{\displaystyle \sigma(i)}}
& \ar@{}[dr] |{{\displaystyle i}}
& \\
& \bullet \ar@{}[dr]|{{\displaystyle A}} \ar@{-}+0;[r]+0 \ar@{-}+0;[d]+0
& \ar@{-}+0;[d]+0
\\
& \ar@{-}+0;[r]+0
&
}}
\, &\boldsymbol{\rightarrow} \,
\vcenter{
\xymatrix@=12pt{
\ar@{}[dr]|{{\displaystyle i}}
&
& \\
\ar@{}[dr]|{{\displaystyle \sigma(i)}}
& {\phantom{\bullet}} \ar@{}[dr]|{{\displaystyle Z}} \ar@{-}+0;[r]+0 \ar@{-}+0;[d]+0
& \ar@{-}+0;[d]+0
\\
& \ar@{-}+0;[r]+0
&
}}
\end{align*}
\caption{Local mutation rules for labellings}
\label{fig:localrules}
\end{figure}

Let $<$ denote the reverse lexicographic ordering
on $R_{\mathbb{Z}}$, i.e. we set $(x,y)<(x',y')$ if 
$y<y'$ or if $y=y'$ and $x<x'$.
Note that $(1,1)$ is the unique minimal element in this ordering and $(k-1,n-k-1)$ is the unique maximal element.
If $(x,y)\not=(k-1,n-k-1)$, we write
$(x,y)^+$ for the successor of $(x,y)$ in this ordering. We
have:
$$(x,y)^+=\begin{cases}
(x+1,y), & \text{if }1\leq x\leq k-2; \\
(1,y+1), & \text{if }x=k-1.
\end{cases}$$

Let
$$\Omega=\{(x,y;l)\in\mathbb{Z}^2\times \mathbb{Z}\,:\,
1\leq l\leq k-1,\ 1\leq x\leq k-l,\ 1\leq y\leq n-k-1\}\sqcup \{\top\}.
$$
We shall use this set to parametrize the seeds appearing in the maximal green
sequence. Note that the pair $(x,y)$ appearing in an element
$(x,y;l)$ of $\Omega$ always lies in $R_{\mathbb{Z}}$.

We denote the reverse lexicographic ordering on $\Omega$ by $<$, extended
by setting $\top$ to be a maximum element. We shall
denote the successor of $\omega\in \Omega\setminus \{\top\}$ in this
ordering by $\omega^+$. We have, for $\omega=(x,y;l)\in \Omega\setminus \{\top\}$:
$$\omega^+=\begin{cases}
((x,y)^+;l) & \text{if } (x,y)\not=(k-l,n-k-1); \\
(1,1;l+1) & \text{if } (x,y)=(k-l,n-k-1),\ l\not=k-1; \\
\top & \text{if } (x,y;l)=(1,n-k-1;k-1).
\end{cases}
$$
We denote the smallest element, $(1,1;1)$, of $\Omega$, by
$\perp$.
We will associate a seed to each element of $\Omega$ with the
property that the mutation at $(x,y)$ of the seed associated to $(x,y;l)$ is the seed associated to $(x,y;l)^+$.

\begin{defn} \label{d:alpha}
Let $\text{pr}$ denote the map from $\Omega\setminus \{\top\}$ to $R_{\mathbb{Z}}$ taking $(x,y;l)$ to $(x,y)$.
We define $\boldsymbol{\alpha}$ to be the sequence of elements of $R_{\mathbb{Z}}$ obtained by applying $\text{pr}$ to each of
the elements of $\Omega\setminus \{\top\}$, written in the total ordering $<$.
\end{defn}

The sequence $\boldsymbol{\alpha}$ can be regarded as scanning the
rows of vertices in $R_{\mathbb{Z}}$ from left to right,
starting at the bottom left and ending at the top right
(giving a first `page' of mutations),
then a second page repeating the process but omitting the last  
vertex in each row, then repeating this but omitting the last 
two vertices in each row, and so on, until a final $(k-1)$st 
page in which only the first vertex of each row is mutated. 
The index $l$ indicates the page number.

For example, if $k=4$ and $n=9$, then the set $\Omega\setminus \{\top\}$ written in order is:
\begin{align*}
&(1,1;1),(2,1;1),(3,1;1),(1,2;1),(2,2;1),(3,2;1),(1,3;1),(2,3;1),(3,3;1), \\
&(1,4;1),(2,4;1),(3,4;1),(1,1;2),(2,1;2),(1,2;2),(2,2;2),(1,3;2),(2,3;2),\\
&(1,4;2),(2,4;2),(1,1;3),(1,2;3),(1,3;3),(1,4;3).\end{align*}
Hence
\begin{align*}
\boldsymbol{\alpha} &=(1,1),(2,1),(3,1),(1,2),(2,2),(3,2),(1,3),(2,3),(3,3),(1,4),(2,4),(3,4),(1,1),(2,1),(1,2),\\
&(2,2),(1,3),(2,3),
(1,4),(2,4),(1,1),(1,2),(1,3),(1,4).
\end{align*}

\begin{defn} \label{d:mutations}
Fix $\omega\in \Omega$. We will associate to
$\omega$ a seed $(\xx_{\omega},\widetilde{Q}_{\omega})$,
Firstly, we define a labelling $L_{\omega}$ as follows.
If $\omega=(x,y;l)$, then
for $1\leq i\leq k-2$ and $1\leq j\leq n-k-2$, we decorate the
unit square $U_{ij}$ in $R$ according to the following rules:
\begin{enumerate}
\item[(a)] $A$ if $(i,j)=(x-1,y)$ or $(i,j)=(k-l,y-1)$;
\item[(b)] $C$ if $(i,j)=(x-1,y-1)$;
\item[(c)] $Y$ if either:
\begin{itemize}
\item[(i)] $1\leq i\leq x-2$ and $j=y$, or
\item[(i)] $x\leq i\leq k-l-1$ and $j=y-1$;
\end{itemize}
\item[(d)] $Z$ if either:
\begin{itemize}
\item[(i)] $i=k-l$ and $1\leq j\leq y-2$, or
\item[(ii)] $i>k-l$;
\end{itemize}
\item[(e)] $X$, otherwise.
\end{enumerate}
If $\omega=\top$, we decorate each unit square $U_{ij}$ in $R$
with a $Z$.

The following table determines a collection of non-corner boundary unit
squares according to their location on the boundary of $\widetilde{R}$
(i.e. left hand side, right hand side, top or bottom) and
according to the label of their unique adjacent square in $R$
(a check mark indicates membership in the collection):
\vskip 0.2cm
\begin{center}
\begin{tabular}{|c|c|c|c|c|}
\hline
 & left hand side & right hand side & top & bottom \\
\hline
$A$ & $\times$ & $\times$ & $\checkmark$ & $\checkmark$ \\
\hline
$C$ & $\checkmark$ & $\checkmark$ & $\times$ & $\times$ \\
\hline
$X$ & $\checkmark$ & $\times$ & $\checkmark$ & $\times$ \\
\hline
$Y$ & $\times$ & $\checkmark$ & $\checkmark$ & $\times$ \\
\hline
$Z$ & $\checkmark$ & $\times$ & $\times$ & $\checkmark$ \\
\hline
\end{tabular}
\end{center}
\vskip 0.2cm
It is easy to check that there are $n-4$ squares in the above
collection. We
label the four corner squares and the boundary squares in this
collection $1,2,\ldots, n$ clockwise around
the boundary in such a way that the top right corner square of
$\widetilde{R}$ is labelled
$k-1$ if $\omega=(x,y;l)$ with $l=1$,
or is labelled $k$ if $\omega=(x,y;l)$ with $l>1$ or if $\omega=\top$.
This completes the definition of $L_{\omega}$.

Then, for each $\omega\in\Omega$, we set
$\widetilde{Q}_{\omega}=Q(L_{\omega})$
and we set $Q_{\omega}$ to be the full subquiver of $\widetilde{Q}_{\omega}$
on the vertices $R_{\mathbb{Z}}$.

Next, suppose that $\omega=(x,y;l)$ where
$1\leq l\leq k-1$, $1\leq x\leq k-l$, and $1\leq y\leq n-k-1$.
For $(i,j)\in R_{\mathbb{Z}}$ we define:
\begin{equation}
\label{e:rij}
r_{\omega}(i,j)=\begin{cases}
l, & i\leq k-l,\ (i,j)<(x,y); \\
l-1, & i\leq k-l,\ (i,j)\nless (x,y); \\
k-i, & i>k-l.
\end{cases}
\end{equation}
We also define $r_{\top}(i,j)=k-i$ for all $(i,j)\in R_{\mathbb{Z}}$.

We associate Pl\"{u}cker coordinates to the
vertices of $\widetilde{Q}_{\omega}$
by associating $\mcoeff{i}$ to the vertex $v_i$ for $1\leq i\leq n$ and
$[\varrho^{r_{\omega}(i,j)}(M_{k,n}(i,j))]$ to each vertex
$(i,j)\in R_{\mathbb{Z}}$.
We then define $\xx_{\omega}$
(respectively, $\mathbf{x}_{\omega}$) to be the set of Pl\"{u}cker coordinates
associated to the vertices of $\widetilde{Q}_{\omega}$ (respectively,
the vertices of $Q_{\omega}$).
\end{defn}

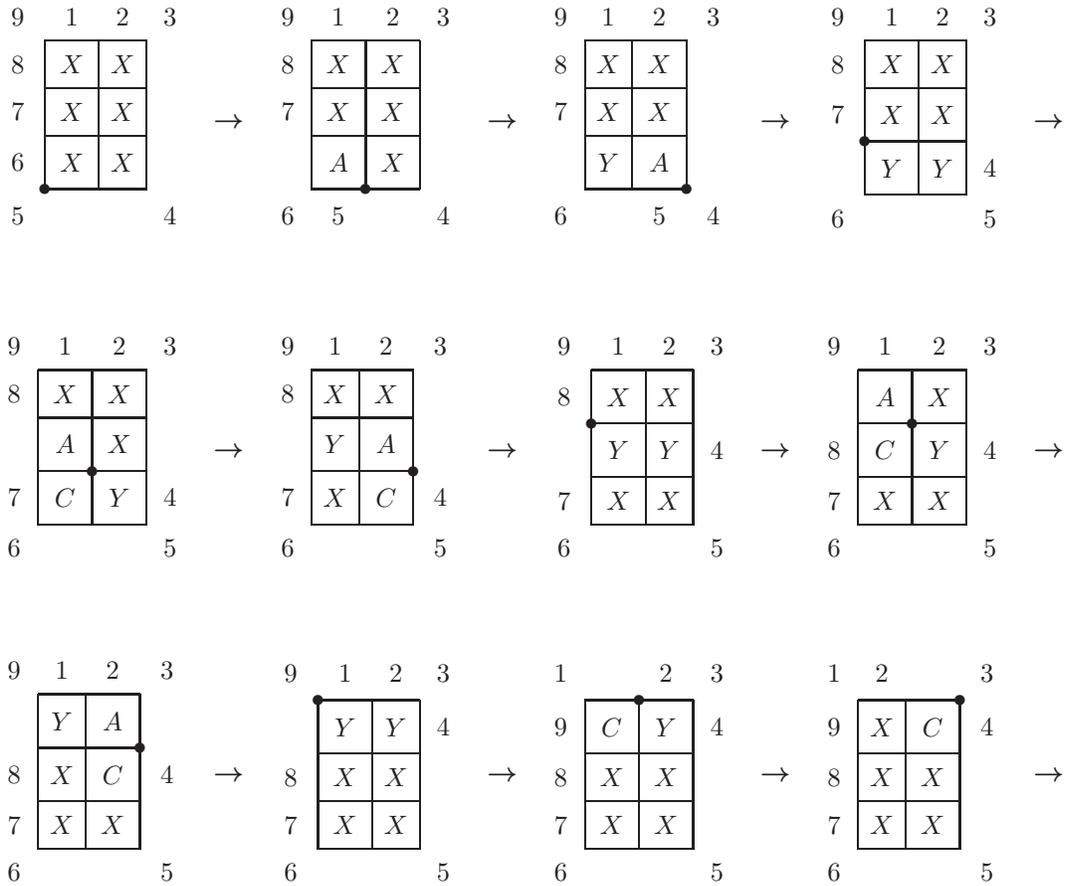
\begin{figure}
\[
\xymatrix@=12pt{
\ar@{}[dr] |{{\displaystyle 15}} &\ar@{}[dr] |{{\displaystyle 16}}  &\ar@{}[dr] |{{\displaystyle 17}}  &\ar@{}[dr] |{{\displaystyle 1}} 
&\ar@{}[dr] |{{\displaystyle 2}}  &\ar@{}[dr] |{{\displaystyle 3}} &\ar@{}[dr] |{{\displaystyle 4}} &\ar@{}[dr] |{{\displaystyle 5}} & 
&\ar@{}[dr] |{{\displaystyle 6}} & \\ 
\ar@{}[dr] |{{\displaystyle 14}}
&\ar@{}[dr] |{{\displaystyle X}}  \ar@{-}+0;[r]+0  \ar@{-}+0;[d]+0 &\ar@{}[dr] |{{\displaystyle X}}  \ar@{-}+0;[r]+0  \ar@{-}+0;[d]+0 &\ar@{}[dr] |{{\displaystyle X}}  \ar@{-}+0;[r]+0  \ar@{-}+0;[d]+0 &\ar@{}[dr] |{{\displaystyle X}}  \ar@{-}+0;[r]+0  \ar@{-}+0;[d]+0 &\ar@{}[dr] |{{\displaystyle X}}  \ar@{-}+0;[r]+0  \ar@{-}+0;[d]+0 &\ar@{}[dr] |{{\displaystyle X}}  \ar@{-}+0;[r]+0  \ar@{-}+0;[d]+0 
&\ar@{}[dr] |{{\displaystyle X}}  \ar@{-}+0;[r]+0  \ar@{-}+0;[d]+0 &\ar@{}[dr] |{{\displaystyle Z}}  \ar@{-}+0;[r]+0  \ar@{-}+0;[d]+0 
&\ar@{-}+0;[d]+0 & \\
&\ar@{}[dr] |{{\displaystyle Y}}  \ar@{-}+0;[r]+0  \ar@{-}+0;[d]+0 &\ar@{}[dr] |{{\displaystyle Y}}  \ar@{-}+0;[r]+0  \ar@{-}+0;[d]+0 &\ar@{}[dr] |{{\displaystyle A}}  \ar@{-}+0;[r]+0  \ar@{-}+0;[d]+0 &\ar@{}[dr] |{{\displaystyle X}}  \ar@{-}+0;[r]+0  \ar@{-}+0;[d]+0 &\ar@{}[dr] 
|{{\displaystyle X}}  \ar@{-}+0;[r]+0  \ar@{-}+0;[d]+0 &\ar@{}[dr] |{{\displaystyle X}}  \ar@{-}+0;[r]+0  \ar@{-}+0;[d]+0 
&\ar@{}[dr] |{{\displaystyle X}}  \ar@{-}+0;[r]+0  \ar@{-}+0;[d]+0 &\ar@{}[dr] |{{\displaystyle Z}}  \ar@{-}+0;[r]+0  \ar@{-}+0;[d]+0 
&\ar@{-}+0;[d]+0 & \\
\ar@{}[dr] |{{\displaystyle 13}}
&\ar@{}[dr] |{{\displaystyle X}}  \ar@{-}+0;[r]+0  \ar@{-}+0;[d]+0 &\ar@{}[dr] |{{\displaystyle X}}  \ar@{-}+0;[r]+0  \ar@{-}+0;[d]+0 &\ar@{}[dr] |{{\displaystyle C}}  \ar@{-}+0;[r]+0  \ar@{-}+0;[d]+0 &\bullet \ar@{}[dr] |{{\displaystyle Y}}  \ar@{-}+0;[r]+0  \ar@{-}+0;[d]+0 &\ar@{}[dr] |{{\displaystyle Y}}  \ar@{-}+0;[r]+0  \ar@{-}+0;[d]+0 &\ar@{}[dr] |{{\displaystyle Y}}  \ar@{-}+0;[r]+0  \ar@{-}+0;[d]+0 
&\ar@{}[dr] |{{\displaystyle A}}  \ar@{-}+0;[r]+0  \ar@{-}+0;[d]+0 &\ar@{}[dr] |{{\displaystyle Z}}  \ar@{-}+0;[r]+0  \ar@{-}+0;[d]+0 
&\ar@{-}+0;[d]+0 & \\ 
\ar@{}[dr] |{{\displaystyle 12}}
&\ar@{}[dr] |{{\displaystyle X}}  \ar@{-}+0;[r]+0  \ar@{-}+0;[d]+0 &\ar@{}[dr] |{{\displaystyle X}}  \ar@{-}+0;[r]+0  \ar@{-}+0;[d]+0 &\ar@{}[dr] |{{\displaystyle X}}  \ar@{-}+0;[r]+0  \ar@{-}+0;[d]+0 &\ar@{}[dr] |{{\displaystyle X}}  \ar@{-}+0;[r]+0  \ar@{-}+0;[d]+0 &\ar@{}[dr] |{{\displaystyle X}}  \ar@{-}+0;[r]+0  \ar@{-}+0;[d]+0 &\ar@{}[dr] |{{\displaystyle X}}  \ar@{-}+0;[r]+0  \ar@{-}+0;[d]+0 
&\ar@{}[dr] |{{\displaystyle Z}}  \ar@{-}+0;[r]+0  \ar@{-}+0;[d]+0 &\ar@{}[dr] |{{\displaystyle Z}}  \ar@{-}+0;[r]+0  \ar@{-}+0;[d]+0 
&\ar@{-}+0;[d]+0 & \\
\ar@{}[dr] |{{\displaystyle 11}}
&\ar@{}[dr] |{{\displaystyle X}}  \ar@{-}+0;[r]+0  \ar@{-}+0;[d]+0 &\ar@{}[dr] |{{\displaystyle X}}  \ar@{-}+0;[r]+0  \ar@{-}+0;[d]+0 &\ar@{}[dr] |{{\displaystyle X}}  \ar@{-}+0;[r]+0  \ar@{-}+0;[d]+0 &\ar@{}[dr] |{{\displaystyle X}}  \ar@{-}+0;[r]+0  \ar@{-}+0;[d]+0 &\ar@{}[dr] |{{\displaystyle X}}  \ar@{-}+0;[r]+0  \ar@{-}+0;[d]+0 &\ar@{}[dr] |{{\displaystyle X}}  \ar@{-}+0;[r]+0  \ar@{-}+0;[d]+0 
&\ar@{}[dr] |{{\displaystyle Z}}  \ar@{-}+0;[r]+0  \ar@{-}+0;[d]+0 &\ar@{}[dr] |{{\displaystyle Z}}  \ar@{-}+0;[r]+0  \ar@{-}+0;[d]+0 
&\ar@{-}+0;[d]+0 & \\
\ar@{}[dr] |{{\displaystyle 10}}
&\ar@{-}+0;[r]+0  &\ar@{-}+0;[r]+0  &\ar@{-}+0;[r]+0  &\ar@{-}+0;[r]+0  &\ar@{-}+0;[r]+0  
&\ar@{-}+0;[r]+0  &\ar@{}[dr] |{{\displaystyle 9}} \ar@{-}+0;[r]+0  &\ar@{}[dr] |{{\displaystyle 8}}  \ar@{-}+0;[r]+0  &\ar@{}[dr] |{{\displaystyle 7}}  & \\
& & & & & & & & & &}
\]
\caption{The labelling $L_{(4,4;3)}$ for $\Gr_{10,17}$.}
\label{f:examplelabelling}
\end{figure}

For an example of a labelling $L_{\omega}$, see Figure~\ref{f:examplelabelling}.
We also show the labelling $L_{\omega}$ for each $\omega\in \Omega$ in the case $\Gr_{4,9}$ below, with the vertex $(x,y)$ indicated by a dot when $\omega=(x,y;l)$.
See Figure~\ref{fig:perplabelling} for the quiver $Q_{\perp}$, and
Figure~\ref{fig:toplabelling} for the quiver $Q_{\top}$, in the case
$k=4$, $n=9$.

\[
\vcenter{
\xymatrix@=12pt{
\ar@{}[dr] |{{\displaystyle 9}} 
&\ar@{}[dr] |{{\displaystyle 1}} 
& \ar@{}[dr] |{{\displaystyle 2}} 
&\ar@{}[dr] |{{\displaystyle 3}}
& \\
\ar@{}[dr] |{{\displaystyle 8}} 
&\ar@{}[dr] |{{\displaystyle X}} \ar@{-}+0;[d]+0 \ar@{-}+0;[r]+0 
&\ar@{}[dr] |{{\displaystyle X}} \ar@{-}+0;[d]+0 \ar@{-}+0;[r]+0  
&\ar@{-}+0;[d]+0  
& \\
\ar@{}[dr] |{{\displaystyle 7}}
&\ar@{}[dr] |{{\displaystyle X}} \ar@{-}+0;[d]+0 \ar@{-}+0;[r]+0 
&\ar@{}[dr] |{{\displaystyle X}} \ar@{-}+0;[d]+0 \ar@{-}+0;[r]+0  
&\ar@{-}+0;[d]+0 
& \\
\ar@{}[dr] |{{\displaystyle 6}} 
&\ar@{}[dr] |{{\displaystyle X}} \ar@{-}+0;[d]+0 \ar@{-}+0;[r]+0 
&\ar@{}[dr] |{{\displaystyle X}} \ar@{-}+0;[d]+0 \ar@{-}+0;[r]+0  
&\ar@{-}+0;[d]+0  
& \\
\ar@{}[dr] |{{\displaystyle 5}} 
&\bullet  \ar@{-}+0;[r]+0  
&\ar@{-}+0;[r]+0  
&\ar@{}[dr] |{{\displaystyle 4}}  
& \\
& & & &}
}\, \boldsymbol{\rightarrow} \, \vcenter{
\xymatrix@=12pt{
\ar@{}[dr] |{{\displaystyle 9}} 
& \ar@{}[dr] |{{\displaystyle 1}} 
& \ar@{}[dr] |{{\displaystyle 2}} 
& \ar@{}[dr] |{{\displaystyle 3}}
& \\
\ar@{}[dr] |{{\displaystyle 8}} 
&\ar@{}[dr] |{{\displaystyle X}} \ar@{-}+0;[d]+0 \ar@{-}+0;[r]+0 
&\ar@{}[dr] |{{\displaystyle X}} \ar@{-}+0;[d]+0 \ar@{-}+0;[r]+0  
&\ar@{-}+0;[d]+0  
& \\
\ar@{}[dr] |{{\displaystyle 7}}
&\ar@{}[dr] |{{\displaystyle X}} \ar@{-}+0;[d]+0 \ar@{-}+0;[r]+0 
&\ar@{}[dr] |{{\displaystyle X}} \ar@{-}+0;[d]+0 \ar@{-}+0;[r]+0  
&\ar@{-}+0;[d]+0 
& \\
&\ar@{}[dr] |{{\displaystyle A}} \ar@{-}+0;[d]+0 \ar@{-}+0;[r]+0 
&\ar@{}[dr] |{{\displaystyle X}} \ar@{-}+0;[d]+0 \ar@{-}+0;[r]+0  
&\ar@{-}+0;[d]+0  
& \\
\ar@{}[dr] |{{\displaystyle 6}} 
&\ar@{-}+0;[r]+0   \ar@{}[dr] |{{\displaystyle 5}}
& \bullet \ar@{-}+0;[r]+0  
& \ar@{}[dr] |{{\displaystyle 4}}  
& \\
& & & &}
}\, \boldsymbol{\rightarrow} \, \vcenter{
\xymatrix@=12pt{
\ar@{}[dr] |{{\displaystyle 9}} 
& \ar@{}[dr] |{{\displaystyle 1}} 
& \ar@{}[dr] |{{\displaystyle 2}} 
& \ar@{}[dr] |{{\displaystyle 3}}
& \\
\ar@{}[dr] |{{\displaystyle 8}} 
&\ar@{}[dr] |{{\displaystyle X}} \ar@{-}+0;[d]+0 \ar@{-}+0;[r]+0 
&\ar@{}[dr] |{{\displaystyle X}} \ar@{-}+0;[d]+0 \ar@{-}+0;[r]+0  
&\ar@{-}+0;[d]+0  
& \\
\ar@{}[dr] |{{\displaystyle 7}}
&\ar@{}[dr] |{{\displaystyle X}} \ar@{-}+0;[d]+0 \ar@{-}+0;[r]+0 
&\ar@{}[dr] |{{\displaystyle X}} \ar@{-}+0;[d]+0 \ar@{-}+0;[r]+0  
&\ar@{-}+0;[d]+0 
& \\
&\ar@{}[dr] |{{\displaystyle Y}} \ar@{-}+0;[d]+0 \ar@{-}+0;[r]+0 
&\ar@{}[dr] |{{\displaystyle A}} \ar@{-}+0;[d]+0 \ar@{-}+0;[r]+0  
&\ar@{-}+0;[d]+0  
& \\
\ar@{}[dr] |{{\displaystyle 6}} 
&\ar@{-}+0;[r]+0  
&\ar@{-}+0;[r]+0  \ar@{}[dr] |{{\displaystyle 5}} 
&\bullet \ar@{}[dr] |{{\displaystyle 4}}  
& \\
& & & &}
}\, \boldsymbol{\rightarrow} \, \vcenter{
\xymatrix@=12pt{
\ar@{}[dr] |{{\displaystyle 9}} 
& \ar@{}[dr] |{{\displaystyle 1}} 
& \ar@{}[dr] |{{\displaystyle 2}} 
& \ar@{}[dr] |{{\displaystyle 3}}
& \\
\ar@{}[dr] |{{\displaystyle 8}} 
&\ar@{}[dr] |{{\displaystyle X}} \ar@{-}+0;[d]+0 \ar@{-}+0;[r]+0 
&\ar@{}[dr] |{{\displaystyle X}} \ar@{-}+0;[d]+0 \ar@{-}+0;[r]+0  
&\ar@{-}+0;[d]+0  
& \\
\ar@{}[dr] |{{\displaystyle 7}}
&\ar@{}[dr] |{{\displaystyle X}} \ar@{-}+0;[d]+0 \ar@{-}+0;[r]+0 
&\ar@{}[dr] |{{\displaystyle X}} \ar@{-}+0;[d]+0 \ar@{-}+0;[r]+0  
&\ar@{-}+0;[d]+0 
& \\
&\bullet \ar@{}[dr] |{{\displaystyle Y}} \ar@{-}+0;[d]+0 \ar@{-}+0;[r]+0 
&\ar@{}[dr] |{{\displaystyle Y}} \ar@{-}+0;[d]+0 \ar@{-}+0;[r]+0  
&\ar@{-}+0;[d]+0  \ar@{}[dr] |{{\displaystyle 4}} 
& \\
\ar@{}[dr] |{{\displaystyle 6}} 
&\ar@{-}+0;[r]+0  
&\ar@{-}+0;[r]+0  
& \ar@{}[dr] |{{\displaystyle 5}}  
& \\
& & & &}
}\, \boldsymbol{\rightarrow}
\]
\[
\vcenter{
\xymatrix@=12pt{
\ar@{}[dr] |{{\displaystyle 9}} 
& \ar@{}[dr] |{{\displaystyle 1}} 
& \ar@{}[dr] |{{\displaystyle 2}} 
& \ar@{}[dr] |{{\displaystyle 3}}
& \\
\ar@{}[dr] |{{\displaystyle 8}} 
&\ar@{}[dr] |{{\displaystyle X}} \ar@{-}+0;[d]+0 \ar@{-}+0;[r]+0 
&\ar@{}[dr] |{{\displaystyle X}} \ar@{-}+0;[d]+0 \ar@{-}+0;[r]+0  
&\ar@{-}+0;[d]+0  
& \\
&\ar@{}[dr] |{{\displaystyle A}} \ar@{-}+0;[d]+0 \ar@{-}+0;[r]+0 
&\ar@{}[dr] |{{\displaystyle X}} \ar@{-}+0;[d]+0 \ar@{-}+0;[r]+0  
&\ar@{-}+0;[d]+0 
& \\
\ar@{}[dr] |{{\displaystyle 7}} 
&\ar@{}[dr] |{{\displaystyle C}} \ar@{-}+0;[d]+0 \ar@{-}+0;[r]+0 
&\bullet \ar@{}[dr] |{{\displaystyle Y}} \ar@{-}+0;[d]+0 \ar@{-}+0;[r]+0  
&\ar@{-}+0;[d]+0  \ar@{}[dr] |{{\displaystyle 4}} 
& \\
\ar@{}[dr] |{{\displaystyle 6}} 
&\ar@{-}+0;[r]+0  
&\ar@{-}+0;[r]+0  
& \ar@{}[dr] |{{\displaystyle 5}}  
& \\
& & & &}
}\, \boldsymbol{\rightarrow} \, \vcenter{ 
\xymatrix@=12pt{
\ar@{}[dr] |{{\displaystyle 9}} 
& \ar@{}[dr] |{{\displaystyle 1}} 
& \ar@{}[dr] |{{\displaystyle 2}} 
& \ar@{}[dr] |{{\displaystyle 3}}
& \\
\ar@{}[dr] |{{\displaystyle 8}} 
&\ar@{}[dr] |{{\displaystyle X}} \ar@{-}+0;[d]+0 \ar@{-}+0;[r]+0 
&\ar@{}[dr] |{{\displaystyle X}} \ar@{-}+0;[d]+0 \ar@{-}+0;[r]+0  
&\ar@{-}+0;[d]+0  
& \\
&\ar@{}[dr] |{{\displaystyle Y}} \ar@{-}+0;[d]+0 \ar@{-}+0;[r]+0 
&\ar@{}[dr] |{{\displaystyle A}} \ar@{-}+0;[d]+0 \ar@{-}+0;[r]+0  
&\ar@{-}+0;[d]+0 
& \\
\ar@{}[dr] |{{\displaystyle 7}} 
&\ar@{}[dr] |{{\displaystyle X}} \ar@{-}+0;[d]+0 \ar@{-}+0;[r]+0 
&\ar@{}[dr] |{{\displaystyle C}} \ar@{-}+0;[d]+0 \ar@{-}+0;[r]+0  
&\bullet \ar@{-}+0;[d]+0  \ar@{}[dr] |{{\displaystyle 4}} 
& \\
\ar@{}[dr] |{{\displaystyle 6}} 
&\ar@{-}+0;[r]+0  
&\ar@{-}+0;[r]+0  
& \ar@{}[dr] |{{\displaystyle 5}}  
& \\
& & & &}
}\, \boldsymbol{\rightarrow} \, \vcenter{
\xymatrix@=12pt{
\ar@{}[dr] |{{\displaystyle 9}} 
& \ar@{}[dr] |{{\displaystyle 1}} 
& \ar@{}[dr] |{{\displaystyle 2}} 
& \ar@{}[dr] |{{\displaystyle 3}}
& \\
\ar@{}[dr] |{{\displaystyle 8}} 
&\ar@{}[dr] |{{\displaystyle X}} \ar@{-}+0;[d]+0 \ar@{-}+0;[r]+0 
&\ar@{}[dr] |{{\displaystyle X}} \ar@{-}+0;[d]+0 \ar@{-}+0;[r]+0  
&\ar@{-}+0;[d]+0  
& \\
&\bullet \ar@{}[dr] |{{\displaystyle Y}} \ar@{-}+0;[d]+0 \ar@{-}+0;[r]+0 
&\ar@{}[dr] |{{\displaystyle Y}} \ar@{-}+0;[d]+0 \ar@{-}+0;[r]+0  
&\ar@{-}+0;[d]+0 \ar@{}[dr] |{{\displaystyle 4}} 
& \\
\ar@{}[dr] |{{\displaystyle 7}} 
&\ar@{}[dr] |{{\displaystyle X}} \ar@{-}+0;[d]+0 \ar@{-}+0;[r]+0 
&\ar@{}[dr] |{{\displaystyle X}} \ar@{-}+0;[d]+0 \ar@{-}+0;[r]+0  
&\ar@{-}+0;[d]+0  
& \\
\ar@{}[dr] |{{\displaystyle 6}} 
&\ar@{-}+0;[r]+0  
&\ar@{-}+0;[r]+0  
& \ar@{}[dr] |{{\displaystyle 5}}  
& \\
& & & &}
}\, \boldsymbol{\rightarrow} \, \vcenter{ 
\xymatrix@=12pt{
\ar@{}[dr] |{{\displaystyle 9}} 
& \ar@{}[dr] |{{\displaystyle 1}} 
& \ar@{}[dr] |{{\displaystyle 2}} 
& \ar@{}[dr] |{{\displaystyle 3}}
& \\
&\ar@{}[dr] |{{\displaystyle A}} \ar@{-}+0;[d]+0 \ar@{-}+0;[r]+0 
&\ar@{}[dr] |{{\displaystyle X}} \ar@{-}+0;[d]+0 \ar@{-}+0;[r]+0  
&\ar@{-}+0;[d]+0  
& \\
\ar@{}[dr] |{{\displaystyle 8}}
&\ar@{}[dr] |{{\displaystyle C}} \ar@{-}+0;[d]+0 \ar@{-}+0;[r]+0 
&\bullet \ar@{}[dr] |{{\displaystyle Y}} \ar@{-}+0;[d]+0 \ar@{-}+0;[r]+0  
&\ar@{-}+0;[d]+0 \ar@{}[dr] |{{\displaystyle 4}} 
& \\
\ar@{}[dr] |{{\displaystyle 7}} 
&\ar@{}[dr] |{{\displaystyle X}} \ar@{-}+0;[d]+0 \ar@{-}+0;[r]+0 
&\ar@{}[dr] |{{\displaystyle X}} \ar@{-}+0;[d]+0 \ar@{-}+0;[r]+0  
&\ar@{-}+0;[d]+0  
& \\
\ar@{}[dr] |{{\displaystyle 6}} 
&\ar@{-}+0;[r]+0  
&\ar@{-}+0;[r]+0  
& \ar@{}[dr] |{{\displaystyle 5}}  
& \\
& & & &}
}\, \boldsymbol{\rightarrow}
\]
\[
\vcenter{
\xymatrix@=12pt{
\ar@{}[dr] |{{\displaystyle 9}} 
& \ar@{}[dr] |{{\displaystyle 1}} 
& \ar@{}[dr] |{{\displaystyle 2}} 
& \ar@{}[dr] |{{\displaystyle 3}}
& \\
&\ar@{}[dr] |{{\displaystyle Y}} \ar@{-}+0;[d]+0 \ar@{-}+0;[r]+0 
&\ar@{}[dr] |{{\displaystyle A}} \ar@{-}+0;[d]+0 \ar@{-}+0;[r]+0  
&\ar@{-}+0;[d]+0  
& \\
\ar@{}[dr] |{{\displaystyle 8}}
&\ar@{}[dr] |{{\displaystyle X}} \ar@{-}+0;[d]+0 \ar@{-}+0;[r]+0 
&\ar@{}[dr] |{{\displaystyle C}} \ar@{-}+0;[d]+0 \ar@{-}+0;[r]+0  
&\bullet \ar@{-}+0;[d]+0 \ar@{}[dr] |{{\displaystyle 4}} 
& \\
\ar@{}[dr] |{{\displaystyle 7}} 
&\ar@{}[dr] |{{\displaystyle X}} \ar@{-}+0;[d]+0 \ar@{-}+0;[r]+0 
&\ar@{}[dr] |{{\displaystyle X}} \ar@{-}+0;[d]+0 \ar@{-}+0;[r]+0  
&\ar@{-}+0;[d]+0  
& \\
\ar@{}[dr] |{{\displaystyle 6}} 
&\ar@{-}+0;[r]+0  
&\ar@{-}+0;[r]+0  
& \ar@{}[dr] |{{\displaystyle 5}}  
& \\
& & & &}
}\, \boldsymbol{\rightarrow} \, \vcenter{
\xymatrix@=12pt{
\ar@{}[dr] |{{\displaystyle 9}} 
& \ar@{}[dr] |{{\displaystyle 1}} 
& \ar@{}[dr] |{{\displaystyle 2}} 
& \ar@{}[dr] |{{\displaystyle 3}}
& \\
&\bullet \ar@{}[dr] |{{\displaystyle Y}} \ar@{-}+0;[d]+0 \ar@{-}+0;[r]+0 
&\ar@{}[dr] |{{\displaystyle Y}} \ar@{-}+0;[d]+0 \ar@{-}+0;[r]+0  
&\ar@{-}+0;[d]+0  \ar@{}[dr] |{{\displaystyle 4}} 
& \\
\ar@{}[dr] |{{\displaystyle 8}}
&\ar@{}[dr] |{{\displaystyle X}} \ar@{-}+0;[d]+0 \ar@{-}+0;[r]+0 
&\ar@{}[dr] |{{\displaystyle X}} \ar@{-}+0;[d]+0 \ar@{-}+0;[r]+0  
&\ar@{-}+0;[d]+0 
& \\
\ar@{}[dr] |{{\displaystyle 7}} 
&\ar@{}[dr] |{{\displaystyle X}} \ar@{-}+0;[d]+0 \ar@{-}+0;[r]+0 
&\ar@{}[dr] |{{\displaystyle X}} \ar@{-}+0;[d]+0 \ar@{-}+0;[r]+0  
&\ar@{-}+0;[d]+0  
& \\
\ar@{}[dr] |{{\displaystyle 6}} 
&\ar@{-}+0;[r]+0  
&\ar@{-}+0;[r]+0  
& \ar@{}[dr] |{{\displaystyle 5}}  
& \\
& & & &}
}\, \boldsymbol{\rightarrow} \, \vcenter{ 
\xymatrix@=12pt{
\ar@{}[dr] |{{\displaystyle 1}} 
& \ar@{}[dr] |{{\displaystyle }} 
& \ar@{}[dr] |{{\displaystyle 2}} 
& \ar@{}[dr] |{{\displaystyle 3}}
& \\
\ar@{}[dr] |{{\displaystyle 9}} 
&\ar@{}[dr] |{{\displaystyle C}} \ar@{-}+0;[d]+0 \ar@{-}+0;[r]+0 
&\bullet \ar@{}[dr] |{{\displaystyle Y}} \ar@{-}+0;[d]+0 \ar@{-}+0;[r]+0  
&\ar@{-}+0;[d]+0  \ar@{}[dr] |{{\displaystyle 4}} 
& \\
\ar@{}[dr] |{{\displaystyle 8}}
&\ar@{}[dr] |{{\displaystyle X}} \ar@{-}+0;[d]+0 \ar@{-}+0;[r]+0 
&\ar@{}[dr] |{{\displaystyle X}} \ar@{-}+0;[d]+0 \ar@{-}+0;[r]+0  
&\ar@{-}+0;[d]+0 
& \\
\ar@{}[dr] |{{\displaystyle 7}} 
&\ar@{}[dr] |{{\displaystyle X}} \ar@{-}+0;[d]+0 \ar@{-}+0;[r]+0 
&\ar@{}[dr] |{{\displaystyle X}} \ar@{-}+0;[d]+0 \ar@{-}+0;[r]+0  
&\ar@{-}+0;[d]+0  
& \\
\ar@{}[dr] |{{\displaystyle 6}} 
&\ar@{-}+0;[r]+0  
&\ar@{-}+0;[r]+0  
& \ar@{}[dr] |{{\displaystyle 5}}  
& \\
& & & &}
}\, \boldsymbol{\rightarrow} \, \vcenter{ 
\xymatrix@=12pt{
\ar@{}[dr] |{{\displaystyle 1}} 
& \ar@{}[dr] |{{\displaystyle 2}} 
& 
& \ar@{}[dr] |{{\displaystyle 3}}
& \\
\ar@{}[dr] |{{\displaystyle 9}} 
&\ar@{}[dr] |{{\displaystyle X}} \ar@{-}+0;[d]+0 \ar@{-}+0;[r]+0 
&\ar@{}[dr] |{{\displaystyle C}} \ar@{-}+0;[d]+0 \ar@{-}+0;[r]+0  
&\bullet \ar@{-}+0;[d]+0  \ar@{}[dr] |{{\displaystyle 4}} 
& \\
\ar@{}[dr] |{{\displaystyle 8}}
&\ar@{}[dr] |{{\displaystyle X}} \ar@{-}+0;[d]+0 \ar@{-}+0;[r]+0 
&\ar@{}[dr] |{{\displaystyle X}} \ar@{-}+0;[d]+0 \ar@{-}+0;[r]+0  
&\ar@{-}+0;[d]+0 
& \\
\ar@{}[dr] |{{\displaystyle 7}} 
&\ar@{}[dr] |{{\displaystyle X}} \ar@{-}+0;[d]+0 \ar@{-}+0;[r]+0 
&\ar@{}[dr] |{{\displaystyle X}} \ar@{-}+0;[d]+0 \ar@{-}+0;[r]+0  
&\ar@{-}+0;[d]+0  
& \\
\ar@{}[dr] |{{\displaystyle 6}} 
&\ar@{-}+0;[r]+0  
&\ar@{-}+0;[r]+0  
& \ar@{}[dr] |{{\displaystyle 5}}  
& \\
& & & &}
}\, \boldsymbol{\rightarrow}
\]
%
%
\[
\vcenter{
\xymatrix@=12pt{
\ar@{}[dr] |{{\displaystyle 1}} 
& \ar@{}[dr] |{{\displaystyle 2}} 
& \ar@{}[dr] |{{\displaystyle 3}} 
& \ar@{}[dr] |{{\displaystyle 4}}
& \\
\ar@{}[dr] |{{\displaystyle 9}} 
&\ar@{}[dr] |{{\displaystyle X}} \ar@{-}+0;[d]+0 \ar@{-}+0;[r]+0 
&\ar@{}[dr] |{{\displaystyle X}} \ar@{-}+0;[d]+0 \ar@{-}+0;[r]+0  
&\ar@{-}+0;[d]+0  
& \\
\ar@{}[dr] |{{\displaystyle 8}}
&\ar@{}[dr] |{{\displaystyle X}} \ar@{-}+0;[d]+0 \ar@{-}+0;[r]+0 
&\ar@{}[dr] |{{\displaystyle X}} \ar@{-}+0;[d]+0 \ar@{-}+0;[r]+0  
&\ar@{-}+0;[d]+0 
& \\
\ar@{}[dr] |{{\displaystyle 7}} 
&\ar@{}[dr] |{{\displaystyle X}} \ar@{-}+0;[d]+0 \ar@{-}+0;[r]+0 
&\ar@{}[dr] |{{\displaystyle X}} \ar@{-}+0;[d]+0 \ar@{-}+0;[r]+0  
&\ar@{-}+0;[d]+0  
& \\
\ar@{}[dr] |{{\displaystyle 6}} 
&\bullet  \ar@{-}+0;[r]+0  
&\ar@{-}+0;[r]+0  
& \ar@{}[dr] |{{\displaystyle 5}}  
& \\
& & & &}
}\, \boldsymbol{\rightarrow} \, \vcenter{ 
\xymatrix@=12pt{
\ar@{}[dr] |{{\displaystyle 1}} 
& \ar@{}[dr] |{{\displaystyle 2}} 
& \ar@{}[dr] |{{\displaystyle 3}} 
& \ar@{}[dr] |{{\displaystyle 4}}
& \\
\ar@{}[dr] |{{\displaystyle 9}} 
&\ar@{}[dr] |{{\displaystyle X}} \ar@{-}+0;[d]+0 \ar@{-}+0;[r]+0 
&\ar@{}[dr] |{{\displaystyle X}} \ar@{-}+0;[d]+0 \ar@{-}+0;[r]+0  
&\ar@{-}+0;[d]+0  
& \\
\ar@{}[dr] |{{\displaystyle 8}}
&\ar@{}[dr] |{{\displaystyle X}} \ar@{-}+0;[d]+0 \ar@{-}+0;[r]+0 
&\ar@{}[dr] |{{\displaystyle X}} \ar@{-}+0;[d]+0 \ar@{-}+0;[r]+0  
&\ar@{-}+0;[d]+0 
& \\
&\ar@{}[dr] |{{\displaystyle A}} \ar@{-}+0;[d]+0 \ar@{-}+0;[r]+0 
&\ar@{}[dr] |{{\displaystyle X}} \ar@{-}+0;[d]+0 \ar@{-}+0;[r]+0  
&\ar@{-}+0;[d]+0  
& \\
\ar@{}[dr] |{{\displaystyle 7}} 
&\ar@{-}+0;[r]+0  \ar@{}[dr] |{{\displaystyle 6}}
&\bullet \ar@{-}+0;[r]+0  
& \ar@{}[dr] |{{\displaystyle 5}}  
& \\
& & & &}
}\, \boldsymbol{\rightarrow} \, \vcenter{
\xymatrix@=12pt{
\ar@{}[dr] |{{\displaystyle 1}} 
& \ar@{}[dr] |{{\displaystyle 2}} 
& \ar@{}[dr] |{{\displaystyle 3}} 
& \ar@{}[dr] |{{\displaystyle 4}}
& \\
\ar@{}[dr] |{{\displaystyle 9}} 
&\ar@{}[dr] |{{\displaystyle X}} \ar@{-}+0;[d]+0 \ar@{-}+0;[r]+0 
&\ar@{}[dr] |{{\displaystyle X}} \ar@{-}+0;[d]+0 \ar@{-}+0;[r]+0  
&\ar@{-}+0;[d]+0  
& \\
\ar@{}[dr] |{{\displaystyle 8}}
&\ar@{}[dr] |{{\displaystyle X}} \ar@{-}+0;[d]+0 \ar@{-}+0;[r]+0 
&\ar@{}[dr] |{{\displaystyle X}} \ar@{-}+0;[d]+0 \ar@{-}+0;[r]+0  
&\ar@{-}+0;[d]+0 
& \\
&\bullet \ar@{}[dr] |{{\displaystyle Y}} \ar@{-}+0;[d]+0 \ar@{-}+0;[r]+0 
&\ar@{}[dr] |{{\displaystyle A}} \ar@{-}+0;[d]+0 \ar@{-}+0;[r]+0  
&\ar@{-}+0;[d]+0  
& \\
\ar@{}[dr] |{{\displaystyle 7}} 
&\ar@{-}+0;[r]+0  
&\ar@{-}+0;[r]+0  \ar@{}[dr] |{{\displaystyle 6}} 
&\ar@{}[dr] |{{\displaystyle 5}}  
& \\
& & & &}
}\, \boldsymbol{\rightarrow} \, \vcenter{
\xymatrix@=12pt{
\ar@{}[dr] |{{\displaystyle 1}} 
& \ar@{}[dr] |{{\displaystyle 2}} 
& \ar@{}[dr] |{{\displaystyle 3}} 
& \ar@{}[dr] |{{\displaystyle 4}}
& \\
\ar@{}[dr] |{{\displaystyle 9}} 
&\ar@{}[dr] |{{\displaystyle X}} \ar@{-}+0;[d]+0 \ar@{-}+0;[r]+0 
&\ar@{}[dr] |{{\displaystyle X}} \ar@{-}+0;[d]+0 \ar@{-}+0;[r]+0  
&\ar@{-}+0;[d]+0  
& \\
&\ar@{}[dr] |{{\displaystyle A}} \ar@{-}+0;[d]+0 \ar@{-}+0;[r]+0 
&\ar@{}[dr] |{{\displaystyle X}} \ar@{-}+0;[d]+0 \ar@{-}+0;[r]+0  
&\ar@{-}+0;[d]+0 
& \\
\ar@{}[dr] |{{\displaystyle 8}} 
&\ar@{}[dr] |{{\displaystyle C}} \ar@{-}+0;[d]+0 \ar@{-}+0;[r]+0 
&\bullet \ar@{}[dr] |{{\displaystyle A}} \ar@{-}+0;[d]+0 \ar@{-}+0;[r]+0  
&\ar@{-}+0;[d]+0  
& \\
\ar@{}[dr] |{{\displaystyle 7}} 
&\ar@{-}+0;[r]+0  
&\ar@{-}+0;[r]+0  \ar@{}[dr] |{{\displaystyle 6}} 
&\ar@{}[dr] |{{\displaystyle 5}}  
& \\
& & & &}
}\, \boldsymbol{\rightarrow}
\]
\[
\vcenter{
\xymatrix@=12pt{
\ar@{}[dr] |{{\displaystyle 1}} 
& \ar@{}[dr] |{{\displaystyle 2}} 
& \ar@{}[dr] |{{\displaystyle 3}} 
& \ar@{}[dr] |{{\displaystyle 4}}
& \\
\ar@{}[dr] |{{\displaystyle 9}} 
&\ar@{}[dr] |{{\displaystyle X}} \ar@{-}+0;[d]+0 \ar@{-}+0;[r]+0 
&\ar@{}[dr] |{{\displaystyle X}} \ar@{-}+0;[d]+0 \ar@{-}+0;[r]+0  
&\ar@{-}+0;[d]+0  
& \\
&\bullet \ar@{}[dr] |{{\displaystyle Y}} \ar@{-}+0;[d]+0 \ar@{-}+0;[r]+0 
&\ar@{}[dr] |{{\displaystyle A}} \ar@{-}+0;[d]+0 \ar@{-}+0;[r]+0  
&\ar@{-}+0;[d]+0 
& \\
\ar@{}[dr] |{{\displaystyle 8}} 
&\ar@{}[dr] |{{\displaystyle X}} \ar@{-}+0;[d]+0 \ar@{-}+0;[r]+0 
&\ar@{}[dr] |{{\displaystyle Z}} \ar@{-}+0;[d]+0 \ar@{-}+0;[r]+0  
&\ar@{-}+0;[d]+0  
& \\
\ar@{}[dr] |{{\displaystyle 7}} 
&\ar@{-}+0;[r]+0  
&\ar@{-}+0;[r]+0  \ar@{}[dr] |{{\displaystyle 6}} 
& \ar@{}[dr] |{{\displaystyle 5}}  
& \\
& & & &}
}\, \boldsymbol{\rightarrow} \, \vcenter{ 
\xymatrix@=12pt{
\ar@{}[dr] |{{\displaystyle 1}} 
& \ar@{}[dr] |{{\displaystyle 2}} 
& \ar@{}[dr] |{{\displaystyle 3}} 
& \ar@{}[dr] |{{\displaystyle 4}}
& \\
&\ar@{}[dr] |{{\displaystyle A}} \ar@{-}+0;[d]+0 \ar@{-}+0;[r]+0 
&\ar@{}[dr] |{{\displaystyle X}} \ar@{-}+0;[d]+0 \ar@{-}+0;[r]+0  
&\ar@{-}+0;[d]+0  
& \\
\ar@{}[dr] |{{\displaystyle 9}}
&\ar@{}[dr] |{{\displaystyle C}} \ar@{-}+0;[d]+0 \ar@{-}+0;[r]+0 
&\bullet \ar@{}[dr] |{{\displaystyle A}} \ar@{-}+0;[d]+0 \ar@{-}+0;[r]+0  
&\ar@{-}+0;[d]+0 
& \\
\ar@{}[dr] |{{\displaystyle 8}} 
&\ar@{}[dr] |{{\displaystyle X}} \ar@{-}+0;[d]+0 \ar@{-}+0;[r]+0 
&\ar@{}[dr] |{{\displaystyle Z}} \ar@{-}+0;[d]+0 \ar@{-}+0;[r]+0  
&\ar@{-}+0;[d]+0  
& \\
\ar@{}[dr] |{{\displaystyle 7}} 
&\ar@{-}+0;[r]+0  
&\ar@{-}+0;[r]+0  \ar@{}[dr] |{{\displaystyle 6}} 
& \ar@{}[dr] |{{\displaystyle 5}}  
& \\
& & & &}
}\, \boldsymbol{\rightarrow} \, \vcenter{
\xymatrix@=12pt{
\ar@{}[dr] |{{\displaystyle 1}} 
& \ar@{}[dr] |{{\displaystyle 2}} 
& \ar@{}[dr] |{{\displaystyle 3}} 
& \ar@{}[dr] |{{\displaystyle 4}}
& \\
&\bullet \ar@{}[dr] |{{\displaystyle Y}} \ar@{-}+0;[d]+0 \ar@{-}+0;[r]+0 
&\ar@{}[dr] |{{\displaystyle A}} \ar@{-}+0;[d]+0 \ar@{-}+0;[r]+0  
&\ar@{-}+0;[d]+0  
& \\
\ar@{}[dr] |{{\displaystyle 9}}
&\ar@{}[dr] |{{\displaystyle X}} \ar@{-}+0;[d]+0 \ar@{-}+0;[r]+0 
&\ar@{}[dr] |{{\displaystyle Z}} \ar@{-}+0;[d]+0 \ar@{-}+0;[r]+0  
&\ar@{-}+0;[d]+0 
& \\
\ar@{}[dr] |{{\displaystyle 8}} 
&\ar@{}[dr] |{{\displaystyle X}} \ar@{-}+0;[d]+0 \ar@{-}+0;[r]+0 
&\ar@{}[dr] |{{\displaystyle Z}} \ar@{-}+0;[d]+0 \ar@{-}+0;[r]+0  
&\ar@{-}+0;[d]+0  
& \\
\ar@{}[dr] |{{\displaystyle 7}} 
&\ar@{-}+0;[r]+0  
&\ar@{-}+0;[r]+0  \ar@{}[dr] |{{\displaystyle 6}} 
& \ar@{}[dr] |{{\displaystyle 5}}  
& \\
& & & &}
}\, \boldsymbol{\rightarrow} \, \vcenter{ 
\xymatrix@=12pt{
\ar@{}[dr] |{{\displaystyle 2}} 
& 
& \ar@{}[dr] |{{\displaystyle 3}} 
& \ar@{}[dr] |{{\displaystyle 4}}
& \\
\ar@{}[dr] |{{\displaystyle 1}} 
&\ar@{}[dr] |{{\displaystyle C}} \ar@{-}+0;[d]+0 \ar@{-}+0;[r]+0 
&\bullet \ar@{}[dr] |{{\displaystyle A}} \ar@{-}+0;[d]+0 \ar@{-}+0;[r]+0  
&\ar@{-}+0;[d]+0  
& \\
\ar@{}[dr] |{{\displaystyle 9}}
&\ar@{}[dr] |{{\displaystyle X}} \ar@{-}+0;[d]+0 \ar@{-}+0;[r]+0 
&\ar@{}[dr] |{{\displaystyle Z}} \ar@{-}+0;[d]+0 \ar@{-}+0;[r]+0  
&\ar@{-}+0;[d]+0 
& \\
\ar@{}[dr] |{{\displaystyle 8}} 
&\ar@{}[dr] |{{\displaystyle X}} \ar@{-}+0;[d]+0 \ar@{-}+0;[r]+0 
&\ar@{}[dr] |{{\displaystyle Z}} \ar@{-}+0;[d]+0 \ar@{-}+0;[r]+0  
&\ar@{-}+0;[d]+0  
& \\
\ar@{}[dr] |{{\displaystyle 7}} 
&\ar@{-}+0;[r]+0  
&\ar@{-}+0;[r]+0  \ar@{}[dr] |{{\displaystyle 6}} 
& \ar@{}[dr] |{{\displaystyle 5}}  
& \\
& & & &}
}\, \boldsymbol{\rightarrow}
\]
%
%
\[
\vcenter{
\xymatrix@=12pt{
\ar@{}[dr] |{{\displaystyle 2}} 
& \ar@{}[dr] |{{\displaystyle 3}} 
& 
& \ar@{}[dr] |{{\displaystyle 4}}
& \\
\ar@{}[dr] |{{\displaystyle 1}} 
&\ar@{}[dr] |{{\displaystyle X}} \ar@{-}+0;[d]+0 \ar@{-}+0;[r]+0 
&\ar@{}[dr] |{{\displaystyle Z}} \ar@{-}+0;[d]+0 \ar@{-}+0;[r]+0  
&\ar@{-}+0;[d]+0  
& \\
\ar@{}[dr] |{{\displaystyle 9}}
&\ar@{}[dr] |{{\displaystyle X}} \ar@{-}+0;[d]+0 \ar@{-}+0;[r]+0 
&\ar@{}[dr] |{{\displaystyle Z}} \ar@{-}+0;[d]+0 \ar@{-}+0;[r]+0  
&\ar@{-}+0;[d]+0 
& \\
\ar@{}[dr] |{{\displaystyle 8}} 
&\ar@{}[dr] |{{\displaystyle X}} \ar@{-}+0;[d]+0 \ar@{-}+0;[r]+0 
&\ar@{}[dr] |{{\displaystyle Z}} \ar@{-}+0;[d]+0 \ar@{-}+0;[r]+0  
&\ar@{-}+0;[d]+0  
& \\
\ar@{}[dr] |{{\displaystyle 7}} 
&\bullet  \ar@{-}+0;[r]+0  
&\ar@{-}+0;[r]+0  \ar@{}[dr] |{{\displaystyle 6}} 
& \ar@{}[dr] |{{\displaystyle 5}}  
& \\
& & & &}
}\, \boldsymbol{\rightarrow} \, \vcenter{
\xymatrix@=12pt{
\ar@{}[dr] |{{\displaystyle 2}} 
& \ar@{}[dr] |{{\displaystyle 3}} 
& 
& \ar@{}[dr] |{{\displaystyle 4}}
& \\
\ar@{}[dr] |{{\displaystyle 1}} 
&\ar@{}[dr] |{{\displaystyle X}} \ar@{-}+0;[d]+0 \ar@{-}+0;[r]+0 
&\ar@{}[dr] |{{\displaystyle Z}} \ar@{-}+0;[d]+0 \ar@{-}+0;[r]+0  
&\ar@{-}+0;[d]+0  
& \\
\ar@{}[dr] |{{\displaystyle 9}}
&\ar@{}[dr] |{{\displaystyle X}} \ar@{-}+0;[d]+0 \ar@{-}+0;[r]+0 
&\ar@{}[dr] |{{\displaystyle Z}} \ar@{-}+0;[d]+0 \ar@{-}+0;[r]+0  
&\ar@{-}+0;[d]+0 
& \\
&\bullet \ar@{}[dr] |{{\displaystyle A}} \ar@{-}+0;[d]+0 \ar@{-}+0;[r]+0 
&\ar@{}[dr] |{{\displaystyle Z}} \ar@{-}+0;[d]+0 \ar@{-}+0;[r]+0  
&\ar@{-}+0;[d]+0  
& \\
\ar@{}[dr] |{{\displaystyle 8}} 
&\ar@{-}+0;[r]+0  \ar@{}[dr] |{{\displaystyle 7}}
&\ar@{-}+0;[r]+0  \ar@{}[dr] |{{\displaystyle 6}} 
& \ar@{}[dr] |{{\displaystyle 5}}  
& \\
& & & &}
}\, \boldsymbol{\rightarrow} \, \vcenter{
\xymatrix@=12pt{
\ar@{}[dr] |{{\displaystyle 2}} 
& \ar@{}[dr] |{{\displaystyle 3}} 
& 
& \ar@{}[dr] |{{\displaystyle 4}}
& \\
\ar@{}[dr] |{{\displaystyle 1}} 
&\ar@{}[dr] |{{\displaystyle X}} \ar@{-}+0;[d]+0 \ar@{-}+0;[r]+0 
&\ar@{}[dr] |{{\displaystyle Z}} \ar@{-}+0;[d]+0 \ar@{-}+0;[r]+0  
&\ar@{-}+0;[d]+0  
& \\
&\bullet \ar@{}[dr] |{{\displaystyle A}} \ar@{-}+0;[d]+0 \ar@{-}+0;[r]+0 
&\ar@{}[dr] |{{\displaystyle Z}} \ar@{-}+0;[d]+0 \ar@{-}+0;[r]+0  
&\ar@{-}+0;[d]+0 
& \\
\ar@{}[dr] |{{\displaystyle 9}} 
&\ar@{}[dr] |{{\displaystyle Z}} \ar@{-}+0;[d]+0 \ar@{-}+0;[r]+0 
&\ar@{}[dr] |{{\displaystyle Z}} \ar@{-}+0;[d]+0 \ar@{-}+0;[r]+0  
&\ar@{-}+0;[d]+0  
& \\
\ar@{}[dr] |{{\displaystyle 8}} 
&\ar@{-}+0;[r]+0  \ar@{}[dr] |{{\displaystyle 7}}
&\ar@{-}+0;[r]+0  \ar@{}[dr] |{{\displaystyle 6}} 
& \ar@{}[dr] |{{\displaystyle 5}}  
& \\
& & & &}
}\, \boldsymbol{\rightarrow} \, \vcenter{
\xymatrix@=12pt{
\ar@{}[dr] |{{\displaystyle 2}} 
& \ar@{}[dr] |{{\displaystyle 3}} 
& 
& \ar@{}[dr] |{{\displaystyle 4}}
& \\
&\bullet \ar@{}[dr] |{{\displaystyle A}} \ar@{-}+0;[d]+0 \ar@{-}+0;[r]+0 
&\ar@{}[dr] |{{\displaystyle Z}} \ar@{-}+0;[d]+0 \ar@{-}+0;[r]+0  
&\ar@{-}+0;[d]+0  
& \\
\ar@{}[dr] |{{\displaystyle 1}}
&\ar@{}[dr] |{{\displaystyle Z}} \ar@{-}+0;[d]+0 \ar@{-}+0;[r]+0 
&\ar@{}[dr] |{{\displaystyle Z}} \ar@{-}+0;[d]+0 \ar@{-}+0;[r]+0  
&\ar@{-}+0;[d]+0 
& \\
\ar@{}[dr] |{{\displaystyle 9}} 
&\ar@{}[dr] |{{\displaystyle Z}} \ar@{-}+0;[d]+0 \ar@{-}+0;[r]+0 
&\ar@{}[dr] |{{\displaystyle Z}} \ar@{-}+0;[d]+0 \ar@{-}+0;[r]+0  
&\ar@{-}+0;[d]+0  
& \\
\ar@{}[dr] |{{\displaystyle 8}} 
&\ar@{-}+0;[r]+0  \ar@{}[dr] |{{\displaystyle 7}}
&\ar@{-}+0;[r]+0  \ar@{}[dr] |{{\displaystyle 6}} 
& \ar@{}[dr] |{{\displaystyle 5}}  
& \\
& & & &}
}\, \boldsymbol{\rightarrow}
\]
%
%
%
\[
\xymatrix@=12pt{
\ar@{}[dr] |{{\displaystyle 3}} 
& 
& 
& \ar@{}[dr] |{{\displaystyle 4}}
& \\
\ar@{}[dr] |{{\displaystyle 2}} 
&\ar@{}[dr] |{{\displaystyle Z}} \ar@{-}+0;[d]+0 \ar@{-}+0;[r]+0 
&\ar@{}[dr] |{{\displaystyle Z}} \ar@{-}+0;[d]+0 \ar@{-}+0;[r]+0  
&\ar@{-}+0;[d]+0  
& \\
\ar@{}[dr] |{{\displaystyle 1}}
&\ar@{}[dr] |{{\displaystyle Z}} \ar@{-}+0;[d]+0 \ar@{-}+0;[r]+0 
&\ar@{}[dr] |{{\displaystyle Z}} \ar@{-}+0;[d]+0 \ar@{-}+0;[r]+0  
&\ar@{-}+0;[d]+0 
& \\
\ar@{}[dr] |{{\displaystyle 9}} 
&\ar@{}[dr] |{{\displaystyle Z}} \ar@{-}+0;[d]+0 \ar@{-}+0;[r]+0 
&\ar@{}[dr] |{{\displaystyle Z}} \ar@{-}+0;[d]+0 \ar@{-}+0;[r]+0  
&\ar@{-}+0;[d]+0  
& \\
\ar@{}[dr] |{{\displaystyle 8}} 
&\ar@{-}+0;[r]+0  \ar@{}[dr] |{{\displaystyle 7}}
&\ar@{-}+0;[r]+0  \ar@{}[dr] |{{\displaystyle 6}} 
& \ar@{}[dr] |{{\displaystyle 5}}  
& \\
& & & &}
\]
If $\omega=(x,y;l)$, recall that the vertex $(x,y)$ is labelled with the Pl\"{u}cker
coordinate $[\varrho^{l-1}(M_{k,n}(x,y))]$ in the
seed $(\xx_{\omega},\widetilde{Q}_{\omega})$.
The following is easy to check.

\begin{lemma} \label{l:local}
Let $\omega=(x,y;l)\in \Omega\setminus \{\top\}$.
Then the arrows incident with the vertex $(x,y)$ in
the seed $(\xx_{\omega},\widetilde{Q}_{\omega})$
are as follows, allowing the cases $i=k$ or $j=0$ in the formula for
$M_{k,n}(x,y)$ in Lemma~\ref{l:rknsubsets}. $\qed$
$$
\xymatrix@R=12pt@C=6pt{
&
{\varrho^{l-1}(M_{k,n}(x,y+1))}
\\
{\varrho^l(M_{k,n}(x-1,y))} \ar[r] 
&
{\varrho^{l-1}(M_{k,n}(x,y))} \ar[u] \ar[d]
&
{\varrho^{l-1}(M_{k,n}(x+1,y))} \ar[l]  \\
&
{\varrho^l(M_{k,n}(x,y-1))}
}
$$
\end{lemma}

\begin{lemma} \label{l:mutationstep}
Let $\omega\in \Omega\setminus \{\top\}$, so $\omega=(x,y;l)$ where
$1\leq l\leq k-1$, $1\leq x\leq k-l$ and $1\leq y\leq n-k-1$.
Then
$$\mu_{(x,y)}(\xx_{\omega},\widetilde{Q}_{\omega})=(\xx_{\omega^+},\widetilde{Q}_{\omega^+}).$$
\end{lemma}
\begin{proof}
Using Lemma~\ref{l:localrules}, it is easy to check in each
case that $\mu_{(x,y)}(Q(L_{\omega}))=Q(L_{\omega^+})$.
The last four rules in Figure~\ref{fig:localrules} occur
on pages $l>1$ only: at the end of a row, at the end of the top row, on the last page, and in the final mutation on the last page, respectively.

Mutating $(\xx_{\omega},\widetilde{Q}_{\omega})$ at $(x,y)$ replaces the
Pl\"{u}cker coordinate
$[\varrho^{l-1}(M_{k,n}(x,y))]$ with a new element of $\mathbb{C}[\Gr_{k,n}]$.
We claim that this new element is
$[\varrho^l(M_{k,n}(x,y))]$.

By Lemma~\ref{l:local}, the cluster variable $[\varrho^{l-1}(M_{k,n}(x,y))]^*$ associated to $(x,y)$ after mutation at $(x,y)$ is
$$\frac{[\varrho^l(M_{k,n}(x-1,y))][\varrho^{l-1}(M_{k,n}(x+1,y))]+
[\varrho^{l-1}(M_{k,n}(x,y+1))][\varrho^l(M_{k,n}(x,y-1))]}{[\varrho^{l-1}(M_{k,n}(x,y))]}.$$
Suppose first that $l=1$, so we have:
$$[M_{k,n}(x,y)]^*=
\frac{[\varrho(M_{k,n}(x-1,y))][M_{k,n}(x+1,y)]+
[M_{k,n}(x,y+1)][\varrho(M_{k,n}(x,y-1))]}{[M_{k,n}(x,y)]}.$$
Recall (Lemma~\ref{l:rknsubsets}) that $M_{k,n}(i,j)=\{1,\ldots ,i\}\cup
\{i+j+1,\ldots ,j+k\}$.
Let $A=\{2,\ldots ,x\}\cup \{x+y+2,\ldots ,y+k\}$.
Then we have, using a Pl\"{u}cker relation (and noting
that $y\leq n-k-1$, so $y+k+1\leq n$ and
$1<x+1<x+y+1<y+k+1$), that $[M_{k,n}(x,y)]^*$ is given by:
\begin{align*}
&\frac{[A\cup \{x+y+1,y+k+1\}][A\cup \{1,x+1\}]+[A\cup \{1,y+k+1\}][A\cup \{x+1,x+y+1\}]}{[M_{k,n}(x,y)]} \\
&=\frac{[A\cup \{1,x+y+1\}][A\cup \{x+1,y+k+1\}]}{[M_{k,n}(x,y)]} \\
&=\frac{[M_{k,n}(x,y)][\varrho(M_{k,n}(x,y))]}{[M_{k,n}(x,y)]}=[\varrho(M_{k,n}(x,y))].
\end{align*}
It follows that, for any $l$, we have
\begin{equation}
\label{e:mutation}
[\varrho^{l-1}(M_{k,n}(x,y))]^*=[\varrho^l(M_{k,n}(x,y))],
\end{equation}
as claimed.

Suppose first that $(x,y)\not=(k-l,n-k-1)$. Then
$(x,y;l)^+=((x,y)^+,l)$. Then the Pl\"{u}cker coordinate
associated to $(i,j)$ in $(\xx_{\omega},\widetilde{Q}_{\omega})$ is $[\varrho^{r_{\omega}(i,j)}(M_{k,n}(i,j))]$,
where $r_{\omega}(i,j)$ is given by~\eqref{e:rij}. The coordinate
associated to $(i,j)$ in $(\xx_{\omega^+},\widetilde{Q}_{\omega})$ is
$[\varrho^{r_{\omega^+}(i,j)}(M_{k,n}(i,j))]$, where
$$r_{\omega^+}(i,j)=\begin{cases}
l, & i\leq k-l,\ (i,j)<(x,y)^+; \\
l-1, & i\leq k-l,\ (i,j)\not<(x,y)^+; \\
k-i, & i>k-l.
\end{cases}$$
We see that $r_{{\omega}^+}(i,j)=r_{\omega}(i,j)$ for all $(i,j)\in R_{\mathbb{Z}}$ except for $(i,j)=(x,y)$: we have $r_{\omega}(x,y)=l-1$
while $r_{\omega^+}(x,y)=l$, so the result holds in this case by~\eqref{e:mutation}.

Suppose that $(x,y)=(k-l,n-k-1)$. Then $(x,y;l)^+=(1,1;l+1)$ (or $\top$ if $l=k-1$).
In either case, the coordinate
associated to $(i,j)$ in $(\xx_{\omega^+},\widetilde{Q}_{\omega^+})$ is
$[\varrho^{r_{\omega^+}(i,j)}(M_{k,n}(i,j))]$, where
$$r_{\omega^+}(i,j)=\begin{cases}
l, & i\leq k-l, \\
k-i, & i>k-l.
\end{cases}
$$
We see that $r_{\omega^+}(i,j)=r_{\omega}(i,j)$ for all $(i,j)$ except that
$r_{\omega}(x,y)=l-1$ while $r_{\omega^+}(x,y)=l$, and we see that the lemma also holds in this case.
\end{proof}

\begin{lemma} \label{l:knownseeds}
\begin{itemize}
\item[(a)] The seed of $\mathbb{C}[\Gr_{k,n}]$ corresponding to the Postnikov diagram $\R_{k,n}$ is $(\xx_{\perp},\widetilde{Q}_{\perp})$.
\item[(b)] For all $\omega\in \Omega$, $(\xx_{\omega},\widetilde{Q}_{\omega})$ is a seed of $\mathbb{C}[\Gr_{k,n}]$.
\item[(c)] The quiver $\widetilde{Q}_{\top}$ is isomorphic to
$\widetilde{Q}(\R^*_{n-k,n})$.
\end{itemize}
\end{lemma}
\begin{proof}
We first prove part (a).
Note that $r_{\perp}(i,j)=0$ for all $(i,j)\in R_{\mathbb{Z}}$,
so the coordinate attached to $(i,j)$ in $(\xx_{\perp},\widetilde{Q}_{\perp})$ is $[M_{k,n}(i,j)]$.
It is easy to check that the quiver
$\widetilde{Q}_{\perp}$ is isomorphic to the quiver of $\R_{k,n}$
(Figure~\ref{fig:perplabelling} gives
an example, in the case $k=4$, $n=9$: compare with Figure~\ref{fig:R49}).

Part (b) follows from part (a) and Lemma~\ref{l:mutationstep}.

Part (c) is an easy check (note that the embedding in the plane of
$\widetilde{Q}(\R^*_{n-k,n})$ is different from that of $\widetilde{Q}_{\top}$).
(Figure~\ref{fig:toplabelling} shows
$L_{\top}$ and the corresponding quiver
in the case $k=4$, $n=9$: compare with
Figure~\ref{fig:R59dual}).
\end{proof}

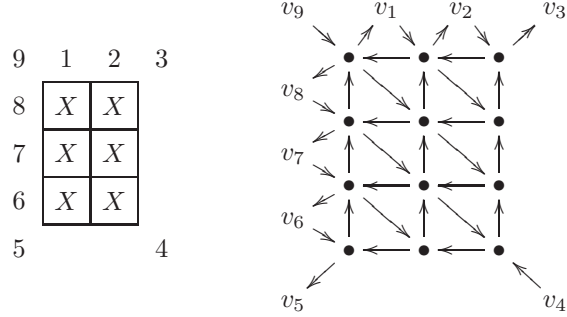
\begin{figure}
$$
\vcenter{
\xymatrix@=12pt{
\ar@{}[dr] |{{\displaystyle 9}} & \ar@{}[dr] |{{\displaystyle 1}} & \ar@{}[dr] |{{\displaystyle 2}} & \ar@{}[dr] |{{\displaystyle 3}}\\
\ar@{}[dr] |{{\displaystyle 8}} & \ar@{}[dr] |{{\displaystyle X}} \ar@{-}+0;[d]+0 \ar@{-}+0;[r]+0 &  \ar@{}[dr] |{{\displaystyle X}} \ar@{-}+0;[d]+0 \ar@{-}+0;[r]+0  & \ar@{-}+0;[d]+0 & \\
\ar@{}[dr] |{{\displaystyle 7}}& \ar@{}[dr] |{{\displaystyle X}} \ar@{-}+0;[d]+0 \ar@{-}+0;[r]+0 &  \ar@{}[dr] |{{\displaystyle X}} \ar@{-}+0;[d]+0 \ar@{-}+0;[r]+0  & \ar@{-}+0;[d]+0 & \\
\ar@{}[dr] |{{\displaystyle 6}} & \ar@{}[dr] |{{\displaystyle X}} \ar@{-}+0;[d]+0 \ar@{-}+0;[r]+0 &  \ar@{}[dr] |{{\displaystyle X}} \ar@{-}+0;[d]+0 \ar@{-}+0;[r]+0  & \ar@{-}+0;[d]+0 & \\
 \ar@{}[dr] |{{\displaystyle 5}} & \ar@{-}+0;[r]+0  & \ar@{-}+0;[r]+0 & \ar@{}[dr] |{{\displaystyle 4}} & \\
& & & & }}
\quad\quad\quad
\vcenter{\xymatrix@=1pt{
v_9 \ar[ddrr] && & v_1 \ar[ddr] && v_2 \ar[ddr] &&& v_3 \\ \\
&& \bullet \ar[lld] \ar[ddrr] \ar[uur] && \bullet \ar[uur] \ar[ll] \ar[ddrr] && \bullet \ar[ll] \ar[uurr] \\
v_8 \ar[rrd] \\
&& \bullet \ar[lld] \ar[ddrr] \ar[uu] && \bullet \ar[ll] \ar[ddrr] \ar[uu] && \bullet \ar[ll] \ar[uu] \\
v_7 \ar[rrd] \\
&& \bullet \ar[lld] \ar[ddrr] \ar[uu] && \bullet \ar[ll] \ar[ddrr] \ar[uu] && \bullet \ar[ll] \ar[uu] \\
v_6 \ar[rrd] \\
&& \bullet \ar[ddll] \ar[uu] && \bullet \ar[ll] \ar[uu] && \bullet \ar[ll] \ar[uu] \\ \\
v_5 && && && && v_4 \ar[uull]
}}$$
\caption{The labelling, $L_{\perp}$ (left) and
the corresponding quiver $Q_{\perp}$, isomorphic
to $\widetilde{Q}(\R_{4,9})$ (right), in the case $k=4$, $n=9$.}
\label{fig:perplabelling}
\end{figure}

\begin{figure}
$$
\vcenter{
\xymatrix@=12pt{
\ar@{}[dr] |{{\displaystyle 3}} & & & \ar@{}[dr] |{{\displaystyle 4}}\\
\ar@{}[dr] |{{\displaystyle 2}} & \ar@{}[dr] |{{\displaystyle Z}} \ar@{-}+0;[d]+0 \ar@{-}+0;[r]+0 &  \ar@{}[dr] |{{\displaystyle Z}} \ar@{-}+0;[d]+0 \ar@{-}+0;[r]+0  & \ar@{-}+0;[d]+0 & \\
\ar@{}[dr] |{{\displaystyle 1}}& \ar@{}[dr] |{{\displaystyle Z}} \ar@{-}+0;[d]+0 \ar@{-}+0;[r]+0 &  \ar@{}[dr] |{{\displaystyle Z}} \ar@{-}+0;[d]+0 \ar@{-}+0;[r]+0  & \ar@{-}+0;[d]+0 & \\
\ar@{}[dr] |{{\displaystyle 9}} & \ar@{}[dr] |{{\displaystyle Z}} \ar@{-}+0;[d]+0 \ar@{-}+0;[r]+0 &  \ar@{}[dr] |{{\displaystyle Z}} \ar@{-}+0;[d]+0 \ar@{-}+0;[r]+0  & \ar@{-}+0;[d]+0 & \\
 \ar@{}[dr] |{{\displaystyle 8}} & \ar@{-}+0;[r]+0  \ar@{}[dr] |{{\displaystyle 7}} & \ar@{-}+0;[r]+0  \ar@{}[dr] |{{\displaystyle 6}} & \ar@{}[dr] |{{\displaystyle 5}} & \\
& & & & }}
\quad\quad\quad
\vcenter{
\xymatrix@=1pt{
v_3 \ar[ddrr] &&& && &&& v_4 \\ \\
&& \bullet \ar[rr] \ar[dll] && \bullet \ar[rr] \ar[ddll] && \bullet \ar[uurr] \ar[ddll] \\
v_2 \ar[drr] \\
&& \bullet \ar[uu] \ar[rr] \ar[dll] && \bullet \ar[uu] \ar[rr] \ar[ddll] && \bullet \ar[uu] \ar[ddll] \\
v_1 \ar[drr] \\
&& \bullet \ar[uu] \ar[rr] \ar[dll] && \bullet \ar[uu] \ar[rr] \ar[ddll] && \bullet \ar[uu] \ar[ddll] \\
v_9 \ar[drr] \\
&& \bullet \ar[uu] \ar[rr] \ar[ddll] && \bullet \ar[uu] \ar[rr] \ar[ddl] && \bullet \ar[uu] \ar[ddl] \\ \\
v_8 &&& v_7 \ar[uul] && v_6 \ar[uul] &&& v_5 \ar[uull]
}}
$$
\caption{The labelling $L_{\top}$ (left) and corresponding quiver $\widetilde{Q}_{\top}$, isomorphic to $\R^*_{5,9}$ (right), in the case $k=4$, $n=9$.}
\label{fig:toplabelling}
\end{figure}
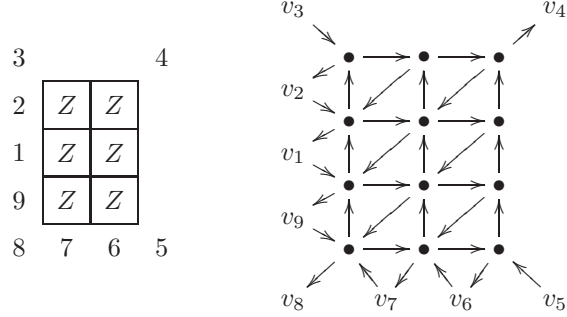

For $(i,j)\in R_{\mathbb{Z}}$, let $[\widetilde{M}_{k,n}(i,j)]$ be the Pl\"{u}cker coordinate in
$\xx_{\top}$ associated to the vertex $(i,j)$ of $\widetilde{Q}_{\top}$.

\begin{lemma} \label{l:tcoincide}
For all $(i,j)\in R_{\mathbb{Z}}$, we have that
$\twist{[M_{k,n}(k-i,j)]}$ is equal to a monomial in the coefficients
multiplied by $[\sigma^k(\widetilde{M}_{k,n}(i,j))]$.
Hence, up to coefficients, 
$$\twist{[M_{k,n}(k-i,j)]}=[\sigma^k(\widetilde{M}_{k,n}(i,j))].$$
\end{lemma}
\begin{proof}
We have
\begin{align} \label{e:tcoincide}
\begin{split}
\widetilde{M}_{k,n}(i,j) &=\varrho^{k-i}(M_{k,n}(i,j)) \\
&= \varrho^{k-i}(\{1,\ldots ,i\}\cup [i+j+1,j+k]).
\end{split}
\end{align}
By Lemma~\ref{l:rkntwist} we have that
$\twist{[M_{k,n}(k-i,j)]}$ is equal to $[I]$ multiplied by a monomial
in the coefficients, where
\begin{align*}
I&=[j+1,k-i+j]\cup [n-i+1,n] \\
&=[n+1-i,n]\cup [j+1,j+k-i] \\
&=\varrho^{-i}(\{1,\ldots ,i\}\cup [i+j+1,j+k]) \\
&=\varrho^{-k}(\widetilde{M}_{k,n}(i,j)),
\end{align*}
by~\eqref{e:tcoincide}, giving the result.
\end{proof}

\begin{remark} \label{r:green}
It follows from Lemma~\ref{l:knownseeds}, parts (a) and (c),
that there
is an isomorphism from $Q_{\top}$ to $Q_{\perp}$ taking $(i,j)$ to $(k-i,j)$
for $1\leq i\leq k-1$ and $1\leq j\leq n-k-1$. Since $(Q_{\perp},\pi(\mathbf{x}_{\perp}))$ and $(Q_{\top},\pi(\mathbf{x_{\top}}))$
are both seeds of $\mathbb{C}[\Gr_{k,n}]_1$
(by Lemma~\ref{l:knownseeds}(b) and Remark~\ref{r:specialize})
there is an automorphism ${\gamma}$ of $\mathbb{C}[\Gr_{k,n}]_1$
sending $\pi([M_{k,n}(i,j)])$ to $\pi([\widetilde{M}_{k,n}(k-i,j)])$ for all $i,j$.
\end{remark}

\begin{corollary} \label{c:greenidentity}
The following identity holds:
$$\gamma^2\pi=\pi\sigma^{-k}.$$
\end{corollary}
\begin{proof}
By Lemma~\ref{l:tcoincide}, $\sigma^{-k}\twistbracket{[M_{k,n}(i,j)]}$
is equal to a monomial in the coefficients multiplied by
$[\widetilde{M}_{k,n}(k-i,j)]$. So, using Remark~\ref{r:green},
\begin{align*}
\pi \sigma^{-k} \twistbracket{[M_{k,n}(i,j)]} &=
\pi [\widetilde{M}_{k,n}(k-i,j)] \\
&= \gamma\pi [M_{k,n}(i,j)].
\end{align*}
Hence, for any $x\in \mathbb{C}[\Gr_{k,n}]$, we have
$\pi \sigma^{-k}\twist{x}=\gamma\pi(x)$.
By Corollary~\ref{c:periodic}, $\pi \dtwistbracket{x}=\pi\sigma^k(x)$
for any $x\in \mathbb{C}[\Gr_{k,n}]$. So, using Lemma~\ref{l:rotatesigns}, we have:
\begin{align*}
\gamma^2\pi(x) &=
\gamma \pi \sigma^{-k} \twistbracket{x} \\
&= \pi \sigma^{-k} \twistbracket{\sigma^{-k}\twistbracket{x}} \\
&= \pi \dtwistbracket{\sigma^{-2k}(x)} \\
&= \pi \sigma^{-k}(x),
\end{align*}
giving the result.
\end{proof}

Let $\mu_{\boldsymbol{\alpha}}$ be the sequence of mutations coming from
the sequence $\boldsymbol{\alpha}$ (see Definition~\ref{d:alpha}).
We have the following result.

\begin{prop} \label{p:mutation}
\begin{itemize}
\item[(a)] 
Up to coefficients,
$\sigma^k(\mathbf{x}_{\top})=\twist{\mathbf{x}(\R_{k,n})}$.
\item[(b)] $\mu_{\boldsymbol{\alpha}}(\xx(\R_{k,n}),\widetilde{Q}(\R_{k,n}))=(\xx_{\top},\widetilde{Q}_{\top})$.
\item[(c)] The sequence $\mu_{\boldsymbol{\alpha}}$ of mutations, when applied to $(\xx(\R_{k,n}),\widetilde{Q}(\R_{k,n}))$,
passes only through seeds corresponding to Postnikov diagrams.
\end{itemize}
\end{prop}
\begin{proof}
Part (a) follows from Lemmas~\ref{l:tcoincide} and~\ref{l:rkntwist}, noting
Lemma~\ref{l:rotatesigns}.
Part (b) follows from Lemma~\ref{l:mutationstep}. For part (c),
we note that, for each mutation in $\mu_{\boldsymbol{\alpha}}$, there are
two arrows pointing towards the mutation vertex and two arrows
pointing away, by Lemma~\ref{l:local}.
It follows that the mutation corresponds to
a quadrilateral move in the Postnikov diagram (see Figure~\ref{fig:quadrilateralmove}).
\end{proof}

It remains to check that $\boldsymbol{\alpha}$ is a maximal green
sequence. We first recall the definition.

Let $Q$ be a quiver with no loops or $2$-cycles, with vertices
$Q_0=\{1,\ldots ,l\}$ and arrows $Q_1$. The \emph{framing}
$\widehat{Q}$ of $Q$ is the quiver with vertices
$\widehat{Q}_0=\{1,\ldots ,2l\}$. Writing $r'=r+l$ for
$r=1,\ldots ,n$, the arrows are
$$\widehat{Q}_1=Q_1\sqcup \{r\rightarrow r'\,:\,r\in Q_0\}.$$
Then $\widehat{Q}$ is an exchange quiver (see the start of Section~\ref{s:clusterstructure}).
Let $\Mut(\widehat{Q})$ denote the mutation class of $\widehat{Q}$, i.e. the set of quivers which can be obtained from
$\widehat{Q}$ by iterated mutation. Let $R\in \Mut(\widehat{Q})$. Then a vertex $r\in Q_0$ (thus non-frozen) is said to be \emph{green} (respectively, \emph{red}) if every arrow between $r$ and a vertex $s'$, $s\in Q_0$, points towards (respectively,
away from) $s'$.

A sequence $r_1,\ldots ,r_m$ of vertices of $Q_0$ is said
to be a \emph{green} sequence for $Q$ if $r_j$ is green
in $\mu_{r_{j-1}}\cdots \mu_{r_1}(\widehat{Q})$ for
$j=1,\ldots ,m$.
If every non-frozen vertex in $\mu_{r_m}\cdots \mu_{r_1}(\widehat{Q})$ is red then $r_1,\ldots ,r_m$ is said to be a
\emph{maximal green sequence for $Q$}.

We recall the following result:

\begin{theorem} \cite[Thm.\ 4.4, \S9.1]{BDP14} \label{t:dynkingreen}
Let $Q$ be a Dynkin quiver with vertices $Q_0$. Let $r_1,\ldots ,r_N$ be a sequence of vertices of $Q_0$ such that:
\begin{itemize}
\item[(a)] $r_1,\ldots ,r_N$ is an admissible sequence of
sinks in $Q$ (i.e. $r_j$ is a sink in $\mu_{r_{j-1}}\cdots \mu_{r_1}(Q)$ for $j=1,\ldots ,N$), and
\item[(b)] $s_{r_1}\cdots s_{r_N}$ is a reduced expression
for the longest element $w_0$ of the Weyl group $W$ of
$Q$, where $s_r$ is the simple reflection in $W$
corresponding to $r\in Q_0$.
\end{itemize}
Then $r_1,\ldots ,r_N$ is a maximal green sequence for $Q$.
\end{theorem}

For $1\leq j\leq n-k-1$, let
$$V_j=\{(i,j)\,:\,1\leq i\leq k-1\}\subseteq R_{\mathbb{Z}}.$$
Note that we can regard $V_j$ as a subset of the
set of vertices of $Q_{\omega}$ for any $\omega\in \Omega$.
Let $\widehat{Q}_{\perp}$ be the framing of $Q_{\perp}$,
with arrows $(i,j)\rightarrow (i,j)'$ for all $(i,j)\in R_{\mathbb{Z}}$.

For any $\omega\in \Omega$, let $\boldsymbol{\alpha}(\omega)$ be the 
sequence of elements of $R_{\mathbb{Z}}$ obtained by 
applying $\text{pr}$ to the elements of
$\{\omega'\in \Omega\,:\,\omega'<\omega\}$ written in the
order $<$.
Thus $\boldsymbol{\alpha}(\top)=\boldsymbol{\alpha}$ (see Definition~\ref{d:alpha}).
By Lemma~\ref{l:mutationstep}, we have
$\mu_{\boldsymbol{\alpha}(\omega)} Q_{\perp}=Q_{\omega}$.
We define $\Gamma_{\omega}=\mu_{\boldsymbol{\alpha}(\omega)}\widehat{Q}_{\perp}$.
In particular, $\G_{\perp}=\widehat{Q}_{\perp}$.

For $1\leq j\leq n-k-1$, let $\boldsymbol{\alpha}^j(\omega)$ be the
subsequence of $\boldsymbol{\alpha}(\omega)$ consisting of only those entries which are vertices in $V_j$.
For $1\leq j\leq n-k-1$, let $\G^j_{\omega}$ be the induced
subquiver of $\G_{\omega}$ on the vertices $V_j\cup V'_j$ and let
$Q^j_{\omega}$ be the induced subquiver of $Q_{\omega}$ on the vertices
$V_j$. Note that $\G^j_{\perp}=\widehat{Q}^j_{\perp}$, the
framing of $Q^j_{\perp}$.

\begin{lemma} \label{l:subsequencegreen}
Fix $1\leq j\leq n-k-1$. Then $\boldsymbol{\alpha}^j$ is a maximal green
sequence for $Q^j_{\perp}$.
\end{lemma}
\begin{proof}
The sequence $\boldsymbol{\alpha}^j$ is an admissible sequence of sinks for $Q^j_{\perp}$ giving a reduced expression
for $w_0$, so this follows from Theorem~\ref{t:dynkingreen}.
\end{proof}

We can now prove the following:

\begin{lemma} \label{l:greenmutation}
Let $\omega\in \Omega$. Then, for any $1\leq j\leq n-k-1$,
we have:
\begin{itemize}
\item[(a)] $$\G^j_{\omega}=\mu_{\boldsymbol{\alpha}^j(\omega)} \widehat{Q}^j_{\perp}$$
\item[(b)] Let $a$ be an arrow in $\G_{\omega}$
with an endpoint in $V'_j$. Then the other endpoint
of $a$ lies in $V_j$.
\item[(c)] If $\omega=(x,y;l)$ then the vertex $(x,y)$ in $\G_{\omega}$ is green.
\end{itemize}
\end{lemma}
\begin{proof}
We prove this by induction on the the length of
$\boldsymbol{\alpha}(\omega)$.
The result clearly holds when $\omega=\perp$, since
$\G_{\perp}=\widehat{Q}_{\perp}$. Suppose that the
result holds for
$\omega=(x,y;l)\in\Omega\setminus \{\top\}$.
We will show that it holds for $\omega^+$.

Note that $\G_{\omega^+}=\mu_{(x,y)}\G_{\omega}$.
By Lemma~\ref{l:local}, we have that

\noindent (i) An arrow in $\G_{\omega}$ between $(x,y)$ and a non-frozen vertex of $\G_{\omega}$ in $V_j$ for some $j\not=y$ must point away from $(x,y)$.

Hence:

\noindent (ii)
If $j\not=y$, there is no path of length $2$ in $\G_{\omega}$ 
starting at a vertex of $V_j$ and ending at a vertex of
$V_j$ and passing through $(x,y)$, i.e. no path of the
form:
$$\xymatrix{
(i,j) \ar[r] & (x,y) \ar[r] & (i',j) \\
}$$
where $1\leq i,i'\leq k-1$.

By (c) for $\omega$, we have:

\noindent (iii)
Any arrow in $\G_{\omega}$ between $(x,y)$ and a vertex of $V'_y$ must point away from $(x,y)$.

By (b) for $\omega$, there can be no arrow in $\G_{\omega}$ between $(x,y)$ and a vertex of $V'_j$ with $j\not=y$. So:

\noindent (iv)
  For $j\not=y$, there can be no path in $\G_{\omega}$ of length $2$ (in either direction) between a vertex in $V'_j$ and a vertex in $V_l$, for any $l$, passing through $(x,y)$, i.e. no path of the form:
$$\xymatrix{
(i_1,j)' \ar[r] & (x,y) \ar[r] & (i_2,l)
}
\quad\quad\quad
\xymatrix{
(i_2,l) \ar[r] & (x,y) \ar[r] & (i_1,j)'
}
$$
where $1\leq i_1,i_2 \leq k-1$ and $1\leq l\leq n-k-1$.

It follows from (ii) and (iv) (with $l=j$) that:

\noindent (v)
For $j\not=y$,
applying $\mu_{(x,y)}$ to $\G_{\omega}$ has no effect on the induced subquiver of $\G_{\omega}$ on the vertices
$V_j\cup V'_j$. In other words, the induced subquiver
of $\mu_{(x,y)}\G_{\omega}$ on $V_j\cup V'_j$ coincides
with $\mu_{(x,y)}\G^j_{\omega}$.

It follows from (iv) that applying $\mu_{(x,y)}$ to
$\G_{\omega}$ does not introduce or delete any arrows between a vertex in $V'_j$, $j\not=y$, and $V_l$ for any $l$.
By (i) and (iii), there are no paths of length $2$ in $\G_{\omega}$ (passing through $(x,y)$) between a vertex of $V'_y$ and a vertex of $V_l$, $l\not=y$,
i.e. no paths of the form:
$$\xymatrix{
(i,y)' \ar[r] & (x,y) \ar[r] & (i',l)
}
\quad\quad\quad
\xymatrix{
(i',l) \ar[r] & (x,y) \ar[r] & (i,y)'
}$$
where $1\leq i,i'\leq k-1$ and $1\leq l\leq n-k-1$.
Hence, applying $\mu_{(x,y)}$ to $\G_{\omega}$
does not introduce any arrows between vertices in $V'_y$
and vertices of $V_l$, $l\not=y$.
We have shown that $(b)$ holds for $\omega^+$.

Note that if $Q'$ is a full subquiver of a quiver $Q$, then
mutating $Q$ at a vertex $v$ of $Q'$ has the same effect
on $Q'$ as mutating $Q'$ at $v$. Hence
\begin{align*}
\G_{\omega^+}^y &= \mu_{(x,y)}\G_{\omega}^y & \text{ (since 
$\G_{\omega^+}=\mu_{(x,y)}\G_{\omega}$)} \\
&= \mu_{(x,y)}\mu_{\boldsymbol{\alpha}^y}(\omega)\widehat{Q}_{\perp}^j &\text{ (by part (a) for $\omega$)} \\ 
&=\mu_{\boldsymbol{\alpha}^y}(\omega^+) \widehat{Q}_{\perp}^j,
\end{align*}
giving part (a) for $\omega^+$ in the case $j=y$.

If $j\not=y$ then we have
\begin{align*}
\G_{\omega^+}^j &= \G_{\omega}^j \text{ (by (v))} \\
&= \mu_{\boldsymbol{\alpha}^j(\omega)}\widehat{Q}^j_{\perp}
\text{ (by part (a) for $\omega$)} \\
&=\mu_{\boldsymbol{\alpha}^j(\omega^+)}\widehat{Q}^j_{\perp} \text{ (as $j\not=y$).}
\end{align*}
It follows that part (a) for $\omega^+$ holds for $j\not=y$ and hence for all $j$.
If $\omega^+\not=\top$ then part (c) for $\omega^+$ follows from part (a) for
$\omega^+$ (with $j=y$), part (b) for $\omega^+$ and
the fact that $\boldsymbol{\alpha}^j$ is a maximal green sequence
for $Q^j_{\perp}$ (Lemma~\ref{l:subsequencegreen}).

By induction, the result holds for all $\omega\in \Omega$.
\end{proof}

\begin{prop} \label{p:isgreen}
The sequence $\boldsymbol{\alpha}$ is a maximal green sequence for
$Q(\R_{k,n})=Q_{\perp}$.
\end{prop}
\begin{proof}
By Lemma~\ref{l:greenmutation}(c), the vertex $(x,y)$ in $\G_{\omega}$ is
green for all $\omega\in \Omega\setminus \{\top\}$, from
which it follows that $\boldsymbol{\alpha}$ is a green sequence for
$Q_{\perp}$. 

Suppose that $1\leq j\leq n-k-1$. Then, by
Lemma~\ref{l:subsequencegreen}, every non-frozen
vertex in $\mu_{\boldsymbol{\alpha}^j}\widehat{Q}^j_{\perp}$
is red. By Lemma~\ref{l:greenmutation}(a) with $\omega=\top$,
$\G^j_{\top}=\mu_{\boldsymbol{\alpha}^j}\widehat{Q}^j_{\perp}$.
Hence, by Lemma~\ref{l:greenmutation}(b) with $\omega=\top$,
every non-frozen vertex in $\G_{\top}$ is red. It follows
that $\boldsymbol{\alpha}$ is a maximal green sequence for
$Q(\R_{k,n})=Q_{\perp}$.
\end{proof}

We can now sum up our results in this section as follows:

\begin{theorem} \label{t:greensummary}
Suppose that $k\not=1,n-1$. Then there is a maximal green sequence, passing only through seeds corresponding to
Postnikov diagrams, taking the seed $(\xx(\R_{k,n}),\widetilde{Q}(\R_{k,n}))$ for
$\mathbb{C}[\Gr_{k,n}]$ to a seed $(\xx_{\top},\widetilde{Q}_{\top})$ whose
quiver is isomorphic to $\widetilde{Q}(\R_{n-k,n}^*)$. The image of the cluster in $\mathbf{x}_{\top}$ under $\sigma^k$ coincides, up to coefficients,
with $\twist{\mathbf{x}(\R_{k,n})}$.
\end{theorem}

\begin{proof}
Note that we may ignore the cases $k=1$ or $n-1$, since then all cluster variables are coefficients. If $k\not=2,n-2$ then the result follows from Lemma~\ref{l:knownseeds},
Propositions~\ref{p:mutation} and~\ref{p:isgreen}.
So, suppose that $k=2$ (the case $k=n-2$ is similar).
Then the seed
$(\xx(\R_{2,n}),\widetilde{Q}(\R_{2,n}))$ is
as shown on the left hand side of Figure~\ref{f:case2}, 
where we have labelled each vertex by the corresponding cluster variable or coefficient
in $\xx(\R_{2,n})$. We see that:
$$\mathbf{x}(\R_{2,n})=\{[1j]\,:\,3\leq j\leq n-1\}.$$
By Remark~\ref{r:twistk2},
we have that $\twist{[1,j]}=[j-1,n]$ for $j=3,\ldots ,n-1$.
It is easy to check directly that mutating at the
vertices initially labelled with $[1,3],\ldots ,[1,n-1]$, in this ordering, transforms
the seed $(\xx(\R_{2,n}),\widetilde{Q}(\R_{2,n}))$
to the seed shown on the right hand side of
Figure~\ref{f:case2}. It is also easy to check that the
quiver in this seed is isomorphic to $\widetilde{Q}(\R_{n-2,n}^*)$. The cluster in this seed is
$$\mathbf{x}=\{[2,j+1]\,:\,3\leq j\leq n-1\}.$$
Since $\sigma^2([2,j+1])=[j-1,n]$ for $3\leq j\leq n-1$,
we see that $\sigma^2(\mathbf{x})=\twist{\mathbf{x}(\R_{2,n})}$.
It is clear that the mutation sequence above passes only
through Postnikov diagrams (indeed, in this case, every seed arises from a Postnikov diagram~\cite[Cor. 2]{scott06}).
Furthermore, the sequence above is a maximal
green sequence for $Q(\R_{2,n})$ by~\cite[Lemma 2.20]{BDP14}.
\end{proof}

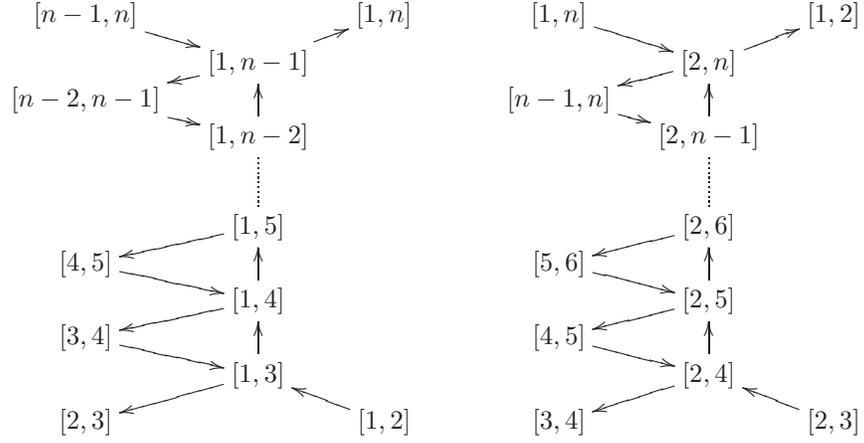
\begin{figure}
$$
\vcenter{
\xymatrix@R=-2pt@C=3pt{
[n-1,n] \ar[ddrr] && && [1,n] \\ \\
&& [1,n-1] \ar[uurr] \ar[dll] \\
[n-2,n-1] \ar[drr] \\
&& [1,n-2] \ar[uu] \\ \\ \\ \\ \\ \\
&& [1,5] \ar@{.}[uuuuuu] \ar[dll] \\
[4,5] \ar[rrd] \\
&& [1,4] \ar[uu] \ar[dll] \\
[3,4] \ar[drr] \\
&& [1,3] \ar[uu] \ar[ddll] \\ \\
[2,3] && && [1,2] \ar[uull]
}}
\quad\quad\quad
\vcenter{
\xymatrix@R=-2pt@C=3pt{
[1,n] \ar[ddrr] && && [1,2] \\ \\
&& [2,n] \ar[uurr] \ar[dll] \\
[n-1,n] \ar[drr] \\
&& [2,n-1] \ar[uu] \\ \\ \\ \\ \\ \\
&& [2,6] \ar@{.}[uuuuuu] \ar[dll] \\
[5,6] \ar[rrd] \\
&& [2,5] \ar[uu] \ar[dll] \\
[4,5] \ar[drr] \\
&& [2,4] \ar[uu] \ar[ddll] \\ \\
[3,4] && && [2,3] \ar[uull]
}}
$$
\caption{The seed $(\widetilde{x}(\R_{2,n}),\widetilde{Q}(\R_{2,n}))$ (left) and
its image after mutating at the maximal green sequence in the part of the
proof of Theorem~\ref{t:greensummary} for $k=2$ (right).}
\label{f:case2}
\end{figure}

\section{Surfaces and future directions}
\label{s:surfaces}
The approach developed here can be generalised in a straightforward way to a
\emph{surface graph} --- a bipartite graph $G$ equipped with an embedding into a surface $\Sigma$ with (or without) boundary $\partial \Sigma$, in such a way that no two edges of $G$
cross within $\Sigma$ and each face of $G$ is homeomorphic to a disk;
by definition faces are the connected components of the complement of
$G$ in $\Sigma$. An example of a surface graph is shown in Figure~\ref{fig:bipartite1}.

\begin{figure}[h]
\begin{center}
\includegraphics[width=2.5in]{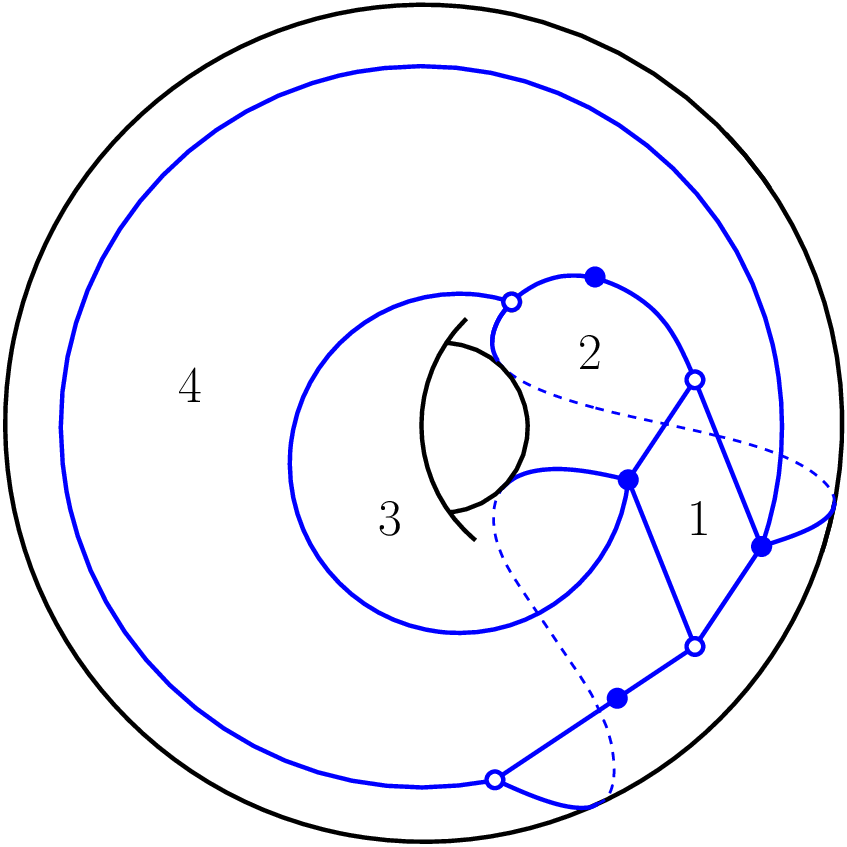}
\end{center}
\caption{Bipartite graph $G$ embedded on a torus $\Sigma$ (drawn in black).
In this case $\partial \Sigma$ is empty and $n = k = 0$ and there are $4$ faces.}
\label{fig:bipartite1}
\end{figure}

By choosing a transcendence basis --- whose elements label the faces of $G$ --- we obtain a seed (and a cluster algebra) whose quiver $Q$ is the {\it face dual} graph of $G$. The vertices of $Q$ correspond to faces of $G$. For each common edge separating a pair of faces $E$ and $F$, an arrow is drawn from the vertex corresponding to face $E$ to the vertex corresponding to face $F$, in such a way that the white vertex lies to the left when crossing the edge from $E$ to $F$; as usual oriented $2$-cycles are annihilated afterwards. See Figure~\ref{fig:bipartite2} for an example.

\begin{figure}[h]
\begin{center}
\includegraphics[width=4in]{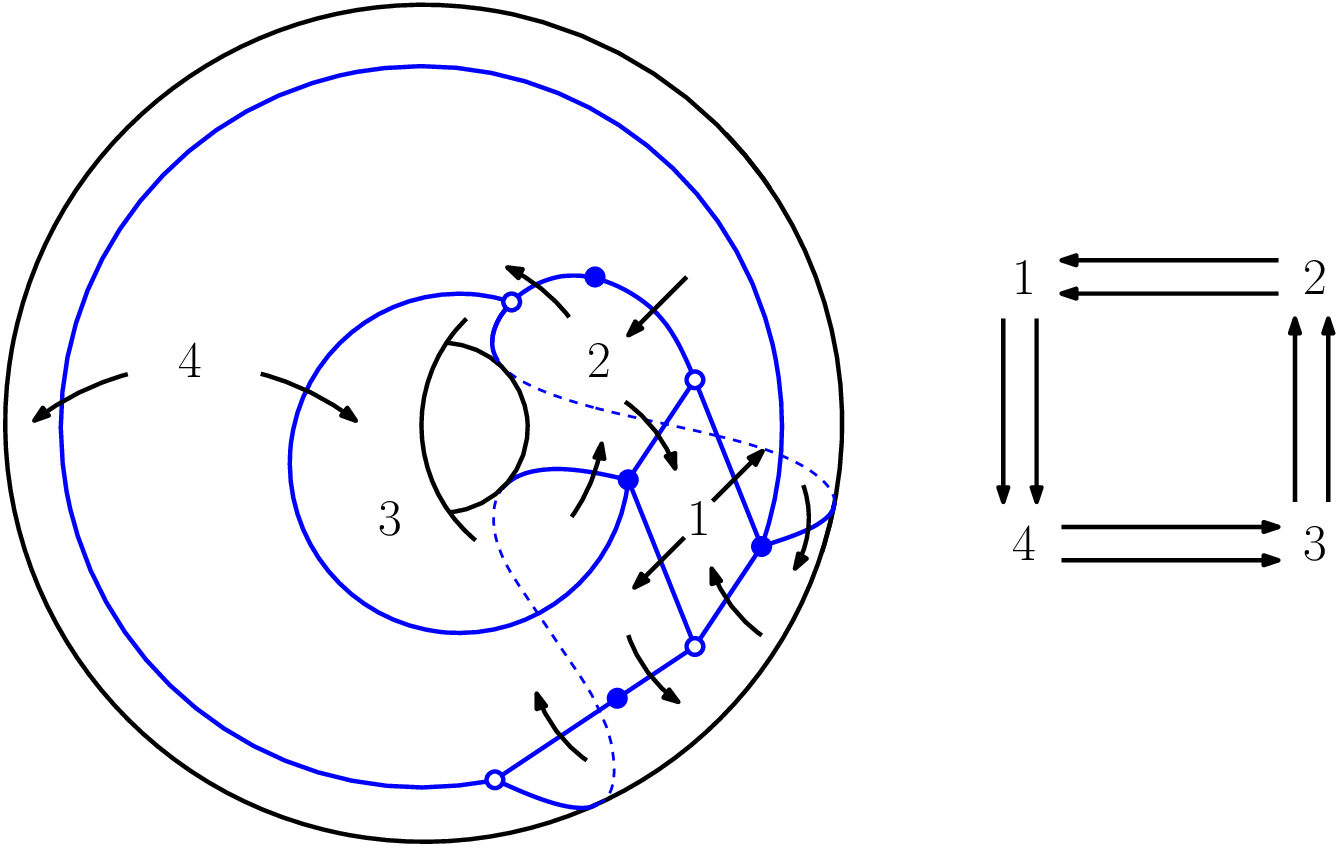}
\end{center}
\caption{The quiver $Q$ of an embedded bipartite graph $G \hookrightarrow \Sigma$: initial stage of the quiver construction (superimposed in black on the left) and then redrawn (on the right) after removing all oriented two cycles.}
\label{fig:bipartite2}
\end{figure}

The concepts explained in Sections~\ref{s:clusterstructure}
and ~\ref{s:dimer} --- the definition of edge weights, the blow-up and
blow-down equivalences, quadrilateral mutation ---
are all local notions, i.e.\ defined and/or constructed with a 
neighbourhood of the participating edges and vertices, and for this 
reason can be used unambiguously in the surface case. 
(See Figure~\ref{fig:bipartite3} for an example of a quadrilateral move 
performed on a torus.)
In addition, the weight of a dimer configuration (as in 
Lemma~\ref{l:blowdown}) will remain invariant under blow-ups and 
blow-downs, and therefore so will the dimer partition function (as in 
Corollary~\ref{c:blowdown}). Similarly, Proposition~\ref{prop:flipmoveinvariance}
generalizes to the surface setting,  ensuring that the
$\uu_{G}(I)$ are invariant under quadrilateral moves. 
In the remainder of this section we formulate some
questions concerning the general surface case.

\begin{figure}[h]
\begin{center}
\includegraphics[width=4in]{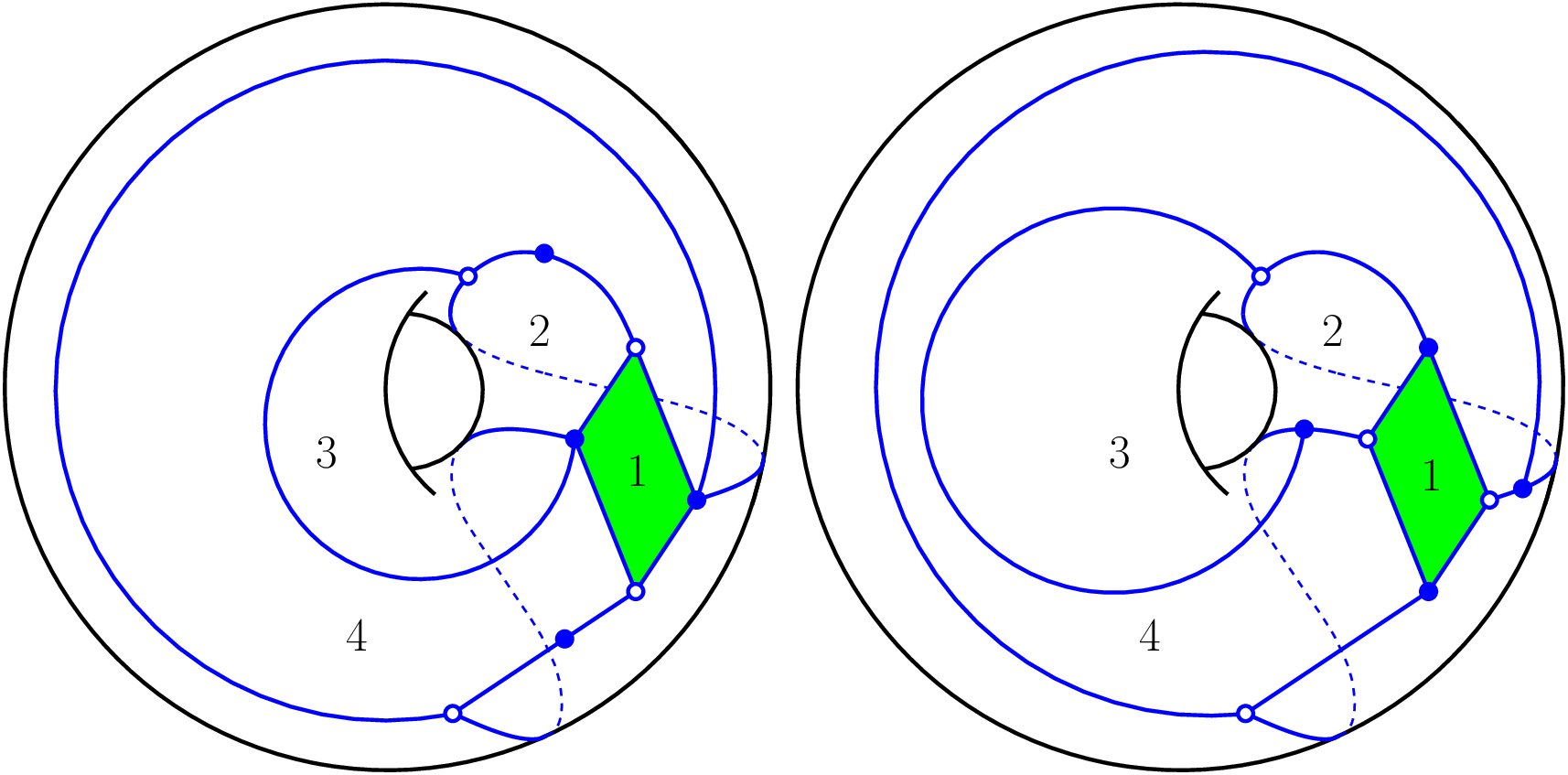}
\end{center}
\caption{Mutation of a embedded bipartite graph $G  \hookrightarrow \Sigma$
associated to a quadrilateral face (shaded in green)}
\label{fig:bipartite3}
\end{figure}
\bigskip
\indent $\bullet$
Can the dimer partition functions for specific boundary conditions be on $G\cap \partial \Sigma$ be identified with specific cluster variables?  If so, which cluster variables are these? Can this be done with reference to the faces of
$G$?
A satisfactory answer to this
question ought to explain how to define and express an analogue of the BFZ-\emph{twist}~\cite{BFZ96, BZ97}
for the cluster algebra associated to $G$ in local coordinates.
On a combinatorial level this question is related to determining those boundary
conditions which admit precisely one \emph{dimer configuration}
for $G$.

\bigskip
\indent $\bullet$ How to view \emph{FZ}-\emph{mutation} in this context? The simplest type of mutation, the so-called quadrilateral move, corresponds to a special kind of local rotation in $G$ which fixes $\Sigma$ and $G \cap \partial \Sigma$ and which conserves the dimer partition functions. Higher
order mutation, however, not only changes $G$ but also $\Sigma$; for instance
hexagonal mutation increases the genus by adding a handle. Since the topology
changes some care is needed in analysing how the dimer partition functions are transformed.

\bigskip
\indent $\bullet$ 
In cases where such a cluster algebra arising from a surface graph can be realised as (or is related to) the coordinate ring $\mathbb{C}[V]$ of some reasonable quasi-projective algebraic variety $V$ (e.g.~\cite{Goncharov}), or in the phase spaces considered in~\cite{GK13}, what is the geometric meaning of the BFZ-twist analogue?

\vskip 0.2cm

\textbf{Acknowledgements:}
Both authors would like to thank the referees for their
comments and suggestions on an earlier version of this mansucript,
and also David Speyer for his helpful comments.

Jeanne would like to thank Professors Claus 
Michael Ringel and Henning Krause
for the opportunity to visit the SFB and 
algebra group in Bielefeld: it has been a 
very productive time. She would also like to 
thank Robert Marsh for numerous fruitful 
discussions during her time at Leeds where 
the principal ideas for this work first 
germinated, and for coming to Bielefeld and 
energizing the common project (and her too).

Robert would like to thank Jeanne Scott, 
Henning Krause and the group at the SFB 701 
in Bielefeld for making him very welcome 
during a visit in September 2013 when part 
of the work for this article was carried 
out. He would also like to thank Bernard
Leclerc and Konstanze Rietsch for
some helpful conversations.

\bibliographystyle{plain}

\end{document}